\documentclass[11pt]{article}
\usepackage{amsmath}
\usepackage{amsthm}
\usepackage{amssymb}
\usepackage{amsfonts}
\usepackage{latexsym}
\usepackage{graphicx}
\usepackage{subcaption}
\usepackage{caption}
\usepackage[pdftex,bookmarks=true,bookmarksnumbered=true]{hyperref}
\usepackage{color}
\usepackage{array}
\usepackage{booktabs} 
\renewcommand{\epsilon}{\varepsilon}
\renewcommand{\rho}{\varrho}
\renewcommand{\phi}{\varphi}
\newcommand{\DS}{\displaystyle}
\newcommand{\N}{{\mathbb N}}
\newcommand{\Z}{{\mathbb Z}}

\newcommand{\R}{{\mathbb R}}

\newcommand{\cF}{{\cal F}}
\newcommand{\cG}{{\cal G}}

\newcommand{\cL}{{\cal L}}

\newcommand{\cX}{{\cal X}}
\newcommand{\cY}{{\cal Y}}
\newcommand{\bg}{\bar{\gamma}}
\newcommand{\bu}{\bar{u}}
\newcommand{\bw}{\bar{\omega}}
\newcommand{\bx}{\bar{x}}
\newcommand{\bB}{\bar{B}}

\newcommand{\blambda}{\bar{\lambda}}
\newcommand{\bphi}{\bar{\phi}}
\newcommand{\bpsi}{\bar{\psi}}
\newcommand{\bmu}{\bar{\mu}}
\newcommand{\bcG}{\bar{\cG}}

\newcommand{\rstar}{r^\star}
\renewcommand{\Box}{\mathrm{Box}}
\newcommand{\finite}{\mathrm{finite}}
\newcommand{\tail}{\mathrm{tail}}
\newcommand{\nsym}{n^{\mathrm{sym}}}
\renewcommand{\d}{\mathrm{d}}
\newcommand{\sym}{\mathrm{sym}}
\newcommand{\asym}{\mathrm{asym}}

\newcommand{\lambdabif}{\lambda_\circ}
\newcommand{\ubif}{u_\circ}
\newcommand{\phibif}{\phi_\circ}
\newcommand{\psibif}{\psi_\circ}
\newcommand{\xibif}{\xi_\circ}
\newcommand{\zetabif}{\zeta_\circ}
\newcommand{\rhobif}{\rho_\circ}

\begin{document}
\newtheorem{definition}{Definition}[section]
\newtheorem{theorem}[definition]{Theorem}
\newtheorem{proposition}[definition]{Proposition}
\newtheorem{corollary}[definition]{Corollary}
\newtheorem{lemma}[definition]{Lemma}
\newtheorem{hypothesis}[definition]{Hypothesis}
\newtheorem{remark}[definition]{Remark}
\newtheorem{conjecture}[definition]{Conjecture}
\newtheorem{problem}[definition]{Problem}
\newtheorem{assumption}[definition]{Assumption}

\title{Validation of Symmetry Induced Bifurcation Points
       in Nonlinear Diffusion Problems}
\author{ Maxime Breden\thanks{
   CMAP, CNRS, \'Ecole polytechnique, Institut Polytechnique de
Paris, 91120 Palaiseau, France.
   ({\tt maxime.breden@polytechnique.edu}).}
\and Evelyn Sander\thanks{
   Department of Mathematical Sciences,
   George Mason University,
   Fairfax, Virginia 22030,
   USA
   ({\tt esander@gmu.edu}).}
\and Thomas Wanner\thanks{
   Department of Mathematical Sciences,
   George Mason University,
   Fairfax, Virginia 22030,
   USA
   ({\tt twanner@gmu.edu}).}
}
\date{September 29, 2026}

\maketitle
\begin{abstract} 
Nonlinear parabolic partial differential equations frequently occur
in the modeling of applied phenomena. They often exhibit complicated
dynamics which is organized by their associated equilibrium solutions.
As a function of the involved model parameters, the equilibria can be
detected using a bifurcation analysis, which in turn is centered around
a deeper understanding of bifurcation points. In many applied systems
establishing these organizing centers rigorously is difficult to achieve
using traditional mathematical techniques. In this paper, we propose a
theoretical framework for computer-assisted proofs of symmetry-induced
transcritical and pitchfork bifurcations in nonlinear problems. This
is accomplished by reformulating the existence question via a zero
finding problem for an associated extended nonlinear system. The approach
is general and allows for a wide variety of underlying symmetries. We
demonstrate how these results can be applied in the setting of parabolic
systems and higher-order equations, on both one- and two-dimensional
domains.

\bigskip\noindent
{\bf AMS subject classifications:} 
Primary: 37G40, 37M20, 65G20, 65P30; Secondary: 37B35, 37C81, 65G30,
74G60, 74N15.

\bigskip\noindent
{\bf Keywords:} Bifurcations, nonlinear diffusion,
symmetry-breaking, transcritical and pitchfork bifurcations,
computer-assisted proofs, interval arithmetic, rigorous validation
\end{abstract}
\newpage
\setcounter{tocdepth}{2}
\tableofcontents
%
%
%
\section{Introduction}
\label{sec:intro}
Bifurcation theory has long been used to study the zeros of nonlinear
parameter-dependent operators, see for example~\cite{chow:hale:82a,
kielhoefer:12a} and the references therein. As the underlying parameters
change, new zeros can only occur via bifurcation points, which therefore
act as organizing centers for the global structure of the zero set.
In view of Sard's theorem, in general nonlinear problems only saddle-node
bifurcations can be observed in a perturbation-robust sense, see for
example the description in~\cite{wanner:18a}. The situation is different
if the nonlinear problem under consideration is for example generated
by the study of equilibrium solutions of parabolic partial differential
equations. In this case, perturbations of the differential equation lead
to very specific perturbations of the associated nonlinear zero finding
problem, and they can and often will exhibit transcritical and 
pitchfork bifurcation points which are robust under parameter changes.
Usually, this robustness is a consequence of some symmetry which is preserved
by the differential equation and some of its stationary states, and which
is then broken at these bifurcation points --- hence the name {\em
symmetry-breaking bifurcation\/}.

\begin{figure}[tb]
      \includegraphics[width=\textwidth]{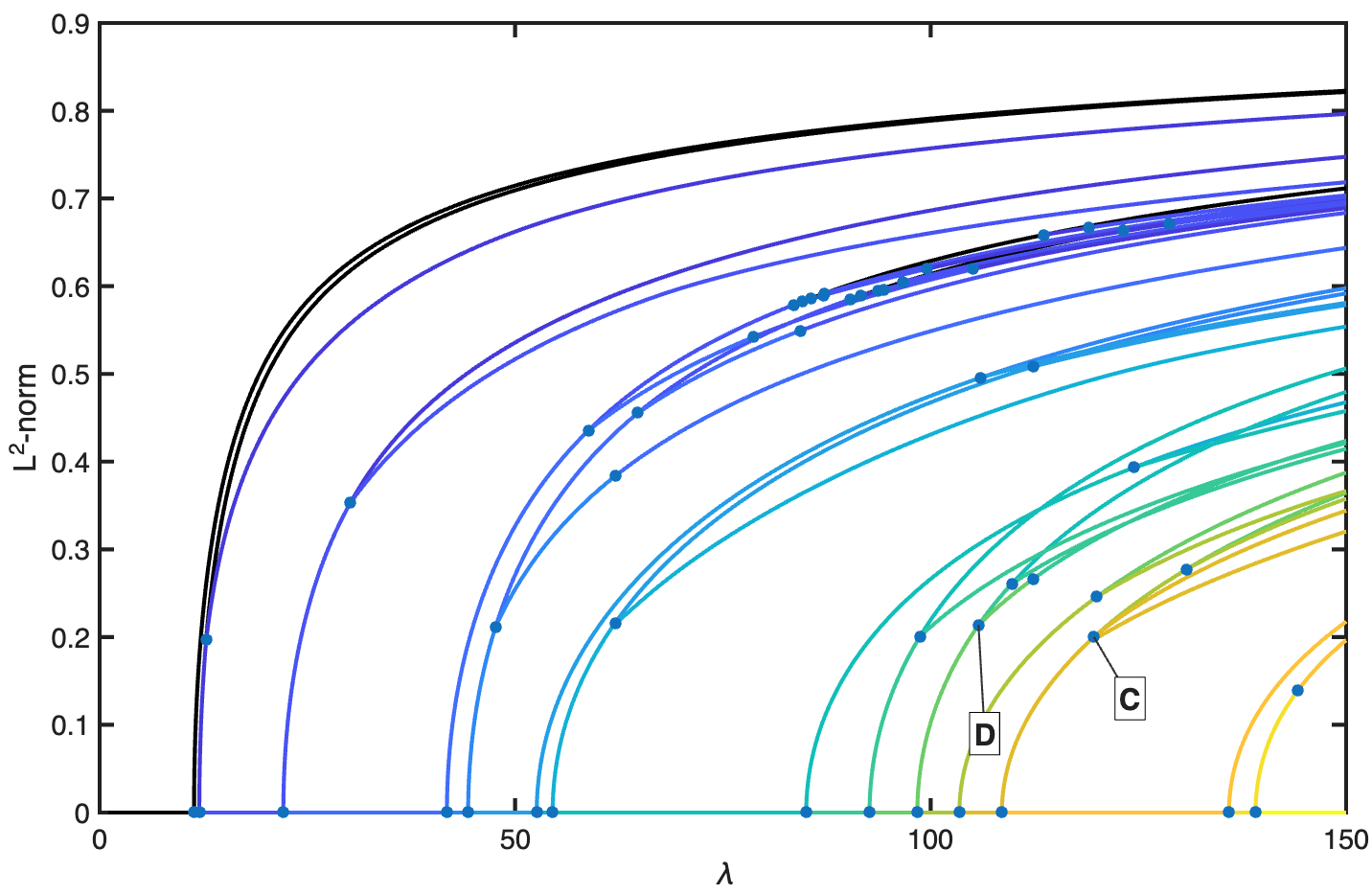}
  \caption{A portion of the bifurcation diagram of steady states for the Ohta--Kawasaki model~\eqref{dbcp} in two dimensions, for total mass $\mu = -0.1$, interaction parameter $\sigma = 1$, and a slightly rectangular domain $\Omega=(0,1)\times(0,0.97)$ to avoid extra symmetries. The pieces of branches are colored by Morse index, and the blue dots represent potential pitchfork and transcritical bifurcations which have been detected using the 
  AUTO software package. The ones labeled C and D are validated in Theorem~\ref{thm:resultsOK}. This figure is by no means a complete picture.}
  \label{fig:diagramOK2D}
\end{figure}

While the general mathematical theory of symmetry-breaking bifurcations
is well-developed~\cite{chossat:lauterbach:00a, golubitsky:etal:88a},
applying these results to specific applied models can be challenging.
Usually, one can only study bifurcation points which are on ``simple''
solution branches directly, such as branches of spatially constant solutions.
As soon as the focus shifts towards secondary bifurcations from nontrivial
solution branches, pen-and-paper existence proofs frequently are no longer
feasible.

In such situations computer-assisted proof techniques
can fill the void. It was shown in~\cite{lessard:sander:wanner:17a} that
both saddle-node and standard $\Z_2$-symmetry breaking pitchfork bifurcations
can be rigorously established in the Ohta--Kawasaki model for the formation
of diblock copolymers. This work was based on the numerical studies
in~\cite{johnson:etal:13a}, and it was accomplished by reformulating both
bifurcation types as zero finding problems for suitable extended
nonlinear systems. These extended systems contain not only the equilibrium
condition, but also force the existence of a zero eigenvalue of the
Fr\'echet derivative at the bifurcation point. While nominally the
extended systems for these two situations look the same, the distinction
between the two types of bifurcation points is achieved via domain restrictions
for the underlying operators. Based on this reformulation, it was possible
in~\cite{lessard:sander:wanner:17a} to apply standard computer-assisted
proof techniques to rigorously verify branches of saddle-node and
$\Z_2$-symmetry breaking bifurcation points in a two-parameter setting.

The results in~\cite{lessard:sander:wanner:17a} led to the validation
of many pitchfork bifurcation points in the Ohta--Kawasaki equilibrium
model on one-dimensional domains. However, numerical 
studies of this model also uncovered numerous potential bifurcation 
points which could not be treated in this way. In these cases, the
underlying symmetries were more complicated in nature, induced by cyclic
groups of higher order. In~\cite{rizzi:sander:wanner:24a} we were able to
extend the previous approach to cover these cases as well, by suitably
changing the extended system. The computer-assisted proofs in that
paper were modeled after the approach described in~\cite{rizzi:etal:22a,
sander:wanner:16a, sander:wanner:21a, wanner:17a, wanner:18b}, but they
relied heavily on the specific form of the underlying cyclic symmetry 
operator. Common to both of the results described so far was the
exclusive focus on saddle-node and pitchfork bifurcations, the
assumption of zero mass in the Ohta--Kawasaki model, and the
restriction to one-dimensional base domains.

In the present paper, we seek to generalize the above-mentioned results
in several significant directions. These can be described as follows:
\begin{itemize}
\item Move beyond the simple $\Z_2$ and cyclic symmetries considered so far
by describing an approach which in principle allows for arbitrary general
symmetries. In particular, shift the emphasis from detailed knowledge of the
involved symmetry operators to a formulation which focuses on the underlying
symmetry subspace only.
\item Derive methods for dealing with transcritical bifurcations,
which have for example been observed numerically in the Ohta--Kawasaki model
for nonzero mass.
\item Develop an approach that allows one to also consider partial differential
equations on higher-dimensional domains as well as to extend the class of equation
types that can be considered.
\end{itemize}
All of the above points are achieved using the results of this paper,
while still employing the appropriate extended systems. In fact,
we can describe in more detail how the extended systems have to be
adjusted based on the underlying symmetries in the system. Moreover, by
restricting attention to a reformulation in terms of a zero finding problem,
one can then establish the existence of the bifurcation point using a variety
of different available computer-assisted proof techniques.

This paper is part of a broader body of works which develop computer-assisted
techniques to study bifurcations, and an approach similar to the one described
here can be used to study cusp bifurcations~\cite{LesPug25}. Also
computer-assisted, but slightly different in spirit, one can rigorously study
bifurcations using desingularization, see for example~\cite{BerLesQue21} for
Hopf bifurcations, or quantitative Lyapunov-Schmidt reduction and normal form
calculations~\cite{AriKoc10,Zgl15}. Going even further, a complete description
of the dynamics near the bifurcation can sometimes be obtained by
computer-assisted means~\cite{KubZglKal25}.
\begin{figure}[tb]
      \includegraphics[width=\textwidth]{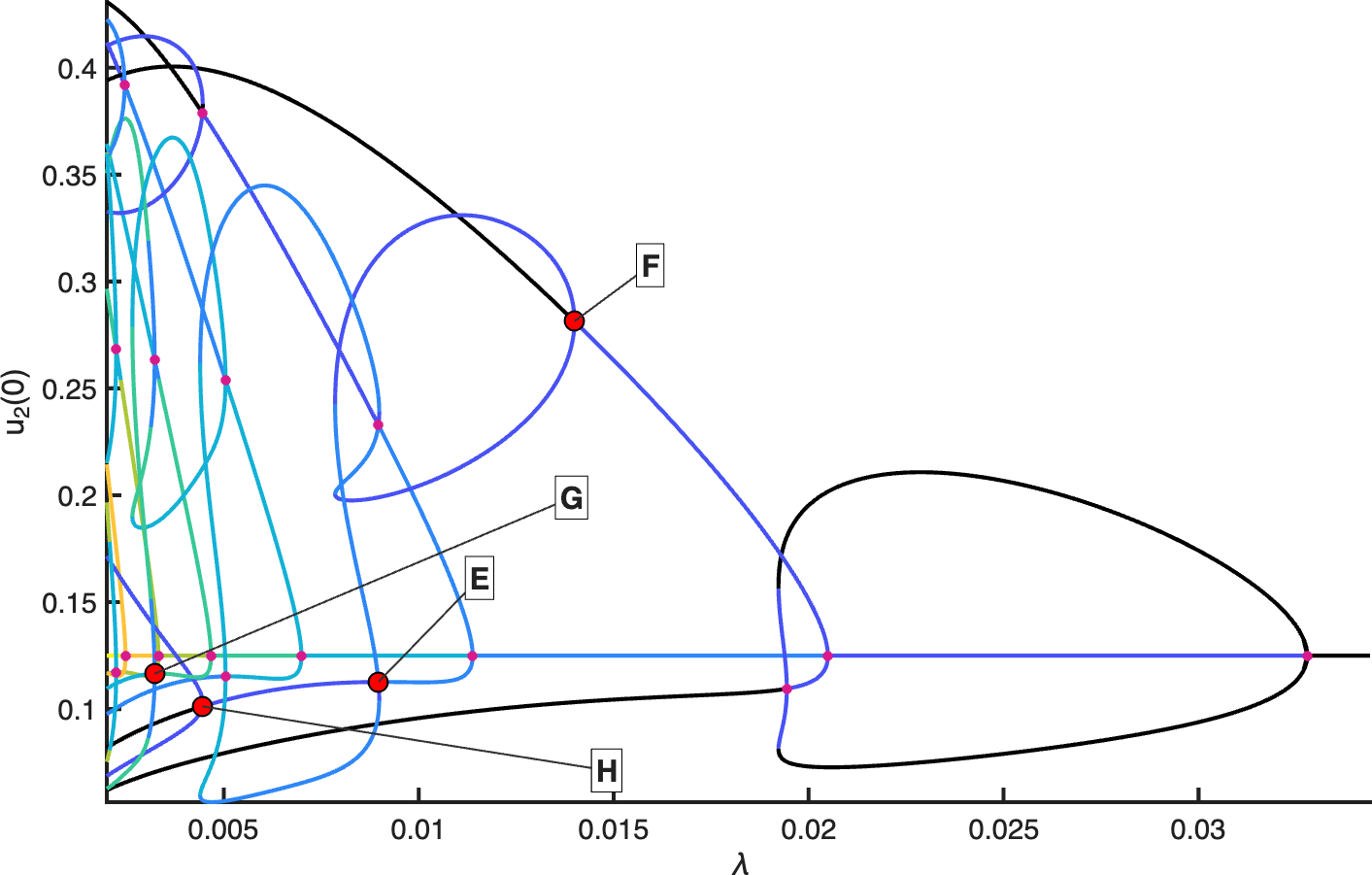}
  \caption{A portion of the bifurcation diagram of steady states for the SKT
      model~\eqref{parabolicsys} on the domain $\Omega=(0,1)$ with parameters~\eqref{crossdiffparams}, first obtained in~\cite{IidMimNim06}. To be consistent with this earlier work, the vertical axis represents the value $u_2(0)$ of the second component at the left boundary. The red dots represent 
  potential pitchfork and transcritical bifurcations which have been detected using the AUTO software package. The ones labeled E to H are validated in Theorem~\ref{thm:resultsSKT}. This figure is by no means a complete picture of all possible bifurcations.}
  \label{fig:diagramSKT1D}
\end{figure}

The main results of this paper are illustrated in the context of two
specific partial differential equations. First, we consider
the previously-mentioned Ohta--Kawasaki equation describing phase separation
in diblock copolymers. This model is given by the parabolic partial
differential equation
\begin{equation} \label{dbcp}
  \partial_t u = -\Delta \left( \Delta u + \lambda \left( u - u^3 \right) \right) -
  \lambda \sigma (u - \mu) \; \mbox{ in } \Omega,
\end{equation}
subject to homogeneous Neumann boundary conditions for both~$u$
and~$\Delta u$. While the parameter~$\lambda$ serves as the main
bifurcation parameter, the model also depends on the total mass~$\mu$
and the polymer interaction parameter~$\sigma$. We will be able to validate
symmetry-breaking bifurcations in the equilibrium structure of this model
for the case of nonzero mass. In addition, we rigorously establish
transcritical bifurcation points which persist under parameter variation
for the first time. Finally, existence proofs for bifurcation points on
two-dimensional rectangular domains are included as well, as shown for
instance in Figure~\ref{fig:OK2D_1}.
\begin{figure}[!h]
\centering
      \includegraphics[width=\textwidth]{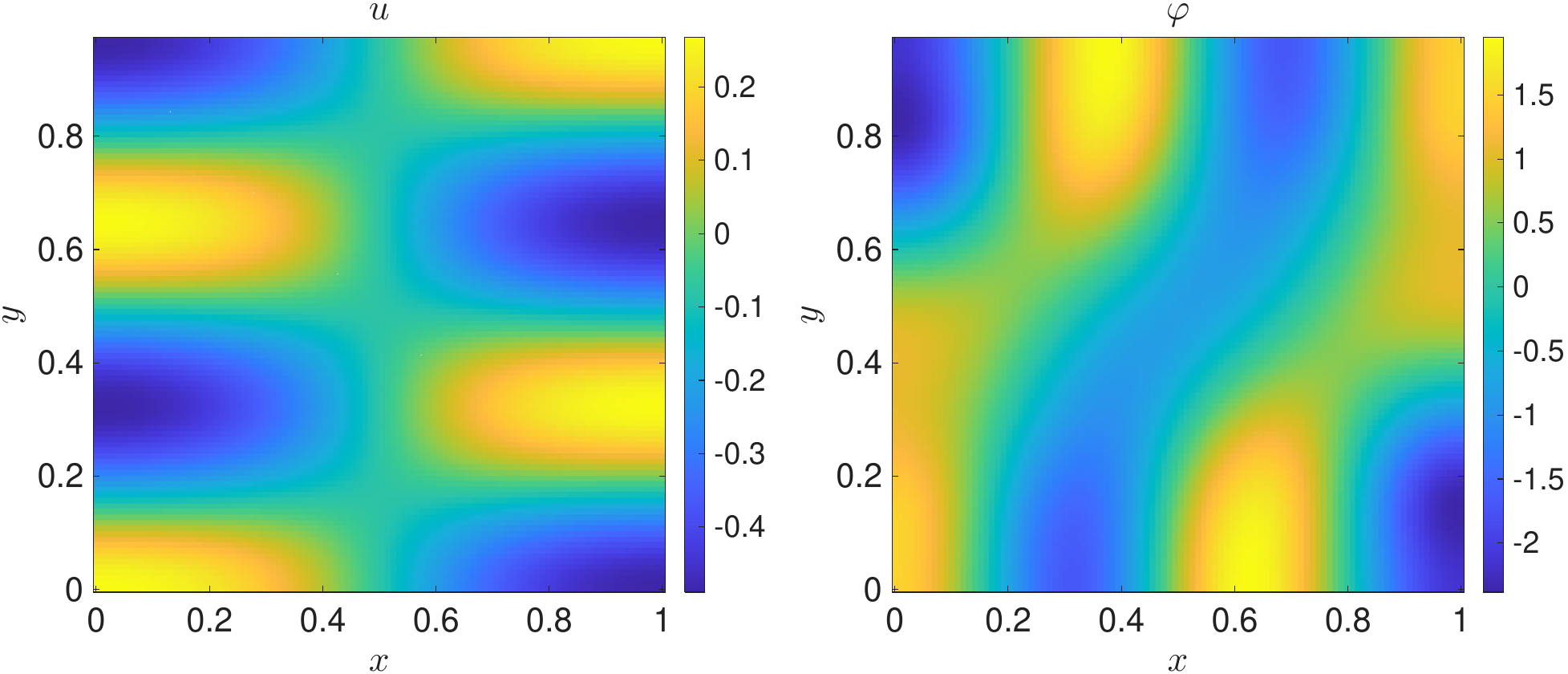}
  \caption{An approximate steady state solution~$u$ of the
  Ohta--Kawasaki equation~\eqref{dbcp} for the parameters
  $\mu=-0.1$, $\sigma=1$, and
  $\Omega=(0,1)\times(0,0.97)$, which undergoes a symmetry-breaking
  transcritical bifurcation at $\lambda\approx 119.66$, with
  approximate eigenfunction $\phi$. This bifurcation is
  labeled~$C$ in Table~\ref{tab:OK}. The occurrence of this
  bifurcation is rigorously established in Theorem~\ref{thm:resultsOK}.}
  \label{fig:OK2D_1}
\end{figure}

The second problem we study in this paper is the  
Shigesada--Kawasaki--Teramoto (SKT) cross diffusion system, a quasi-linear parabolic system introduced in~\cite{ShiKawTer79} to model competing species
\begin{equation} \label{parabolicsys}
  \begin{array}{rcl}
    \DS \partial_t u_1 & = & \DS \Delta \left( d_1 u_1 + d_{11} u_1^2 + d_{12} u_1 u_2 \right)
      + \left( r_1 u_1 - a_1 u_1^2 - b_1 u_1 u_2 \right) \; \\[1ex]
    \DS \partial_t u_2 & = & \DS \Delta \left( d_2 u_2 + d_{21} u_1 u_2 + d_{22} u_2^2 \right)
      + \left( r_2 u_2 - a_2 u_2^2 - b_2 u_1 u_2 \right) \;
  \end{array}
   \; \mbox{ in } \Omega,
\end{equation}
subject to homogeneous Neumann boundary conditions. The presence of
cross-diffusion terms in this system provides a potentially destabilizing
mechanism akin to a Turing instability~\cite{IidNimYam18}, which can lead to 
pattern formation, and intricate bifurcation diagrams of steady states have
indeed been obtained numerically for this model~\cite{BreKueSor21,IidMimNim06,KueSor20}, 
as also shown in Figure~\ref{fig:diagramSKT1D}. Following previous studies, we 
take $d_1=d_2=\lambda$ as the bifurcation parameter, while the other parameters 
are fixed as
\begin{equation} \label{crossdiffparams} 
  d_{11} = d_{21} = d_{22} = 0 \; , \;\;
  d_{12} = 3 \; , \;\;
  r_1 = 5 \; , \;\;
  r_2 = 2 \; , \;\;
  a_1 = a_2 = 3 \; , \;\;
  b_1 = b_2 = 1 \;.
\end{equation}
For this system as well, we validate a variety of symmetry-breaking pitchfork and transcritical bifurcation points, for one- and two-dimensional domains, one of which is illustrated in Figure~\ref{fig:SKT1D_3}. In fact, we will see that while in some cases the numerical
approximations do not provide a clear answer as to which type of bifurcation
point is observed, the computer-assisted proof result resolves the question
immediately.

\begin{figure}[!h]
\centering
      \includegraphics[width=0.7\textwidth]{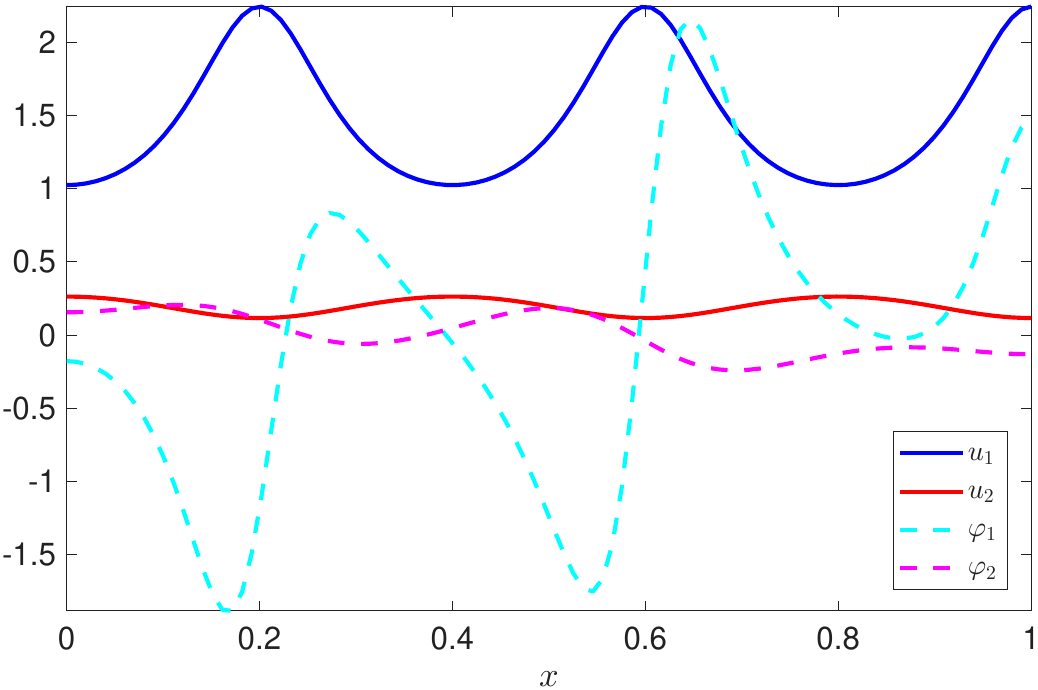}
      \caption{An approximate steady state $(u_1,u_2)$ of the SKT system~\eqref{parabolicsys} with $\Omega=(0,1)$ and parameters as in~\eqref{crossdiffparams}, which undergoes a symmetry-breaking pitchfork bifurcation at $\lambda\approx 3.227\times 10^{-3}$, with approximate eigenfunction $(\phi_1,\phi_2)$. This bifurcation is labeled $G$ in Table~\ref{tab:SKT}. The occurrence of this bifurcation is rigorously established in Theorem~\ref{thm:resultsSKT}.}
  \label{fig:SKT1D_3}
\end{figure}

The remainder of this paper is organized as follows. In Section~\ref{sec:genbif}
we describe the main theoretical results necessary for our approach. It is based
on a general bifurcation result in the sense of Crandall-Rabinowitz, which
provides precise conditions for the occurrence of a transcritical or pitchfork
bifurcation. In contrast to earlier results, the required assumptions focus
on the actual symmetry subspace of the primary branch, rather than the
equivariance operators describing its symmetry. In addition, we show that this
reformulation naturally allows us to use a suitable extended system combined
with a zero finding method to establish the respective bifurcation point. This
result is formulated in general Banach spaces to make it as applicable as 
possible, and we also discuss how its assumptions can be verified in practice.
In Section~\ref{sec:examples} we illustrate our approach by presenting new bifurcation 
results for both the Ohta--Kawasaki equation and the SKT system in dimensions one and two. 
The computer-assisted approach underlying these results is presented in 
Section~\ref{sec:CAPs}, where we explain how to rigorously validate zeros 
of the extended system introduced in Section~\ref{sec:genbif}. Finally the
additional work required to determine whether the obtained bifurcations are of 
transcritical or pitchfork type is described in Section~\ref{sec:bif}.
%
\section{A general bifurcation result}
\label{sec:genbif}
In this section we present a general bifurcation result which can be used to
establish the existence of transcritical or pitchfork bifurcation points in a variety
of symmetry-breaking settings. It is based on the celebrated Crandall-Rabinowitz
result~\cite{crandall:rabinowitz:71a}, which for the sake of completeness is
recalled in Section~\ref{subsec:genbif:crsimple}. Our new result is the subject
of Section~\ref{subsec:genbif:suffcond}. In this section we isolate the crucial
framework necessary for symmetry-breaking bifurcations with a focus on the 
underlying symmetry space, rather than focusing on the equivariance operators.
In Section~\ref{subsec:genbif:ext}, we continue in the spirit
of~\cite{lessard:sander:wanner:17a, rizzi:sander:wanner:24a} and demonstrate
how these symmetry-breaking bifurcations can be detected via root-finding in
a suitable extended system. The result contains  as special cases both
previously-developed cases:  The classical $\Z_2$-symmetry
breaking~\cite{lessard:sander:wanner:17a} and the more general cyclic
symmetry results of~\cite{rizzi:sander:wanner:24a}. In contrast to the
previous results,  this more general framework opens the door to
straightforward applications for partial differential equations on
higher-dimensional domains. Finally, in Section~\ref{subsec:genbif:sym}
we indicate how the main assumption of our bifurcation result can be
validated in practice, thereby establishing its versatility.
\subsection{Classical bifurcation from a simple eigenvalue}
\label{subsec:genbif:crsimple}
One of the central results in early bifurcation theory is the 
Crandall-Rabinowitz theorem on bifurcations from  simple
eigenvalues~\cite{crandall:rabinowitz:71a}.  We briefly recall
it in this section, in order to be able to contrast
the classical theory with our new result given in the next section. In the stated form, the theorem is the Banach
space analogue of the finite-dimensional version given in~\cite[Theorem~5.28,
p.~707]{sander:wanner:pdebook}.
\begin{theorem}[Secondary bifurcation from a simple eigenvalue]
\label{thm:simplesecondarybif}
For Banach spaces~$X$ and~$Y$ consider a smooth function
$F : \R \times X \to Y$, as well as the associated zero finding
problem   
\begin{equation} \label{thm:simplesecondarybif1}
  F(\lambda,u) = 0
  \qquad\mbox{ for }\qquad
  (\lambda,u) \in \R \times X \; .
\end{equation}
Let~$p : I \to X$, where $I \subset \R$ is an open interval, denote
a smooth function which satisfies
\begin{displaymath}
  F(\lambda,p(\lambda)) = 0
  \qquad\mbox{ for all }\qquad
  \lambda \in I \subset \R \; ,
\end{displaymath}
i.e., it establishes a primary solution branch
of~(\ref{thm:simplesecondarybif1}). Suppose further that for
some $\lambdabif \in I$ and $\ubif = p(\lambdabif)$ the Fr\'echet
derivative~$L = D_uF(\lambdabif,\ubif)$ is a Fredholm operator of index
zero which has a simple eigenvalue zero. That is, we assume that the
dimension of the nullspace of~$L$ is~$1$ and
that there is a nonzero element~$\phibif \in X$ such that $N(L) = \mathrm{span}[\phibif]$. 
Let~$\psibif^* \in Y^*$ denote any element of the dual space of~$Y$
for which $R(L) = N(\psibif^*)$. Finally, suppose that
we have
\begin{equation} \label{thm:simplesecondarybif2}
  D_{\lambda u} F(\lambdabif,\ubif)[\phibif] + 
    D_{uu} F(\lambdabif,\ubif)[\phibif,p'(\lambdabif)] \not\in
    R(D_uF(\lambdabif,\ubif)) \; .
\end{equation}
Then the point~$(\lambdabif,\ubif)$ is a bifurcation point for the
nonlinear problem~(\ref{thm:simplesecondarybif1}). Additionally,
the precise bifurcation type can be determined as follows:
\begin{itemize}
\item[(a)] If one has in addition that
\begin{equation} \label{thm:simplesecondarybif3}
  D_{uu} F(\lambdabif,\ubif)[\phibif,\phibif] \not\in
    R(D_uF(\lambdabif,\ubif)) \; ,
\end{equation}
then the point~$(\lambdabif,\ubif)$ is a transcritical bifurcation point.
\item[(b)] On the other hand, if we have
\begin{equation} \label{thm:simplesecondarybif4}
  D_{uu} F(\lambdabif,\ubif)[\phibif,\phibif] \in
    R(D_uF(\lambdabif,\ubif)) \; ,
\end{equation}
as well as

\begin{displaymath}
  \rhobif \; = \;
  \frac{\psibif^* D_{uuu}F(\lambdabif,\ubif)[\phibif,\phibif,\phibif] +
    3 \psibif^* D_{uu}F(\lambdabif,\ubif)[\phibif,\zetabif]}
    {3 \psibif^* D_{\lambda u} F(\lambdabif,\ubif)[\phibif] +
     3 \psibif^* D_{uu} F(\lambdabif,\ubif)[\phibif,p'(\lambdabif)]}
     \; \neq \; 0 \; ,
\end{displaymath}
where $\zetabif \in X$ is any element satisfying
\begin{displaymath}
  D_{uu} F(\lambdabif,\ubif)[\phibif,\phibif] +
  D_uF(\lambdabif,\ubif)[\zetabif] = 0 \; ,
\end{displaymath}
then the point~$(\lambdabif,\ubif)$ is a pitchfork bifurcation point.
If~$\rhobif$ is strictly positive, then the bifurcating solutions exist
for $\lambda < \lambdabif$ close to the bifurcation point, while
for negative~$\rhobif$ they exist for~$\lambda > \lambdabif$. In the
case~$\rhobif = 0$ the local shape of the bifurcating branch cannot be
determined conclusively, yet the bifurcating branch is still orthogonal
to the parameter direction at the bifurcation point.
\end{itemize}
\end{theorem}
The above theorem is remarkable in its simplicity. By just assuming
the existence of a solution branch $u = p(\lambda)$, and a noninvertible
Fr\'echet derivative at a point on this branch with simple zero 
eigenvalue, as well as the generically satisfied nondegeneracy
condition~(\ref{thm:simplesecondarybif2}), the result guarantees
a bifurcating solution branch. Parts~(a) and~(b) are basically just
window dressing in that they establish the nature of this new branch
in more detail. Nevertheless, while this is a remarkable {\em theoretical\/}
result, its usefulness hinges on knowledge of the primary solution
branch $u = p(\lambda)$ --- and this fact severely restricts its 
applicability in concrete situations. The only exception to this comment
is an existing {\em trivial solution branch\/} $p \equiv 0$, and in fact,
in many texts the above result is only formulated for this special
situation.

In the context of computer-assisted proofs,
Theorem~\ref{thm:simplesecondarybif} poses challenges. The primary branch
is usually only known in an approximate way, and therefore special care
has to be taken to obtain tight enough error estimates so that the nondegeneracy
condition can be verified at the (unknown) precise bifurcation point location.
While this can be achieved in many cases, see~\cite{kamimoto:kim:sander:wanner:22a}
for an example, it becomes significantly more difficult if one is interested in 
rigorously establishing branches of bifurcation points as
in~\cite{lessard:sander:wanner:17a}. Consequently, in our previous
work~\cite{lessard:sander:wanner:17a, rizzi:sander:wanner:24a}
we only considered the Ohta--Kawasaki model with zero total mass and
odd nonlinearity, which in turn provided additional information on the
primary branch. Essentially, this allowed us to consider a problem with
trivial solution after a suitable transformation. Yet, this special 
situation also implied the automatic validity
of~(\ref{thm:simplesecondarybif4}), which explains why these papers
only considered the case of pitchfork bifurcations.
\subsection{A sufficient condition for symmetry-breaking bifurcation}
\label{subsec:genbif:suffcond}
In order to turn Theorem~\ref{thm:simplesecondarybif} into a result which
is more flexible and general but also amenable to the use of computer-assisted proofs, it needs to be
reformulated in such a way that the existence of the primary branch does not 
need to be verified separately, and simply 
becomes a part of the theorem, with conditions which can be checked at a single fixed
parameter value. This
will be accomplished in the present section.

We begin by presenting a modified sufficient condition that ensures the
occurrence of a transcritical or a pitchfork bifurcation in nonlinear
problems. Rather than considering the generality of the above
Crandall-Rabinowitz result, we now restrict our attention to the
situation of {\em symmetry-breaking\/}. For our result below, we
need two assumptions, which we now state. 

The first part of Assumption~\ref{genbif:assump:nec} formalizes the statement that at the
point~$(\lambdabif, \ubif)$ the implicit function theorem fails, and therefore
this point has the potential to be a bifurcation point. The additional parts
of Assumption~\ref{genbif:assump:nec} give rise to a situation which
stays close to the one of finite-dimensional linear algebra.
 It is given as follows.
\begin{assumption}[Necessary condition for bifurcation]
\label{genbif:assump:nec}
Let~$X$ and~$Y$ denote two Banach spaces, and suppose that $F: \R \times X
\to Y$ is a smooth nonlinear operator. Furthermore, assume that there exists
a pair~$(\lambdabif, \ubif) \in \R \times X$ with
\begin{equation} \label{genbif:assump:nec1}
    F(\lambdabif, \ubif) = 0 \; ,
\end{equation}
and that the Fr\'echet derivative $L = D_uF(\lambdabif, \ubif) \in \cL(X,Y)$ is
a Fredholm operator of index zero. Finally, suppose that the nullspace
of~$L$ is one-dimensional and that we have both
\begin{equation} \label{genbif:assump:nec2}
    N(L) = \mathrm{span}[\phibif]
    \quad\mbox{ and }\quad
    R(L) = N(\psibif^*)
\end{equation}
for some nonzero elements $\phibif \in X$ and $\psibif^* \in Y^*$, where~$Y^*$
is the dual space of~$Y$.
\end{assumption}

While Assumption~\ref{genbif:assump:nec} above describes  the general
framework that is necessary for the Crandall-Rabinowitz approach, 
Assumption~\ref{genbif:assump:sym} below focuses specifically on the aspect
of symmetry-breaking. In particular, we assume the existence of closed
symmetry subspaces~$X_{\sym}$ and~$Y_{\sym}$  which are preserved by~$F$,
along with a Fredholm condition on how~$L$ acts on these spaces. 
\begin{assumption}[Symmetry-breaking]
\label{genbif:assump:sym}
Suppose that Assumption~\ref{genbif:assump:nec} is satisfied, and that there
exist closed symmetry subspaces $X_{\sym} \subset X$ and $Y_{\sym} \subset Y$ with
$F(\R, X_{\sym}) \subset Y_{\sym}$. In addition, assume that the restriction
$L|_{X_{\sym}} : X_{\sym} \to Y_{\sym}$ is again a Fredholm operator of index zero,
and that we have both
\begin{equation} \label{genbif:assump:sym1}
    \ubif \in X_{\sym}
    \quad\mbox{ and }\quad
    \phibif \not\in X_{\sym} \; .
\end{equation}
\end{assumption}

We would like to emphasize that an especially crucial aspect
of Assumption~\ref{genbif:assump:sym} is the fact that the restriction
$L|_{X_{\sym}} : X_{\sym} \to Y_{\sym}$ is again a Fredholm operator of
index zero. Note that this is by no means automatic. For example, there
are situations where one could choose a finite-dimensional~$X_{\sym}$
together with $Y_{\sym} = Y$, and in the case of infinite-dimensional~$Y$
this would violate the finite codimension of~$L(X_{\sym})$. Even if the
restricted operator~$L|_{X_{\sym}} : X_{\sym} \to Y_{\sym}$ is Fredholm,
one can easily come up with examples in which its index is nonzero, by
choosing the space~$Y_{\sym}$ too large. Nevertheless, as we will elaborate in
Section~\ref{subsec:genbif:sym}, in our applications the verification
of Assumption~\ref{genbif:assump:sym} amounts to little more than considering
the differential operator describing~$L$ for functions which are defined on a
suitable subdomain of~$\Omega$.

Notice also that Assumption~\ref{genbif:assump:sym} does not
rely on a priori knowledge of a primary solution branch for the nonlinear
problem $F(\lambda, u) = 0$. In fact, its formulation enables us to
prove the existence of such a branch as part of the theorem, using only 
information at a single point~$(\lambdabif,\ubif)$. In this
way, the more specific requirements in Assumptions~\ref{genbif:assump:nec}
and~\ref{genbif:assump:sym} remove one of the main disadvantages of
Theorem~\ref{thm:simplesecondarybif}.

The following result confirms that the combination of the above two assumptions
provides a sufficient condition for the existence of a transcritical or pitchfork
bifurcation. In fact, the only additional requirement is a mild nondegeneracy
condition.
\begin{theorem}[Symmetry breaking bifurcations]
\label{genbif:thm:suffcond}
Suppose that Assumptions~\ref{genbif:assump:nec} and~\ref{genbif:assump:sym}
are satisfied, i.e., in particular we have $F(\lambdabif,\ubif) = 0$, 
$L\phibif = D_u F(\lambdabif,\ubif)[\phibif] = 0$, the nullspace of~$L$ is
one-dimensional, $\ubif \in X_{\sym}$, and $\phibif \not\in X_{\sym}$. Under these
assumptions, there exists a unique function $\xibif \in X_{\sym}$ such that
\begin{equation} \label{genbif:thm:suffcond1}
  L\xibif + D_\lambda F(\lambdabif,\ubif) = 0 \; . 
\end{equation}
Under the further assumption that the nondegeneracy condition
\begin{equation} \label{genbif:thm:suffcond2}
  D_{\lambda u} F(\lambdabif,\ubif)[ \phibif ] +
    D_{uu} F(\lambdabif,\ubif)[\phibif,\xibif] \notin R(L)
\end{equation}
is satisfied, the point~$(\lambdabif,\ubif)$ is a bifurcation point,
and the following hold:
\begin{itemize}
\item[(a)] If $D_{uu}F(\lambdabif,\ubif)[\phibif,\phibif] \not\in R(L)$,
then there is a transcritical bifurcation at~$(\lambdabif,\ubif)$.
\item[(b)] If $D_{uu}F(\lambdabif,\ubif)[\phibif,\phibif]
\in R(L)$, then there is a pitchfork bifurcation at~$(\lambdabif,\ubif)$.
\end{itemize}
In both cases, the bifurcating solution branch is not a subset of the
symmetry space~$X_{\sym}$. That is, the point~$(\lambdabif,\ubif)$ is a symmetry-breaking
bifurcation point.
\end{theorem}
\begin{proof}
We first show that~(\ref{genbif:thm:suffcond1}) has a unique solution
$\xibif \in X_{\sym}$. To see this, note that the inclusion $F(\R,X_{\sym}) \subset
Y_{\sym}$ and $\ubif \in X_{\sym}$ immediately imply $D_\lambda F(\lambdabif,\ubif)
\in Y_{\sym}$. Due to $\phibif \not\in X_{\sym}$ the restriction~$L|_{X_{\sym}} :
X_{\sym} \to Y_{\sym}$ is one-to-one. Since this restricted operator is a one-to-one 
index zero Fredholm operator, it is also onto. 
This establishes the existence of a solution~$\xibif \in X_{\sym}$ 
solving~\eqref{genbif:thm:suffcond1}. Since
the difference of any two such solutions is contained in~$N(L) \cap X_{\sym} =
\{ 0 \}$, the uniqueness follows.

The proof of the second part of the theorem is based on a standard 
Lyapunov-Schmidt reduction argument. Let~$\tilde{Y}$ denote an 
arbitrary one-dimensional subspace of~$Y$
such that $Y = R(L) \oplus \tilde{Y}$. In view of $\phibif \not\in X_{\sym}$,
there is a closed subspace~$\tilde{X} \subset X$ such that both $X_{\sym} \subset
\tilde{X}$ and $X = N(L) \oplus \tilde{X}$ are satisfied. Such a space can
be defined as~$\tilde{X} = N(\psi^*)$, where $\psi^* \in X^*$ is any functional
which vanishes on~$X_{\sym}$ and is nontrivial on~$N(L) = \mathrm{span}[\phibif]$,
and which exists due to the Hahn-Banach theorem. Finally, let $P : Y \to Y$
and $Q : X \to X$ denote the continuous projectors with $R(P) = \tilde{Y}$
and $N(P) = R(L)$, as well as $R(Q) = N(L)$ and $N(Q) = \tilde{X}$. In
other words, both projectors have rank one.

The properties of~$P$ and~$Q$ imply that for any $u \in X$ the elements
$v = Qu$ and $w = (I-Q)u$ satisfy $u = v + w \in N(L) \oplus \tilde{X}$.
Furthermore, the nonlinear equation $F(\lambda,u) = 0$ is equivalent to
the system
\begin{equation} \label{genbif:thm:suffcond3}
  P F(\lambda,v+w) = 0
  \qquad\mbox{ and }\qquad
  G(\lambda,v,w) := (I-P) F(\lambda,v+w) = 0 \; ,
\end{equation}
where $G: \R \times N(L) \times \tilde{X} \to R(L)$. 
If we write $\ubif = v_* + w_* \in N(L) \oplus \tilde{X}$, then
$\ubif \in X_{\sym} \subset \tilde{X}$ yields both $v_* = 0$ and
$w_* = \ubif$. Moreover, the Fr\'echet
derivative~$D_wG(\lambdabif,0,\ubif)$ satisfies
\begin{displaymath}
    D_wG(\lambdabif,0,\ubif) =
    (I-P) D_uF(\lambdabif,\ubif)|_{\tilde{X}} \in
    \cL(\tilde{X}, R(L)) \; . 
\end{displaymath}
The fact that $N(L) \cap \tilde{X} = \{ 0 \}$ implies that this 
Fr\'echet derivative is invertible. An application of the implicit
function theorem guarantees the existence of open neighborhoods
$\lambdabif \in \Lambda_0 \subset \R$ and $0 \in V_0 \subset N(L)$,
as well as $\ubif \in \tilde{X}_0 \subset \tilde{X}$, 
and a smooth function $W : \Lambda_0 \times V_0 \to
\tilde{X}_0$,  
such that
for all $(\lambda,v) \in \Lambda_0 \times V_0$ one has
\begin{equation} \label{genbif:thm:suffcond4}
  G(\lambda, v, W(\lambda,v)) =
  (I-P) F(\lambda,v+W(\lambda,v)) = 0 \; ,
\end{equation}
and such that~$W(\lambda,v) \in \tilde{X}_0$ is the unique value 
in $\tilde{X}_0$
with
this property. Note that this implies in particular the
identity $W(\lambdabif,0) = \ubif$.

In order to solve the full original problem $F(\lambda,u) = 0$,
one simultaneously also has to solve the first equation
in~(\ref{genbif:thm:suffcond3}). Since~$N(L)$ is one-dimensional,
every $v \in N(L)$ can be uniquely written as $\alpha \phibif$, where $\alpha \in \R$. 
Now consider the bifurcation function~$b: \R^2 \to \R$ defined via
\begin{equation} \label{genbif:thm:suffcond5}
  b(\lambda,\alpha) :=
  \psibif^* PF(\lambda, \alpha \phibif + W(\lambda,\alpha \phibif))
  \; ,
\end{equation}
where~$\psibif^*$ was introduced in~(\ref{genbif:assump:nec2}).
Then one has $F(\lambda,v+w) = 0$ for~$(\lambda,v,w) \in \Lambda_0
\times V_0 \times \tilde{X}_0$ if and only if both $b(\lambda,\alpha) = 0$
and $u = v+w = \alpha \phibif + W(\lambda,\alpha \phibif)$ are satisfied.

We now show that the bifurcation equation $b(\lambda,\alpha) = 0$ has the
solution~$(\lambda,0)$ for all~$\lambda$ sufficiently near~$\lambdabif$. To see this,
consider the restricted operator $F : \R \times X_{\sym} \to Y_{\sym}$,
which has the zero $(\lambdabif,\ubif) \in \R \times X_{\sym}$. The Fr\'echet
derivative at this point is given
by the restriction~$L|_{X_{\sym}} : X_{\sym} \to Y_{\sym}$, and in view of $\phibif \notin X_{\sym}$ and
the fact that it is Fredholm with index zero it also has a continuous
inverse. Thus, after possibly reducing the size of~$\Lambda_0$,
one can apply the implicit function theorem to obtain a smooth function
$p : \Lambda_0 \to X_{\sym}$ which satisfies $F(\lambda,p(\lambda)) = 0$ for
all $\lambda \in \Lambda_0$. 
For each fixed $\lambda$, the value~$W(\lambda,0) \in
\tilde{X}_0$ is uniquely determined by~(\ref{genbif:thm:suffcond4}).
Since $p(\lambda) \in X_{\sym} \subset \tilde{X}$, this 
immediately yields that  $W(\lambda,0) = p(\lambda)$ and $b(\lambda,0) =
\psibif^* PF(\lambda, p(\lambda)) = 0$ for all $\lambda \in \Lambda_0$,
again after potentially reducing the size of the open
set~$\Lambda_0 \ni \lambdabif$.

The trivial solution immediately furnishes the identity
$\partial b/\partial \lambda^k(\lambda,0) = 0$ for all~$k$, and as in the proof
of~\cite[Theorem~3.2]{rizzi:sander:wanner:24a} one can verify
$b_\alpha(\lambdabif,0) = 0$. Consider now the function
\begin{displaymath}
  r(\lambda,\alpha) = \left\{ \begin{array}{rcl}
                        \DS \frac{b(\lambda,\alpha)}{\alpha} & \mbox{for} &
                          \alpha \neq 0 \; , \\[3ex]
                        \DS b_\alpha(\lambda,0)
                          & \mbox{for} & \alpha = 0 \; ,
                      \end{array} \right.
\end{displaymath}
which is defined and smooth in a neighborhood of~$(\lambdabif,0)$, and
which will be used to construct the bifurcating branch which breaks the
$X_{\sym}$-symmetry. As in~\cite[Proposition~2.11]{lessard:sander:wanner:17a}
one can show that due to $r(\lambdabif, 0) = b_\alpha(\lambdabif,0) = 0$
the function~$r$ has the expansion
\begin{eqnarray*}
    r(\lambda, \alpha) & = &  
    (\lambda - \lambdabif) \, b_{\lambda\alpha}(\lambdabif,0) +
    \frac{1}{2} \, \alpha \, b_{\alpha\alpha}(\lambdabif,0) +
    \frac{1}{2} \, (\lambda - \lambdabif)^2 \,
      b_{\lambda\lambda\alpha}(\lambdabif,0) + \\[1ex]
    & & \quad +
    \frac{1}{2} \, (\lambda - \lambdabif) \alpha \,
      b_{\lambda\alpha\alpha}(\lambdabif,0) +
  \frac{1}{6} \, \alpha^2 \, b_{\alpha\alpha\alpha}(\lambdabif,0) +
    R(\lambda - \lambdabif, \alpha)
\end{eqnarray*}
with $R(\nu,\alpha) = O(\| (\nu,\alpha) \|^3)$. 
Note that since we do not assume any even or odd symmetry, we 
have the additional~$b_{\alpha \alpha}$ term compared to~\cite{lessard:sander:wanner:17a}. 
Now let~$\xibif \in X_{\sym}$ be defined by the equation
$L\xibif + (I-P) D_\lambda F(\lambdabif,\ubif) = 0$. 
The value is unique since $L|_{X_{\sym}}$ is continuously invertible.
Notice, however, that then~$\xibif$ also satisfies~(\ref{genbif:thm:suffcond1}),
since $\ubif \in X_{\sym}$ and $F(\R,X_{\sym}) \subset Y_{\sym}$ 
yield
the inclusion $D_\lambda F(\lambdabif,\ubif) \in Y_{\sym} \subset R(L)$,
and therefore $PD_\lambda F(\lambdabif,\ubif) = 0$.
According
to~\cite[Table~1]{lessard:sander:wanner:17a} along with the 
nondegeneracy condition~\eqref{genbif:thm:suffcond2} we have
\begin{equation}\label{genbif:thm:suffcond2B}
  r_{\lambda}(\lambdabif,0) = b_{\lambda\alpha}(\lambdabif,0) =
  \psibif^* D_{\lambda u}F(\lambdabif,\ubif)[\phibif] +
  \psibif^* D_{uu}F(\lambdabif,\ubif)[\phibif,\xibif] \neq 0 \; . 
\end{equation}
An application of the implicit function theorem to $r(\lambda,
\alpha) = 0$ then yields a smooth real-valued function~$\alpha \mapsto
h(\alpha)$, which is defined on an open interval around~$0$,
satisfies $h(0) = \lambdabif$, and such that in a neighborhood
of~$(\lambdabif,0)$ we have $r(\lambda, \alpha) = 0$ if and only
if $\lambda = h(\alpha)$. This establishes the second solution
branch $\alpha \mapsto \alpha \phibif + W(h(\alpha), \alpha \phibif)$.

It remains to establish the precise bifurcation type. The
bifurcation equation $b(\lambda,\alpha) = 0$ has both a trivial
solution indexed by~$\lambda$, which corresponds to the branch of
symmetric solutions, and a secondary branch indexed by~$\alpha$, which
consists of solutions breaking the primary symmetry. Thus, the specific
bifurcation type is determined by the value of the derivative~$h'(0)$.
If we have $h'(0) \neq 0$, then the bifurcation is transcritical,
and it is of pitchfork type if one has $h'(0) = 0$. In the latter case,
it is supercritical for $h''(0) > 0$, subcritical if $h''(0) < 0$,
and degenerate for $h''(0) = 0$. In fact, in the final case it might
not have the typical pitchfork shape, as the shape is determined 
by the first nonzero derivative of~$h$ at~$0$. Nevertheless, we
still refer to it as a pitchfork bifurcation.

In order to find the first two derivatives of the function~$h$
describing the secondary branch we twice differentiate its defining
equation, i.e., the identity
\begin{displaymath}
    r(h(\alpha), \alpha) = 0
    \quad\mbox{ for all~$\alpha$ near~$0$.}
\end{displaymath}
This leads to the equation
\begin{displaymath}
    r_\lambda(h(\alpha),\alpha) h'(\alpha) +
    r_\alpha(h(\alpha),\alpha) = 0 \; ,
\end{displaymath}
as well as to
\begin{eqnarray}\label{hdoubleprime}
    & & r_\lambda(h(\alpha),\alpha) h''(\alpha) +
    r_{\lambda\lambda}(h(\alpha),\alpha) h'(\alpha)^2 +
    r_{\lambda\alpha}(h(\alpha),\alpha) h'(\alpha) + \\[1ex] \nonumber
    & & \qquad + \,
    r_{\alpha\lambda}(h(\alpha),\alpha) h'(\alpha) +
    r_{\alpha\alpha}(h(\alpha),\alpha) = 0 \; .
\end{eqnarray}
Evaluating the first equation at $\alpha = 0$ yields, together
with~\cite[Table~1]{lessard:sander:wanner:17a}, the explicit 
formula
\begin{equation} \label{genbif:thm:suffcond6}
    h'(0) \; = \; -\frac{r_\alpha(\lambdabif,0)}{r_\lambda(\lambdabif,0)} = 
    -\frac{\frac{1}{2} \psibif^* D_{uu}F(\lambdabif,\ubif)[\phibif,\phibif]}
      {\psibif^* D_{\lambda u}F(\lambdabif,\ubif)[\phibif] +
       \psibif^* D_{uu}F(\lambdabif,\ubif)[\phibif,\xibif]} \; ,
\end{equation}
which by \eqref{genbif:thm:suffcond2B} has a nonzero denominator. This establishes
both~{\it (a)\/} and~{\it (b)}. For later reference, we also note
that if one assumes $h'(0) = 0$ and evaluates 
\eqref{hdoubleprime} at $\alpha = 0$,
then~\cite[Table~1]{lessard:sander:wanner:17a} furnishes
\begin{equation} \label{genbif:thm:suffcond7}
  h''(0) \; = \; -\frac{r_{\alpha \alpha}(\lambdabif,0)}{r_\lambda(\lambdabif,0)} = 
 - \frac{\frac{1}{3}\left(\psibif^* D_{uuu}F(\lambdabif,\ubif)[\phibif,\phibif,\phibif] +
    3 \psibif^* D_{uu}F(\lambdabif,\ubif)[\phibif,\eta_*] \right)}
    {  \psibif^* D_{\lambda u} F(\lambdabif,\ubif)[\phibif] + 
       \psibif^* D_{uu} F(\lambdabif,\ubif)[\phibif,\xibif]} \; ,
\end{equation}
where $\eta_* \in \tilde{X}$ is the unique element satisfying
\begin{equation} \label{genbif:thm:suffcond8}
    L \eta_* + 
    (I-P) D_{uu} F(\lambdabif,\ubif)[\phibif,\phibif] = 0 \; .
\end{equation}
This completes the proof of the theorem.
\end{proof}

\medskip
Note that in contrast to the classical approach to
$\Z_2$-symmetry breaking, as well as to our previous results
in~\cite{lessard:sander:wanner:17a, rizzi:sander:wanner:24a},
the above theorem does not require knowledge of a full space
decomposition of~$X$ or~$Y$. Consequently there is also no
requirement on~$L$ to have all the spaces in the decomposition
as invariant subspaces. The {\em only\/} requirement is the
invariance of the symmetry space coupled with the Fredholm property,
and this renders Theorem~\ref{genbif:thm:suffcond} much more applicable.
It needs to be mentioned, however, that in specific applications
the knowledge of such a space decomposition can be extremely
useful --- and this will be the case in our main applications
in this paper. For example, equivariance in combination with
a suitable space decomposition might allow one to simplify the
defining equation~(\ref{genbif:thm:suffcond8}) for~$\eta_*$
by omitting the projector~$P$, as mentioned for example
in~\cite[Remark~2.13]{lessard:sander:wanner:17a}.
\subsection{Detecting bifurcations via extended systems}
\label{subsec:genbif:ext}
While Theorem~\ref{genbif:thm:suffcond} is a typical sufficient
result for establishing the existence of certain bifurcation points,
its assumptions are not readily verifiable in all situations. For
example, if one is interested in studying secondary bifurcations
from a nontrivial branch of solutions, one rarely has access to
the functions~$\ubif$ and~$\phibif$ in Assumption~\ref{genbif:assump:nec},
or even the precise parameter value~$\lambdabif$ of the potential
bifurcation point.

To address this problem, the current section is devoted to reformulating
the sufficient condition of Theorem~\ref{genbif:thm:suffcond} in the
form of a zero-finding problem for a suitable extended system. Our
approach follows the ideas from~\cite{lessard:sander:wanner:17a,
rizzi:sander:wanner:24a}, but has been extended to the more
general setting of this paper.

For this, consider two Banach spaces~$X$ and~$Y$, as well as a smooth
parameter-dependent operator $F : \R \times X \to Y$. In addition, let
$X_{\sym} \subset X$ and $Y_{\sym} \subset Y$ denote two closed symmetry subspaces
such that 
\begin{equation}\label{eqn:inclusionf}
F(\R, X_{\sym}) \subset Y_{\sym} \; ,
\end{equation} 
and choose a fixed nontrivial
element $\ell \in X^*$. Then in the following we consider the
{\em extended system for~$F$\/} given by
\begin{equation} \label{genbif:eqn:extsys1}
  \cF(\lambda,u,\phi) = (0,0,0)
  \quad\mbox{with}\quad
  \cF : \left\{ 
    \begin{array}{c}
      \R \times X_{\sym} \times X \to \R \times Y_{\sym} \times Y \\[0.5ex]
      (\lambda,u,\phi) \mapsto (\ell(\phi)-1, F(\lambda,u),
        D_uF(\lambda,u)[\phi]) .
    \end{array}
  \right.
\end{equation}
The operator~$\cF$ is well-defined in view of the inclusion
\eqref{eqn:inclusionf}. Moreover, its Fr\'echet derivative
in $\cL(\R \times X_{\sym} \times X, \, \R \times Y_{\sym} \times Y)$
satisfies
\begin{eqnarray}
  D\cF(\lambda, u, \phi)[\tilde{\lambda},\tilde{u},\tilde{\phi}] & = &
    \left( \ell(\tilde{\phi}) \; , \;\;
    \tilde{\lambda} D_\lambda F(\lambda,u) + D_u F(\lambda,u)[\tilde{u}] ,
    \right. \nonumber \\[0.5ex]
  & & \;\;\;\left. \tilde{\lambda} D_{\lambda u} F(\lambda,u)[\phi] +
    D_{uu} F (\lambda,u)[\phi,\tilde{u}] + D_u F(\lambda,u)[\tilde{\phi}]
    \right) .
    \label{genbif:eqn:extsys2}
\end{eqnarray}
One can readily see that if~$(\lambdabif, \ubif, \phibif)$ is a zero of
the extended system, i.e., if we have the identity $\cF(\lambdabif, \ubif,
\phibif) = 0$, then~$(\lambdabif, \ubif)$ is a symmetric zero of~$F$,
and~$\phibif$ is a kernel function of the Fr\'echet derivative of~$F$
at this zero. As the following result shows, if we add both the
nondegeneracy of the zero of the extended system and the condition
$\phibif \not\in X_{\sym}$ to the mix, then we have actually established
a symmetry-breaking bifurcation point.
\begin{theorem}[Symmetry-breaking bifurcations via extended systems]
\label{genbif:thm:extsys}
Let~$X$ and~$Y$ be two Banach spaces, and suppose that the
parameter-dependent operator $F : \R \times X \to Y$ is smooth.
Moreover, let $X_{\sym} \subset X$ and $Y_{\sym} \subset Y$ be closed
subspaces with $F(\R, X_{\sym}) \subset Y_{\sym}$. Finally, assume that
for every~$\lambda \in \R$ and every~$u \in X_{\sym}$ the Fr\'echet
derivative~$D_u F(\lambda,u) \in \cL(X,Y)$ and its restriction
$D_u F(\lambda,u)|_{X_{\sym}} \in \cL(X_{\sym},Y_{\sym})$ are Fredholm operators
of index zero. Then we have:
\begin{itemize}
\item[(a)] Suppose that all the assumptions of
Theorem~\ref{genbif:thm:suffcond} hold, and choose the
functional~$\ell \in X^*$ in such a way that $\ell(\phibif) = 1$.
Then the Fr\'echet derivative~$D\cF(\lambdabif,\ubif,\phibif)$ of the
operator defined in~(\ref{genbif:eqn:extsys1}) is invertible, i.e.,
the solution~$(\lambdabif,\ubif,\phibif) \in \R \times X_{\sym} \times X$
of the extended system
\begin{equation} \label{genbif:thm:extsys1}
  \cF(\lambda,u,\phi) = (0,0,0)
\end{equation}
is an isolated nondegenerate zero.
\item[(b)] Conversely, if there exists an~$\ell \in X^*$ and
a~$\phibif \in X \setminus X_{\sym}$ such that~$(\lambdabif,\ubif,\phibif)$
is a zero of the map~$\cF$, and if the Fr\'echet
derivative~$D\cF(\lambdabif,\ubif,\phibif)$ is invertible, then the
nonlinear operator~$F$ satisfies all the assumptions of
Theorem~\ref{genbif:thm:suffcond}.
\end{itemize}
In other words, the nonlinear problem~$F(\lambda,u) = 0$ has a
symmetry-breaking transcritical or pitchfork bifurcation point
at~$(\lambdabif,\ubif) \in \R \times X_{\sym}$ if and only if the
triple~$(\lambdabif,\ubif,\phibif)$ is a nondegenerate zero of~$\cF$
with $\phibif \not\in X_{\sym}$.
\end{theorem}
\begin{proof}
We first establish the validity of {\em (a)\/}. The assumptions
of Theorem~\ref{genbif:thm:suffcond} immediately imply
that~$(\lambdabif,\ubif,\phibif) \in \R \times  X_{\sym} \times X$
is a solution of~(\ref{genbif:thm:extsys1}), since we assume
$\ell(\phibif) = 1$.

For the verification of the injectivity of the Fr\'echet
derivative~$D\cF(\lambdabif,\ubif,\phibif)$, suppose there exists
a solution $(\tilde{\lambda},\tilde{u},\tilde{\phi}) \in \R
\times X_{\sym} \times X$ of the equation
\begin{equation} \label{genbif:thm:extsys2}
  D\cF(\lambdabif,\ubif,\phibif)[\tilde{\lambda},\tilde{u},\tilde{\phi}]
  = (0,0,0) \; .
\end{equation}
We have to show that this implies $(\tilde{\lambda},\tilde{u},
\tilde{\phi}) = (0,0,0)$. Suppose first that $\tilde{\lambda}
\ne 0$. Then the second component in~(\ref{genbif:thm:extsys2})
can be rewritten as
\begin{displaymath}
    D_\lambda F(\lambdabif,\ubif) +
    L[ \tilde{u} / \tilde{\lambda}] = 0 \; ,
\end{displaymath}
and since~$\xibif \in X_{\sym}$ is the unique solution
of~(\ref{genbif:thm:suffcond1}) one further obtains
$\xibif = \tilde{u} / \tilde{\lambda}$. Now the third
component of~(\ref{genbif:thm:extsys2}) yields
\begin{displaymath}
  D_{\lambda u} F(\lambdabif,\ubif)[\phibif] +
    D_{uu} F (\lambdabif,\ubif)[\phibif,\xibif] =
  -L[\tilde{\phi} / \tilde{\lambda}] \in R(L) \; ,
\end{displaymath}
which contradicts~(\ref{genbif:thm:suffcond2}). Thus, the identity
$\tilde{\lambda} = 0$ has to be satisfied, and the second components
of both~(\ref{genbif:thm:extsys2}) and~(\ref{genbif:eqn:extsys2})
imply $L[\tilde{u}] = 0$, i.e., one has $\tilde{u} \in N(L) \cap
X_{\sym} = \{ 0 \}$. Substituting the identities $\tilde{\lambda} = 0$
and $\tilde{u} = 0$ into the third component immediately yields
$L \tilde{\phi} = 0$, and therefore $\tilde{\phi} = \alpha \phibif$.
Finally, the first component of~(\ref{genbif:thm:extsys2}) implies
$0 = \ell(\alpha \phibif) = \alpha$, which in turn gives the equality
$\tilde{\phi}  = 0$ and completes the verification
that~$D\cF(\lambdabif,\ubif,\phibif)$ is one-to-one.

Next we consider the surjectivity of~$D\cF(\lambdabif,\ubif,\phibif)$.
Let $(\tau,y,z) \in \R \times Y_{\sym} \times Y$ be arbitrary. Note
first that since $N(L) \cap X_{\sym} = \{ 0 \}$, and thus
$y \in Y_{\sym} = L(X_{\sym})$, there exists a unique
$\tilde{u} \in X_{\sym}$ with $L[\tilde{u}] = y$. Due to
Assumption~\ref{genbif:assump:nec} we have $R(L) = N(\psibif^*)$.
Now define
\begin{displaymath}
  \tilde{\lambda} =
  \frac{\psibif^* \left( z - D_{uu}F(\lambdabif,\ubif)[\phibif,\tilde{u}]\right)}
    {\psibif^* \left( D_{\lambda u}F(\lambdabif,\ubif)[ \phibif ] +
     D_{uu}F(\lambdabif,\ubif)[\phibif,\xibif] \right)} \; ,
\end{displaymath}
where the denominator is nonzero in view of~(\ref{genbif:thm:suffcond2}).
This identity can be rewritten in the form
\begin{displaymath}
  \psibif^* \left( \tilde{\lambda} D_{\lambda u}F(\lambdabif,\ubif)[ \phibif ] +
    D_{uu}F(\lambdabif,\ubif)[\phibif,\tilde{u} + \tilde{\lambda} \xibif] -
    z \right) = 0 \; .
\end{displaymath}
The choice of~$\psibif^*$ implies that the argument in this equation
has to be contained in~$R(L)$, and since $L[\beta \phibif] = 0$ for
any~$\beta \in \R$, there exists a $\tilde{\phi} \in X$ such that
for every~$\beta$ the equation
\begin{equation} \label{genbif:thm:extsys3}
  \tilde{\lambda} D_{\lambda u}F(\lambdabif,\ubif)[ \phibif ] +
    D_{uu}F(\lambdabif,\ubif)[\phibif,\tilde{u} + \tilde{\lambda}\xibif] +
    L[\tilde{\phi} + \beta \phibif] = z
\end{equation}
holds. Moreover, the definition of~$\xibif$
in~(\ref{genbif:thm:suffcond1}) furnishes
\begin{equation} \label{genbif:thm:extsys4}
  \tilde{\lambda} D_\lambda F (\lambdabif,\ubif) +
    L[\tilde{u} + \tilde{\lambda}\xibif] =
  \tilde{\lambda} \left( D_\lambda F (\lambdabif,\ubif) +
    L [\xibif] \right) + L[\tilde{u}] =
  L[\tilde{u}] = y \; . 
\end{equation}
Finally, the inclusion $\xibif \in X_{\sym}$ implies $\tilde{u} +
\tilde{\lambda} \xibif \in X_{\sym}$, and so for all $\beta \in \R$
the identities in~(\ref{genbif:thm:extsys3})
and~(\ref{genbif:thm:extsys4}) give the second
and third components of the desired equation
\begin{displaymath}
  D\cF(\lambdabif,\ubif,\phibif)[\tilde{\lambda} , \;
    \tilde{u} + \tilde{\lambda} \xibif, \;
    \tilde{\phi} + \beta \phibif] = (\tau,y,z) \; .
\end{displaymath}
As the last step, one needs to choose~$\beta$ so that the
first component of the equation also holds. Due to
$\ell(\phibif) = 1$, this reduces to $\tau = \ell(\tilde{\phi}
+ \beta \phibif) = \ell(\tilde{\phi}) + \beta$, and thus
$\beta = \tau - \ell(\tilde{\phi})$. This shows that the
Fr\'echet derivative~$D\cF(\lambdabif, \ubif, \phibif)$ is onto,
and completes the proof of~{\em (a)\/}.

We now proceed to the verification of part~{\em (b)\/}, i.e., we
assume that there exists an~$\ell \in X^*$ and a~$\phibif \in
X \setminus X_{\sym}$ such that~$(\lambdabif,\ubif,\phibif)$ is a zero
of~$\cF$ with invertible Fr\'echet derivative~$D\cF(\lambdabif,
\ubif, \phibif) \in \cL(\R \times X_{\sym} \times X, \; \R \times Y_{\sym}
\times Y)$. We have to show that then all assumptions of
Theorem~\ref{genbif:thm:suffcond} are satisfied.

In view of our assumptions for Theorem~\ref{genbif:thm:extsys},
only a few parts of Assumptions~\ref{genbif:assump:nec}
and~\ref{genbif:assump:sym} have to be established. From
$\cF(\lambdabif,\ubif,\phibif) = 0$ one readily obtains
both $F(\lambdabif,\ubif) = 0$ and the identities $L[\phibif] =
D_uF(\lambdabif,\ubif)[\phibif] = 0$. Moreover, the kernel
element~$\phibif$ is nontrivial due to $\ell(\phibif) = 1$, and
this yields $\dim N(L) \ge 1$. That this last inequality is
in fact an equality follows as in the proof
of~\cite[Theorem~2.12]{lessard:sander:wanner:17a} from the
invertibility of~$D\cF(\lambdabif, \ubif, \phibif)$.

We still have to verify the nondegeneracy
condition~(\ref{genbif:thm:suffcond2}). For this, let~$z \in Y
\setminus R(L)$ be arbitrary. Then the invertibility of~$D\cF(\lambdabif,
\ubif, \phibif)$ guarantees a triple~$(\tilde{\lambda}, \tilde{u},
\tilde{\phi}) \in \R \times X_{\sym} \times X$ which satisfies the
identity
\begin{equation} \label{genbif:thm:extsys5}
    D\cF(\lambdabif, \ubif, \phibif)
    [\tilde{\lambda}, \tilde{u}, \tilde{\phi}]
    = (0,0,z) \; .
\end{equation}
Let~$\xibif$ denote the unique solution of~(\ref{genbif:thm:suffcond1}).
If we assume $\tilde{\lambda} = 0$, then the explicit form of the second
component of the above identity yields $0 = \tilde{\lambda} D_\lambda
F(\lambdabif,\ubif) + L \tilde{u} = L \tilde{u}$, which in turn implies
$\tilde{u} \in N(L) \cap X_{\sym} = \{ 0 \}$. Now the last component
of~(\ref{genbif:thm:extsys5}) reads 
\begin{displaymath}
  z = \tilde{\lambda} D_{\lambda u} F(\lambdabif,\ubif)[\phibif] +
    D_{uu} F(\lambdabif,\ubif)[\phibif,\tilde{u}] + L \tilde{\phi} =
  L \tilde{\phi} \; ,
\end{displaymath}
which contradicts $z \notin R(L)$. Thus, we have to have
$\tilde{\lambda} \ne 0$. But then the second component
of~(\ref{genbif:thm:extsys5}) furnishes
$D_\lambda F(\lambdabif,\ubif) + L[ \tilde{u} / \tilde{\lambda}] = 0$,
and this in turn implies $\xibif = \tilde{u} / \tilde{\lambda}$.
Substituting this into the third component
of~(\ref{genbif:thm:extsys5}), one finally obtains
\begin{displaymath}
  D_{\lambda u} F(\lambdabif,\ubif)[\phibif] +
    D_{uu} F(\lambdabif,\ubif)[\phibif,\xibif] \; = \;
    (z - L\tilde{\phi}) / \tilde{\lambda}
  \; \not\in \; R(L) \; ,
\end{displaymath}
since we assumed that $z \not\in R(L)$, yet clearly
$L \tilde{\phi} \in R(L)$. This establishes the nondegeneracy
condition~(\ref{genbif:thm:suffcond2}) and completes the
proof of the theorem.
\end{proof}

\medskip
As in our previous work~\cite{lessard:sander:wanner:17a,
rizzi:sander:wanner:24a}, the above result enables one to
use standard computer-assisted proof techniques to establish
the existence of transcritical or pitchfork bifurcations.
All one has to do is to identify the correct symmetry
space~$X_{\sym}$, and find a verifiable condition to make sure
that the kernel function~$\phibif$ does not lie in this 
symmetry subspace. This is a significant simplification from
our previous approach. However, establishing the actual
type of the bifurcation still relies on knowledge of the second
derivative $D_{uu}F(\lambdabif,\ubif)[\phibif,\phibif]$. For
this, it is extremely useful to have equivariance 
properties available, since such properties often simplify
the task.
\subsection{Verifying the symmetry assumptions}
\label{subsec:genbif:sym}
We close this section with a few comments concerning
Assumption~\ref{genbif:assump:sym}, which is the main symmetry
assumption underlying our above results. In contrast
to~\cite{lessard:sander:wanner:17a, rizzi:sander:wanner:24a},
its formulation is based on the Fredholm property of the linearized
operator on symmetry subspaces, rather than using explicit
equivariance operators. While at first glance this approach 
seems to be unnecessarily indirect, it has inherent advantages.

To demonstrate these advantages, consider a partial differential
equation problem on the one-dimensional base domain $\Omega = (0,1)$,
such as for example the diblock copolymer equation. In all of the cases
studied in~\cite{lessard:sander:wanner:17a, rizzi:sander:wanner:24a},
the function~$\ubif$ has reflection symmetries with respect to several
points in the domain. These imply that there exists an $n \in \N$ with
$n \ge 2$ such that~$\ubif : \Omega \to \R$ can be recovered from its
restriction to the subdomain $\Omega_{\sym} = (0,1/n)$ via successive
even reflections across the boundary. The differential operators in the
diblock copolymer model are based on the Laplacian, so they are
automatically equivariant with respect to these reflections. This
means that the symmetry spaces~$X_{\sym}$ and~$Y_{\sym}$ can be
defined as the closed subspaces of~$X$ and~$Y$ that have these
even reflection symmetries around the reflection points
$x = k/n$ for $k = 1,\ldots,n-1$. 

Note that in view of this framework, the symmetry spaces are
isomorphic in a straightforward way to the original function
spaces for the problem, but this time restricted to the symmetry
fundamental domain~$\Omega_{\sym}$. This is due to the fact that
the original problem required homogeneous Neumann boundary conditions,
and on the restricted domain the function~$\ubif$ satisfies the
same constraints because of the even reflections mentioned above.

These observations lead to the following approach for establishing
the Fredholm property of Assumption~\ref{genbif:assump:sym}. In many
applications, the linearization of the underlying nonlinear operator
is a compact perturbation of a leading order differential operator
which is based on the Laplacian. Thus, the term consisting of the
highest-order derivatives is selfadjoint in the appropriate function
spaces, and one can then use the closed range theorem to establish
that it is Fredholm of index zero. The full linearization is a compact
perturbation, since usually one adds only lower-order derivatives, or
products of the argument functions with coefficient functions in the
symmetry subspace. These additional terms preserve the Fredholm
property and the index. Most importantly, the same argument works
for both the spaces defined over the domain~$\Omega$ and over the
domain~$\Omega_{\sym}$, and therefore the Fredholm assumptions are
immediate consequences both for~$L$ and for the
restriction~$L|_{X_{\sym}}$.

We would like to emphasize that this approach is not restricted 
to one-dimensional base domains, and can equally be applied in the
higher-dimensional setting. For example, consider the square domain
$\Omega = (0,1)^2$ and assume that~$\ubif$ is a function that has the
same even reflection symmetries as $\cos(2\pi x) \cos(3\pi y)$.
Then~$\ubif$ is uniquely determined by its restriction to the
fundamental domain $\Omega_{\sym} = (0,1/2) \times (0,1/3)$, and
the above approach can be applied without change.

Finally, note that establishing the index of the involved
Fredholm operators usually can be achieved by relying on
the closed range theorem. For more details, we refer the
reader to~\cite[Section~III.1]{kielhoefer:12a}. As we will see
later in this paper, the closed range theorem is also important
for the discussion of transcritical bifurcation points, since 
this requires an explicit representation of the range of the
linearization~$L$. For this, one can combine orthogonality
with a kernel function of the adjoint operator.
%
%

\section{Results}
\label{sec:examples}

In this section, we apply our results to two examples. These examples 
serve to illustrate how the general approach presented in Section~\ref{sec:genbif} 
can be used to establish the occurrence of symmetry-breaking bifurcations. 
Some numerically obtained bifurcation diagrams of steady states were already 
shown in Section~\ref{sec:intro}, for the Ohta--Kawasaki equation in two 
dimensions (Figure~\ref{fig:diagramOK2D}) and for the SKT system in one
dimension (Figure~\ref{fig:diagramSKT1D}). We also provide here a numerical
bifurcation diagram for the Ohta--Kawasaki equation in one dimension 
(Figure~\ref{fig:diagramOK1D}), which differs from~\cite{rizzi:sander:wanner:24a} 
because we focus on the non-zero mass case, and for the SKT system in two 
dimensions (Figure~\ref{fig:diagramSKT2D}). We now rigorously establish 
some of the bifurcations suggested by these numerical simulations.

\begin{figure}[tb]
      \includegraphics[width=\textwidth]{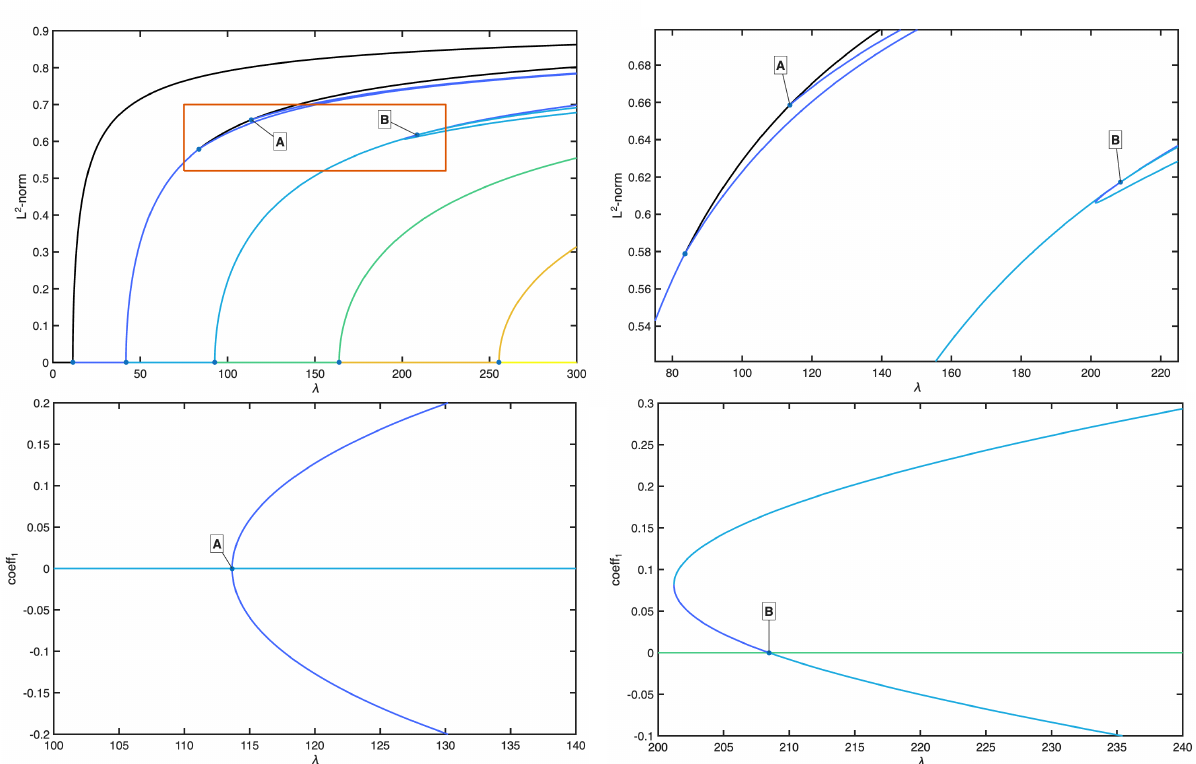}
  \caption{A portion of the bifurcation diagram for the Ohta--Kawasaki equation in one 
  space dimension. The two panels in the top row show a plot of~$\lambda$
  versus the $L^2$-norm of the solution in the large and in close up. 
  The blue dots represent potential pitchfork and transcritical bifurcations which have been detected using the AUTO software package. The ones labeled A and B are validated in 
  Theorem~\ref{thm:resultsOK}.
  The
  bottom row contains closeups of two specific bifurcation points indicated
  in the closeup. In the latter images, the solution measure used for the
  vertical axis is the first cosine Fourier coefficient~$u_1$. Plotted in
  this way, the two bifurcation points have the typical look of a pitchfork
  and transcritical bifurcation, respectively. 
}
  \label{fig:diagramOK1D}
\end{figure}

\begin{figure}[tb]
      \includegraphics[width=\textwidth]{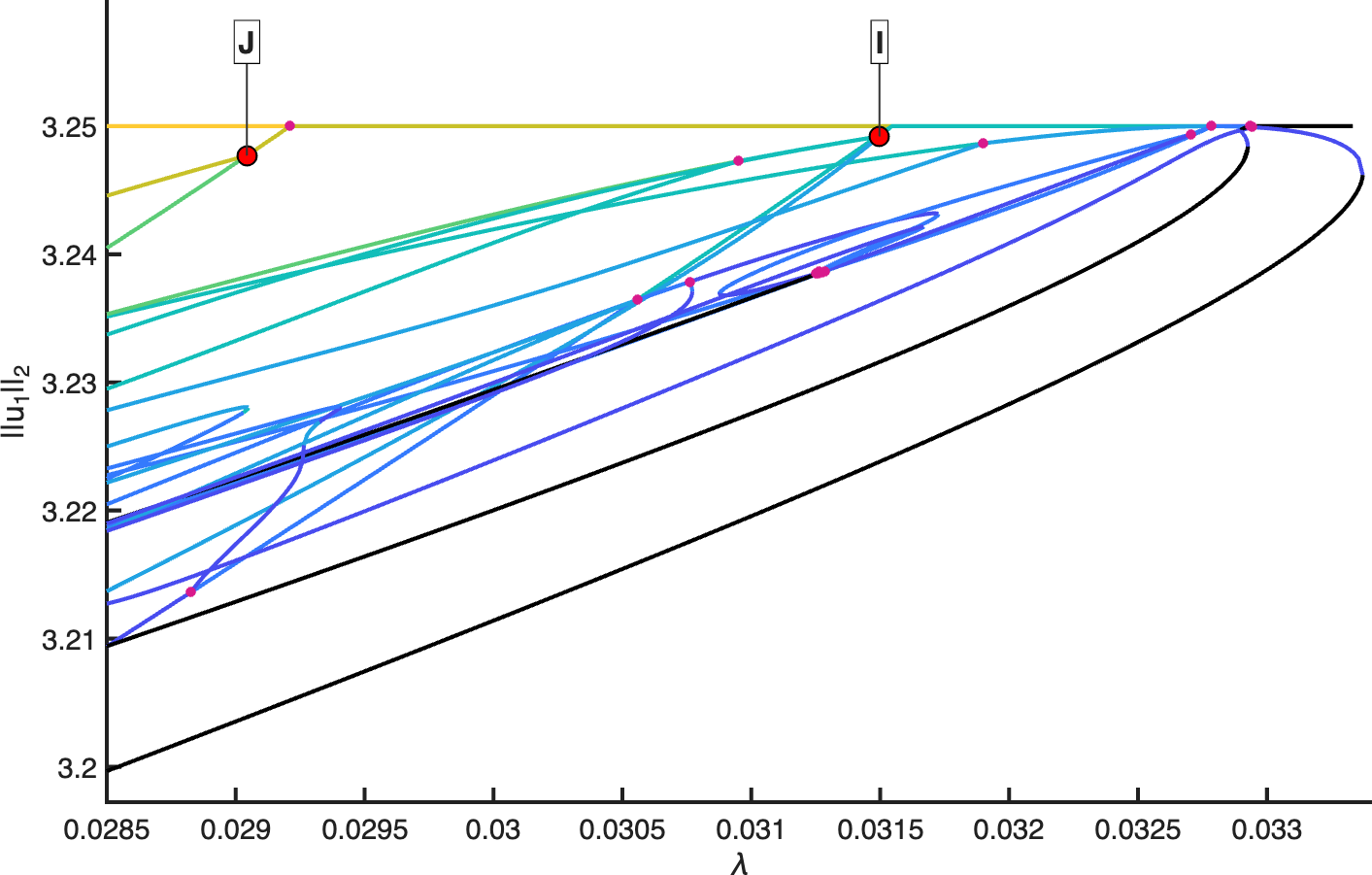}
  \caption{A portion of the bifurcation diagram of steady states for the two-dimensional SKT system~\eqref{parabolicsys} on $\Omega=(0,1)\times (0,4)$, with parameters as in~\eqref{crossdiffparams}. 
  To be consistent with previous work~\cite{KueSor20}, the $L^2$-norm shown is with respect to the first component $u_1$ only.   The dots represent potential pitchfork and transcritical bifurcations which have been detected using the 
  AUTO software package.  The ones labeled I and J are validated in Theorem~\ref{thm:resultsSKT}.
}
  \label{fig:diagramSKT2D}
\end{figure}

\begin{theorem}
\label{thm:resultsOK}
Consider the Ohta--Kawasaki equation~\eqref{dbcp}, with $\mu=-0.1$ 
and $\sigma =1$. Each line of Table~\ref{tab:OK} specifies a 
domain~$\Omega$ and a value~$\blambda$ of the bifurcation parameter~$\lambda$
near which a symmetry-breaking bifurcation of steady-states occurs, of the
type specified in the third column of the table. The corresponding approximate 
steady state~$\bu$ and approximate kernel function~$\bphi$ are represented
in the figure indicated in the table. For each line, the exact 
parameter~$\lambdabif$ at which the bifurcation occurs, and the 
exact steady state~$\ubif$ and kernel function~$\phibif$ satisfy
\begin{align*}
\left\vert \lambdabif-\blambda\right\vert \leq 10^{-8} ,\quad \sup_{x\in\Omega}\left\vert \ubif(x) - \bu(x)\right\vert \leq 10^{-7},\quad \sup_{x\in\Omega}\left\vert \phibif(x) - \bphi(x)\right\vert \leq 10^{-6}.
\end{align*}
\end{theorem}
The approximate steady states~$\bu$ and approximate kernel functions~$\bphi$ 
are given by trigonometric polynomials, whose coefficients are available 
at~\cite{BreSanWan26code}. 

\begin{theorem}
\label{thm:resultsSKT}
Consider the SKT system~\eqref{parabolicsys}, with $d_1=d_2=\lambda$ and 
all other parameters as in~\eqref{crossdiffparams}. Each line of
Table~\ref{tab:SKT} specifies a domain~$\Omega$ and a value~$\blambda$
of the bifurcation parameter~$\lambda$ near which a symmetry-breaking 
bifurcation of steady-states occurs, of the type specified in the third
column of the table. The corresponding approximate steady state~$\bu$ 
and approximate kernel function~$\bphi$ are represented in the figure
indicated in the table. For each line, the exact parameter~$\lambdabif$ 
at which the bifurcation occurs, and the exact steady state~$\ubif$ and
kernel function~$\phibif$ satisfy
\begin{align*}
\left\vert \lambdabif-\blambda\right\vert \leq 10^{-14} ,\quad \max_{i=1,2}\sup_{x\in\Omega}\left\vert (\ubif)_i(x) - \bu_i(x)\right\vert \leq 10^{-11},\quad \max_{i=1,2}\sup_{x\in\Omega}\left\vert (\phibif)_i(x) - \bphi_i(x)\right\vert \leq 10^{-9}.
\end{align*}
\end{theorem}
Here as well, the approximate steady states~$\bu$ and approximate kernel
functions~$\bphi$ are given by trigonometric polynomials, whose 
coefficients are available at~\cite{BreSanWan26code}. 

\begin{table}[!h]
\centering
\begin{tabular}{cccccc}
$\Omega$ & $\blambda$ & Bifurcation & Solution & Location & $\nsym$ \\
\midrule[2pt]
$(0,1)$ & $113.670521515$ & pitchfork & Fig.~\ref{fig:OK1D_1} &
Fig.~\ref{fig:diagramOK1D}, A & $2$\\
\midrule
$(0,1)$ & $208.461521833$ & transcritical & Fig.~\ref{fig:OK1D_2} & Fig.~\ref{fig:diagramOK1D}, B & $3$\\
\midrule
$(0,1)\times (0,0.97)$ & $119.663923396$ & transcritical & Fig.~\ref{fig:OK2D_1} &
Fig.~\ref{fig:diagramOK2D}, C & $(1,3)$\\
\midrule
$(0,1)\times (0,0.97)$ & $106.018283308$ & pitchfork & Fig.~\ref{fig:OK2D_2} &
Fig.~\ref{fig:diagramOK2D}, D & $(2,1)$ \\
\bottomrule[2pt]
\end{tabular}
\caption{Data for some symmetry-breaking bifurcations of steady states in the
Ohta--Kawasaki equation~\eqref{dbcp}, validated in Theorem~\ref{thm:resultsOK}.
In the table, $\Omega$ is the spatial domain on which the equation is considered,
and~$\blambda$ the approximate value at which the bifurcation occurs. The corresponding
approximate steady state and symmetry-breaking eigenfunction are depicted in the figures
indicated in the solution column in the table, and the bifurcation location is
specified in the following column. Finally, $\nsym$ specifies the symmetry 
space~$X_\sym$ used for the proof, as defined in Section~\ref{sec:bif}.}
\label{tab:OK}
\end{table}

\begin{table}[!h]
\centering
\begin{tabular}{cccccc}
$\Omega$ & $\blambda$ & Bifurcation & Solution & Location & $\nsym$ \\
\midrule[2pt]
$(0,1)$ & $0.008966808767687$ & transcritical & Fig.~\ref{fig:SKT1D_1} &
Fig.~\ref{fig:diagramSKT1D}, E & $3$ \\
\midrule
$(0,1)$ & $0.013993670967280$ & pitchfork & Fig.~\ref{fig:SKT1D_2} &
Fig.~\ref{fig:diagramSKT1D}, F & $2$\\
\midrule
$(0,1)$ & $0.003226739574507$ & pitchfork & Fig.~\ref{fig:SKT1D_3} &
Fig.~\ref{fig:diagramSKT1D}, G & $5$\\
\midrule
$(0,1)$ & $0.004452465315501$ & transcritical & Fig.~\ref{fig:SKT1D_4} &
Fig.~\ref{fig:diagramSKT1D}, H & $3$\\
\midrule
$(0,1)\times (0,4)$ & $0.031497237998270$ & transcritical & Fig.~\ref{fig:SKT2D_1} &
Fig.~\ref{fig:diagramSKT2D}, I & $(1,3)$\\
\midrule
$(0,1)\times (0,4)$ & $0.029043593993081$ & pitchfork & Fig.~\ref{fig:SKT2D_2} &
Fig.~\ref{fig:diagramSKT2D}, J & $(1,4)$\\
\bottomrule[2pt]
\end{tabular}
\caption{Data for some symmetry-breaking bifurcations of steady states in the
SKT system~\eqref{parabolicsys}, validated in Theorem~\ref{thm:resultsSKT}.
The spatial domain underlying the equation is denoted by~$\Omega$, 
and~$\blambda$ is the approximate value at which the bifurcation occurs. 
The corresponding approximate steady state and symmetry-breaking 
eigenfunction are depicted in the figures indicated in the solution column
in the table, and the bifurcation location is specified in the following column.
Finally, $\nsym$ specifies the symmetry space~$X_\sym$ used for the proof, as
defined in Section~\ref{sec:bif}.}
\label{tab:SKT}
\end{table}

The proofs of Theorem~\ref{thm:resultsOK} and Theorem~\ref{thm:resultsSKT} consist 
in showing that the extended system~$\cF$ introduced in Section~\ref{sec:genbif} 
does have a non-degenerate zero~$(\lambdabif,\ubif,\phibif)$ near each of the approximate 
zeros~($\blambda,\bu,\bphi)$ described in Tables~\ref{tab:OK} and~\ref{tab:SKT} 
and in the corresponding figures. Theorem~\ref{genbif:thm:extsys}(b) and 
Theorem~\ref{genbif:thm:suffcond} then show that a symmetry-breaking bifurcation
does occur at~$(\lambdabif,\ubif,\phibif)$. The existence of these zeros is obtained using
computer-assisted proofs based on a Newton-Kantorovich argument, which are
described in Section~\ref{sec:CAPs}, and which automatically yield 
that~$D\cF(\lambdabif,\ubif,\phibif)$ is indeed invertible. The type of bifurcation
is then determined by checking whether assumption~(a) or~(b) of 
Theorem~\ref{genbif:thm:suffcond} holds. In the pitchfork case, this
is obtained by studying more carefully the symmetries involved 
(which we note was not required to prove that a symmetry-breaking 
bifurcation occurs), as explained in Section~\ref{sec:pitchfork}. 
In the transcritical case, we conduct a second computer-assisted proof, 
this time to rigorously enclose the left eigenfunction~$\psibif^*$ which 
characterizes the range of~$L$ as in~\eqref{genbif:assump:nec2}, which
then enables us to check that we are in case~(a). This step is 
described in Section~\ref{sec:transcritical}. All the computer-assisted 
parts of the proofs can be reproduced using the Matlab code available
at~\cite{BreSanWan26code}, together with Intlab~\cite{Intlab} for 
interval arithmetic calculations. Each proof takes between~1 and~20 
seconds for the one-dimensional examples, and between~1 and~2 minutes
for the two-dimensional examples, on a laptop with an Intel Core 
Ultra 5 335 processor and 32GB of RAM, but we note that the code 
could be heavily optimized, especially in the two-dimensional case, 
for instance by taking advantage of the symmetries as 
in~\cite{vandenberg:williams:21a} or~\cite{RPjl}. We also emphasize
that, for most examples, the obtained error bounds are actually 
smaller than those reported in Theorem~\ref{thm:resultsOK} and 
Theorem~\ref{thm:resultsSKT}, where we only gave the worst case. 
\begin{figure}[!h]
    \centering
    \begin{subfigure}[b]{0.45\textwidth}
        \centering
        \includegraphics[width=\textwidth]{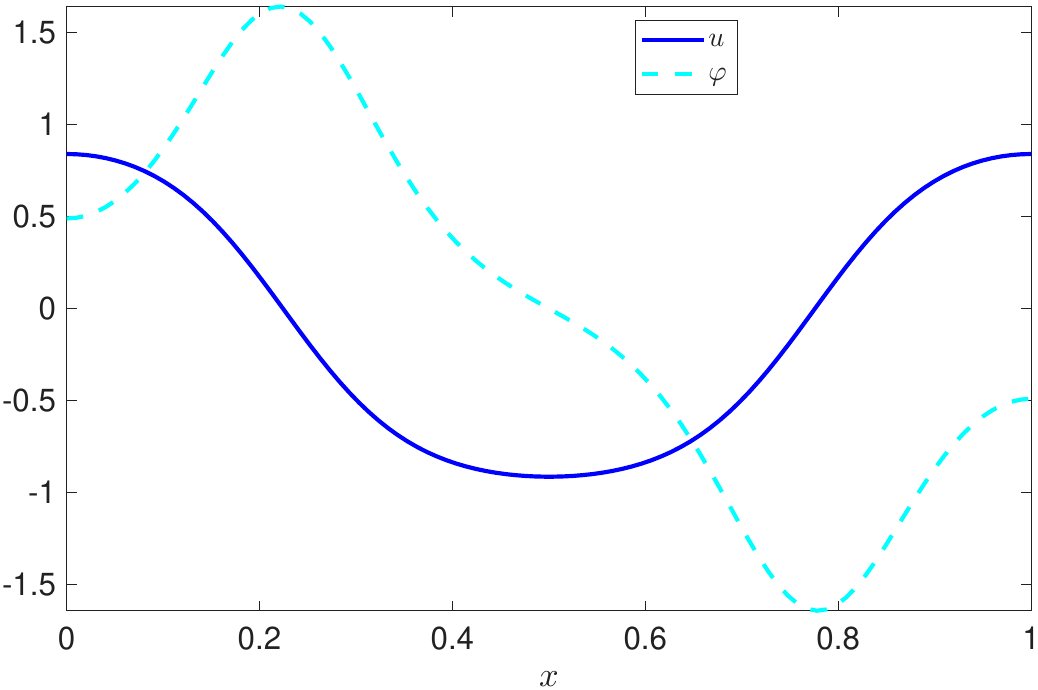}
        \caption{Label $A$ on Line 1 of Table~\ref{tab:OK}.}
        \label{fig:OK1D_1}
    \end{subfigure}
    \hfill
    \begin{subfigure}[b]{0.45\textwidth}
        \centering
        \includegraphics[width=\textwidth]{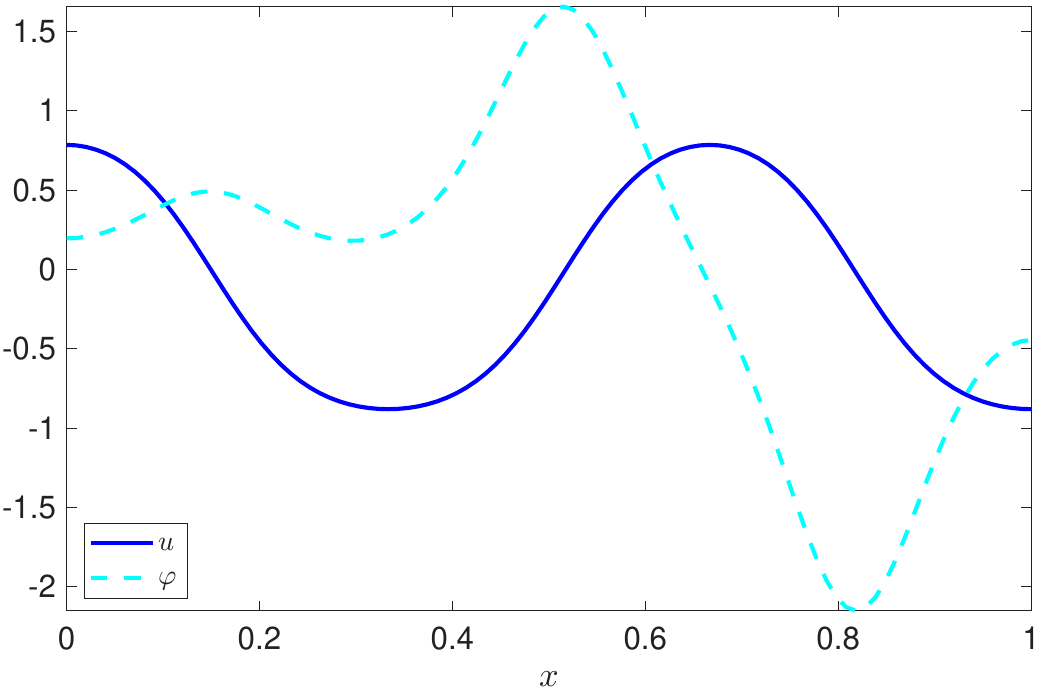}
        \caption{Label $B$ on Line 2 of Table~\ref{tab:OK}.}
        \label{fig:OK1D_2}
    \end{subfigure}
    \caption{Two different approximate steady states~$\bu$ and kernel 
    functions~$\bphi$ of the Ohta--Kawasaki equation~\eqref{dbcp} near bifurcation points.}
    \label{fig:OK1D}
\end{figure}

\begin{figure}[!h]
\centering
      \includegraphics[width=\textwidth]{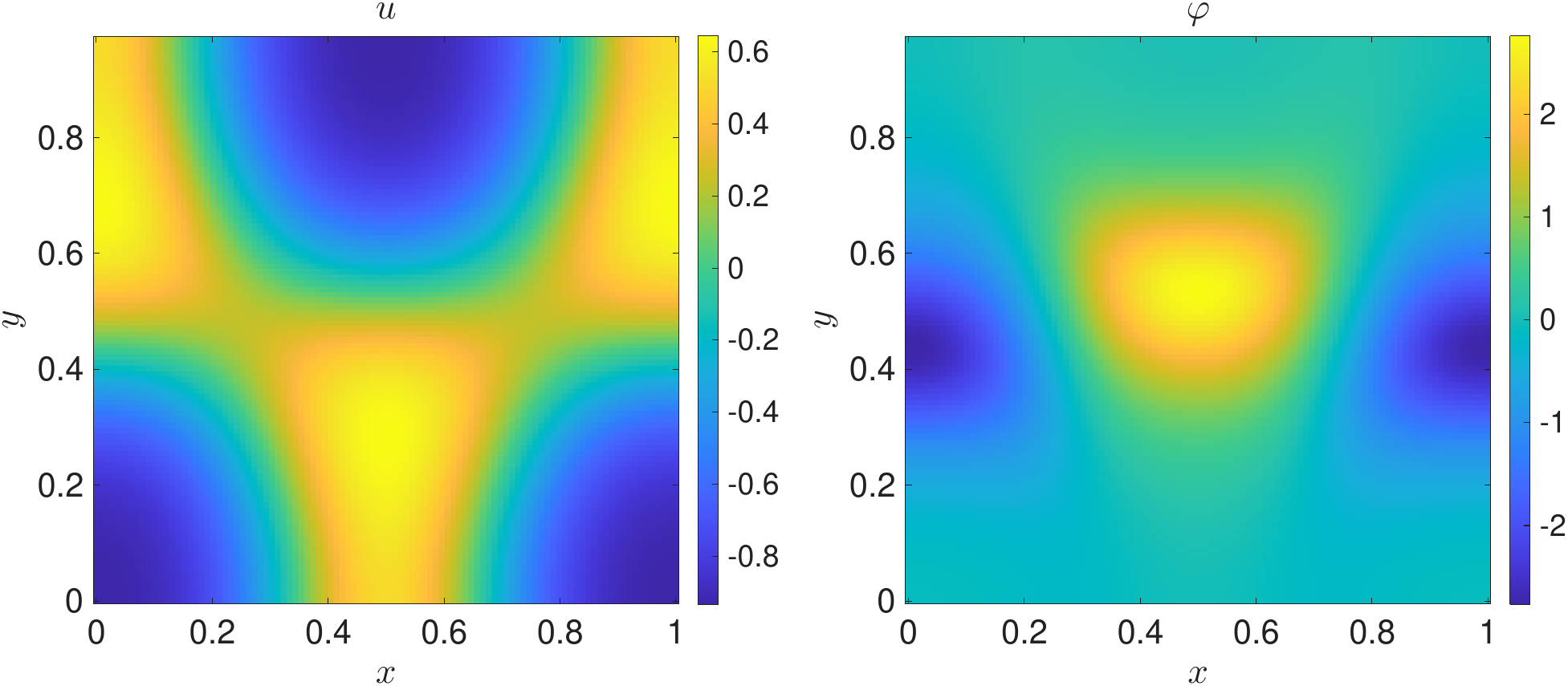}
  \caption{The approximate steady state~$\bu$ and kernel function~$\bphi$ of the
  Ohta--Kawasaki equation~\eqref{dbcp} corresponding to label~$D$ on Line~4 of Table~\ref{tab:OK}.}
  \label{fig:OK2D_2}
\end{figure}

\begin{figure}[!h]
    \centering
    \begin{subfigure}[b]{0.45\textwidth}
        \centering
        \includegraphics[width=\textwidth]{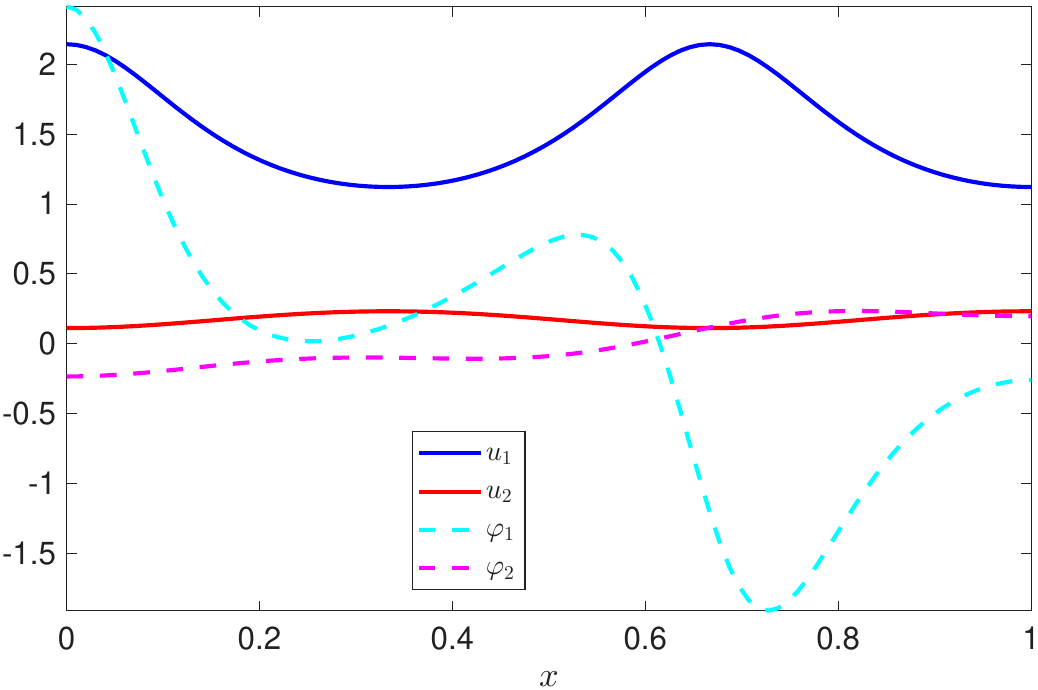}
        \caption{Label $E$ on Line 1 of Table~\ref{tab:SKT}.}
        \label{fig:SKT1D_1}
    \end{subfigure}
    \hfill
    \begin{subfigure}[b]{0.45\textwidth}
        \centering
        \includegraphics[width=\textwidth]{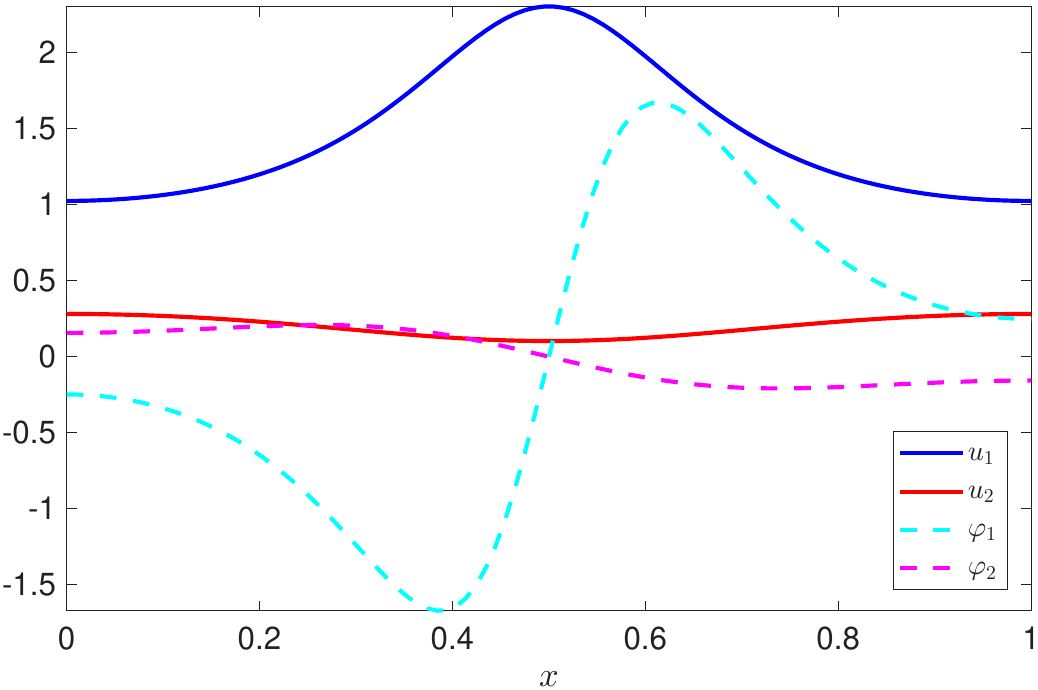}
        \caption{Label $F$ on Line 2 of Table~\ref{tab:SKT}.}
        \label{fig:SKT1D_2}
    \end{subfigure} \\[3ex]
    \begin{subfigure}[b]{0.45\textwidth}
        \centering
        \includegraphics[width=\textwidth]{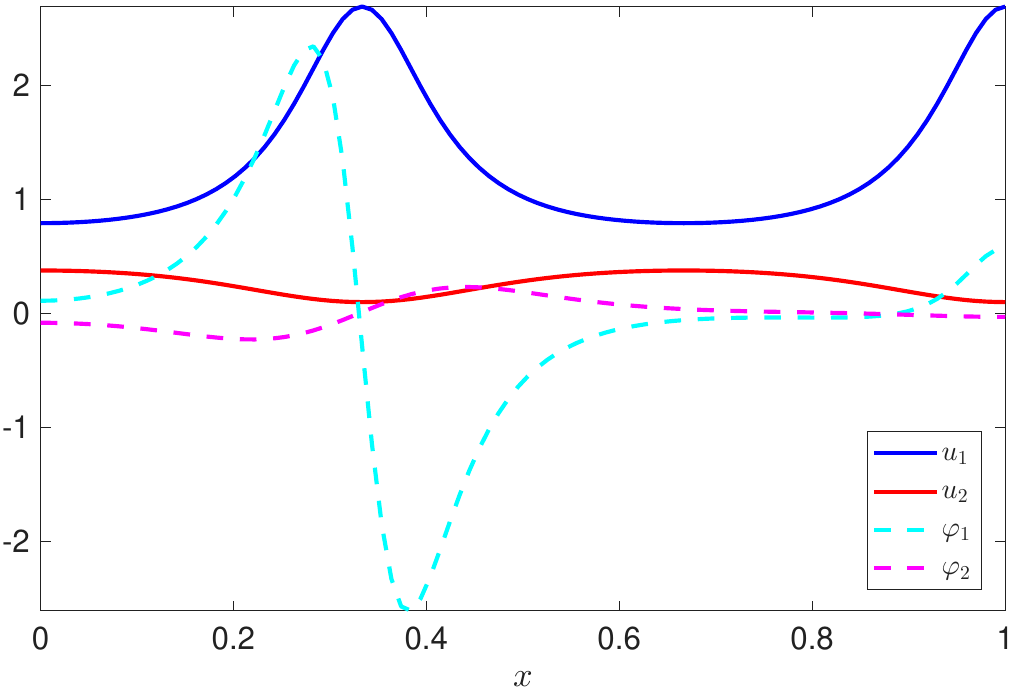}
        \caption{Label $H$ on Line 4 of Table~\ref{tab:SKT}.}
        \label{fig:SKT1D_4}
    \end{subfigure}
    \caption{Three different approximate steady states~$\bu$ and 
    kernel functions~$\bphi$ of the SKT system~\eqref{parabolicsys} near bifurcation points.}
    \label{fig:SKT1D}
\end{figure}

\begin{figure}[!h]
\centering
      \includegraphics[width=\textwidth]{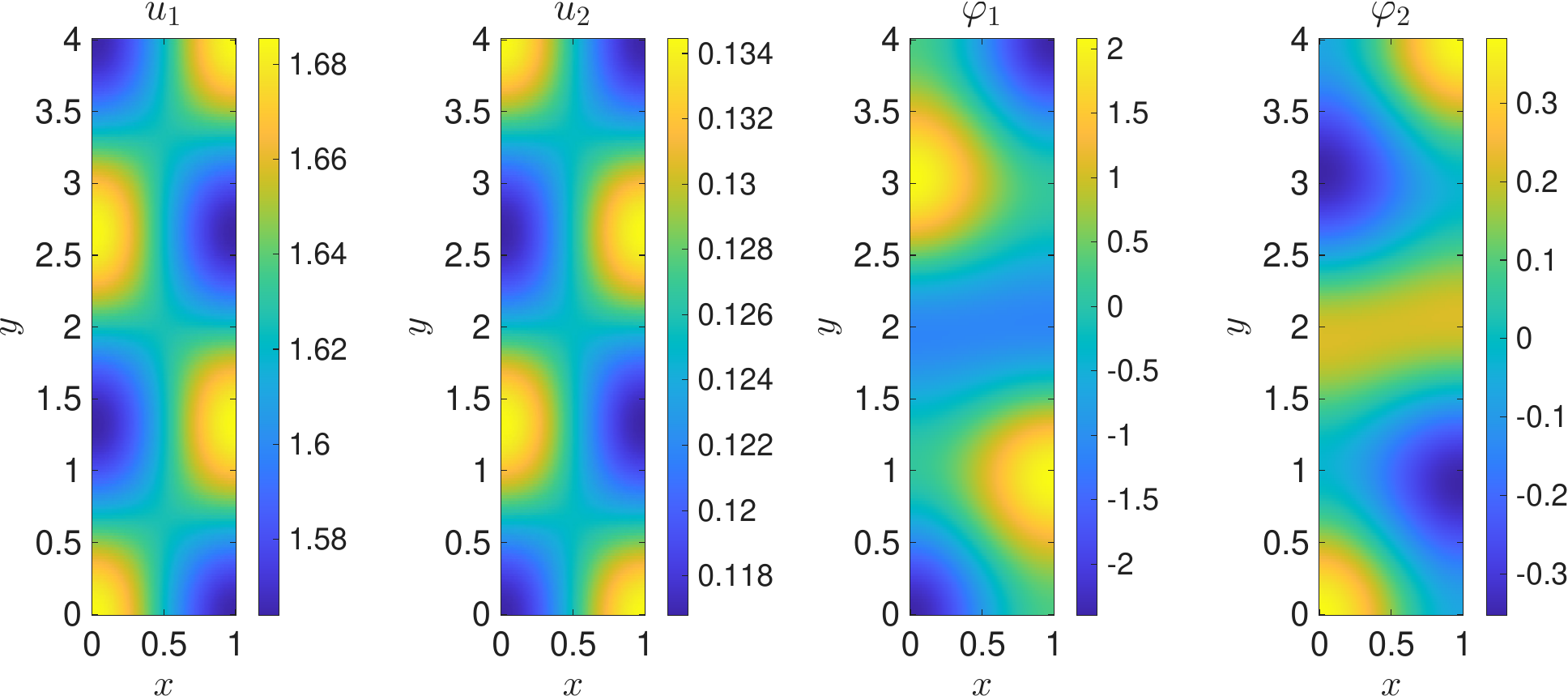}
  \caption{The approximate steady state~$\bu$ and kernel function~$\bphi$ of the
  SKT system~\eqref{parabolicsys} corresponding to label~$I$ on Line~5 of Table~\ref{tab:SKT}.}
  \label{fig:SKT2D_1}
\end{figure}

\begin{figure}[!h]
\centering
      \includegraphics[width=\textwidth]{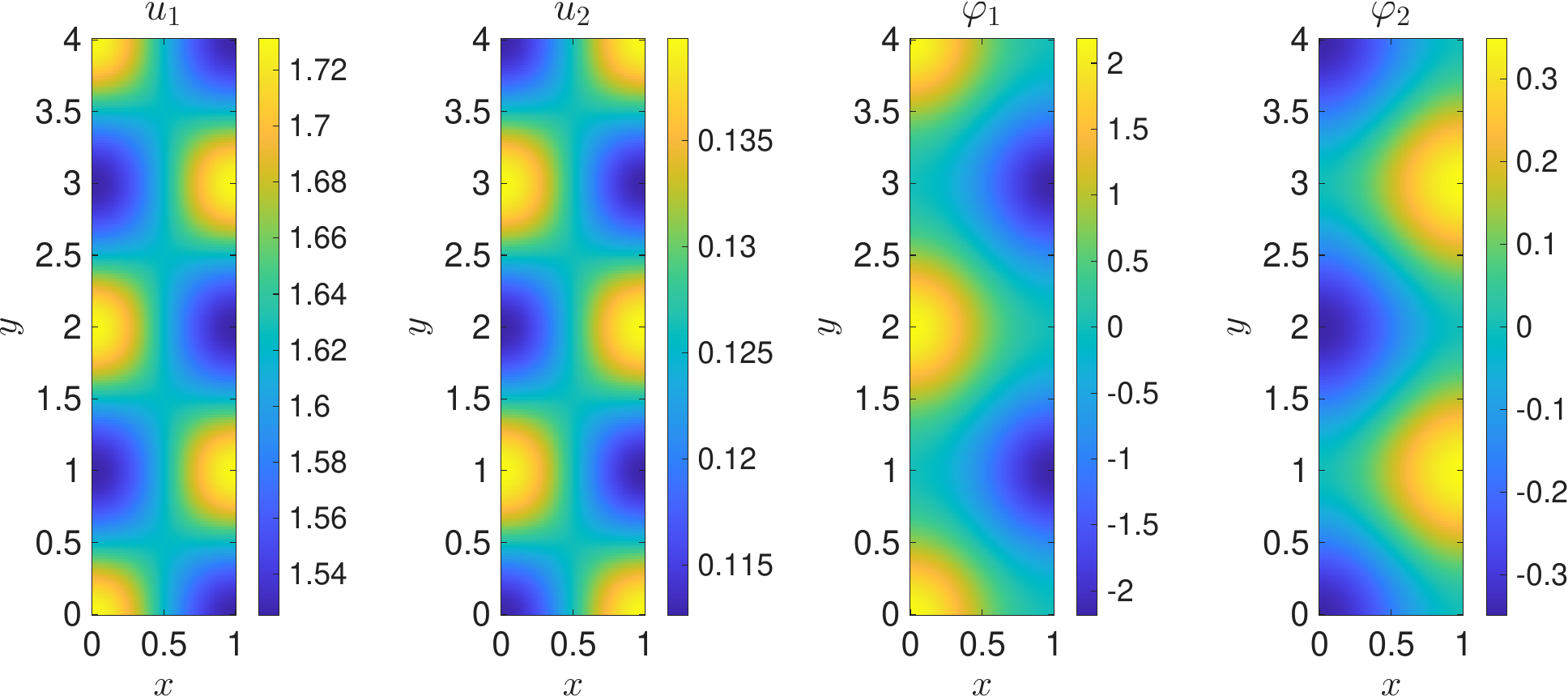}
  \caption{The approximate steady state~$\bu$ and kernel function~$\bphi$ of the
  SKT system~\eqref{parabolicsys} corresponding to label~$J$ on Line~6 of Table~\ref{tab:SKT}.}
  \label{fig:SKT2D_2}
\end{figure}

\section{Computer-assisted proof techniques}
\label{sec:CAPs}

One of the main features of Theorem~\ref{genbif:thm:extsys}(b) is that it
reformulates the study of symmetry-breaking bifurcations into a problem which
can be settled using computer-assisted proofs. In this section, we showcase how
to apply the result in concrete examples. Our procedure is as follows. We first compute numerically what we hope to be an accurate approximation of a zero of the extended system $\cF$ introduced in~\eqref{genbif:eqn:extsys1}, and then prove a posteriori the existence of a nearby zero of $\cF$, using a fixed point argument which also provides us with guaranteed error bounds between the approximate zero and the exact one. The invertibility of the Fréchet derivative required in Theorem~\ref{genbif:thm:extsys}(b) is obtained automatically with our approach.
This type of strategy goes back at least to the proof of the Feigenbaum conjecture~\cite{Lan82}, and is by now common in computer-assisted proofs for differential equations and other dynamical systems~\cite{Gom19,NakPluWat19,plum:95a,plum:96a,BerLes15}. 

We recall some of the basic mathematical tools underlying such computer-assisted
proofs of zeros of systems in Section~\ref{sec:NewtonKantorovich}. We then
consider two specific examples, the  Ohta--Kawasaki model \eqref{dbcp} in
Section~\ref{sec:proofOK} and the SKT system \eqref{parabolicsys} in
Section~\ref{sec:proofSKT}. A strength of our approach is that it does not
require very precise information on the involved underlying symmetries. However,
one byproduct of this is that Theorem~\ref{genbif:thm:extsys}(b) by itself does
not allow us to determine whether the established bifurcation is of transcritical
or pitchfork type. This information can often be recovered after the bifurcation 
has been validated, by checking whether condition~(a) or~(b) of 
Theorem~\ref{genbif:thm:suffcond} holds. If the bifurcation is a pitchfork, 
this requires a deeper analysis of the symmetries (to be done by hand), presented 
in Section~\ref{sec:pitchfork}. On the other hand, if the bifurcation is transcritical, a 
subsequent computer-assisted proof can be conducted to check condition~(a)
of Theorem~\ref{genbif:thm:suffcond}, and this is presented in Section~\ref{sec:transcritical}.

\subsection{The Newton-Kantorovich approach}
\label{sec:NewtonKantorovich}

In this section, we provide sufficient conditions, that can be checked in practice,
and which ensure that a zero of $\cF$ exists near an element $\bx\in\cX$. A common strategy to accomplish this is to consider a Newton-like fixed point operator 
\begin{equation}
\label{eq:T}
T:x\mapsto x - A\cF(x), 
\end{equation} 
where $A:\cY\to\cX$ is an injective linear map approximating $D\cF(\bx)^{-1}$. 

The following statement, based on the Banach fixed point theorem, provides explicit conditions under which $T$ has a unique fixed point near $\bx$. Many similar versions of this theorem have been used in the last decades for computer-assisted proofs. This specific instance is the one introduced in~\cite[Theorem 2.12]{BerBreShe24}, and is well adapted to situations where the space $\cX$ has several components, as is the case when we deal with the extended system~\eqref{genbif:eqn:extsys1}. Its main feature is that the norm on the product space $\cX$ is not chosen a priori, and that we do not try to choose the neighborhood of $\bx$ on which we apply Banach fixed point theorem as a ball for this norm. This provides some extra flexibility when trying to prove that $T$ is a contraction, see~\cite{BerBreShe24} for a more detailed discussion.

\begin{theorem}
\label{thm:NewtonKantorovich}
Let $M\in\N_{\geq 1}$, $(\cX^m,\|\cdot\|_{\cX^m})_{m=1}^M$ be Banach spaces,
$\cX=\Pi_{m=1}^M \cX^m$ the product space, and $\pi^m:\cX\to \cX^m$ the
projections onto the components. Let $\rstar = (\rstar_m)_{m=1}^M \in \R_{>0}^M$
and~$\bar{x} \in \cX$. For any $r\in \R_{>0}^M$, we define 
\begin{equation*}
\Box(\bar{x},r) = \{ x \in \cX :  \left\Vert\pi^{m}(x-\bar{x}) \right\Vert_{\cX^m} \leq r_m \text{ for } 1\leq m\leq M\}.
\end{equation*} 
Assume that $T \in C^1(\Box(\bar{x},r),\cX)$. For $r,\rstar \in \R_{>0}^M$ we
say that $r \leq \rstar$ if we have $r_m \leq \rstar_m$ for all $1\leq m\leq M$. Finally, we denote partial Fr\'echet derivatives by $D_i$.
 
Finally, suppose that $Y_m \geq 0$, $Z^1_{m,i} \geq 0$, and $Z^2_{m,i,j} \geq 0$ for $1 \leq  i,j,m \leq M$ satisfy
\begin{alignat}{2}
	\left\Vert\pi_m (T(\bar{x})-\bar{x}) \right\Vert_{\cX^m} &\leq Y_m, \label{eq:def_Y}\\
	\left\Vert\pi_m D_i T(\bar{x}) \right\Vert_{B(\cX^i,\cX^m)} &\leq Z^1_{m,i}, \label{eq:def_Z}\\
	\left\Vert\pi_m (D_i T(x)- D_i T(\bar{x})) \right\Vert_{B(\cX^i,\cX^m)} &\leq \sum_{j=1}^M Z^2_{m,i,j} \left\Vert \pi_j(x-\bar{x}) \right\Vert_{\cX^j}
	&\quad&\text{for all } x\in \Box(\bar{x},\rstar). \label{eq:def_W}
\end{alignat}
If $r,\eta \in \R_{>0}^M$ with $r \leq \rstar$ satisfy
\begin{alignat}{1}
	Y_m + \sum_{i=1}^M Z^1_{m,i}r_i  
	+ \frac{1}{2}\sum_{i,j=1}^M Z^2_{m,i,j}r_i r_j  & \leq r_m 
	\label{eq:inequalities1} \\
	 \sum_{i=1}^M Z^1_{m,i}\eta_i
	 + \sum_{i,j=1}^M Z^2_{m,i,j}\eta_i r_j  & < \eta_m
	 \label{eq:inequalities2} 
\end{alignat}
for $1\leq m\leq M$,
then $T$ has a unique fixed point in $\Box(\bar{x},r)$.
\end{theorem}

\begin{remark}
If assumption~\eqref{eq:inequalities2} is satisfied, we can 
consider on the product space~$\cX$ the norm given by $\Vert x \Vert_\cX = \max_{1\leq m\leq M} \frac{1}{\eta_m}\Vert \pi_m x\Vert_{\cX^m}$, for which we get, for all $x$ in $\Box(\bar{x},r)$,
\begin{align*}
\Vert DT(x) \Vert_{\cX} \leq \max_{1\leq m\leq M} \frac{1}{\eta_m} \left(\sum_{i=1}^M Z^1_{m,i}\eta_i
	 + \sum_{i,j=1}^M Z^2_{m,i,j}\eta_i r_j \right) < 1.
\end{align*}
Therefore, recalling that $DT(x) = I - AD\cF(x)$, $AD\cF(x)$ is invertible. If $D\cF(x)$ is a Fredholm operator of index $0$, then $D\cF(x)$ itself is also invertible. 

What we have just shown is that, if the existence of a zero $x$ of $\cF$ is obtained via Theorem~\ref{thm:NewtonKantorovich}, then the invertibility of $D\cF(x)$ required in Theorem~\ref{genbif:thm:extsys}(b) automatically holds.
\end{remark}

We now consider two examples. For each of them, we first introduce suitable
Banach spaces $\cX$ and $\cY$ in order to study the $\cF=0$ problem given by the
extended system~\eqref{genbif:eqn:extsys1}, and construct a well chosen
approximate inverse $A$ to be used in~\eqref{eq:T}. After that, we derive computable estimates $Y_m$, $Z^1_{m,i}$ and $Z^2_{m,i,j}$ satisfying assumptions~\eqref{eq:def_Y}-\eqref{eq:def_W} of Theorem~\ref{thm:NewtonKantorovich}. Finally, for each of the approximate bifurcation points obtained numerically and described in Table~\ref{tab:OK} and Table~\ref{tab:SKT}, the code available at~\cite{BreSanWan26code} can be used to evaluate these estimates and to check that conditions~\eqref{eq:inequalities1}-\eqref{eq:inequalities2} hold, which proves that a symmetry-breaking bifurcation does occur nearby, according to Theorem~\ref{thm:NewtonKantorovich} and Theorem~\ref{genbif:thm:extsys}.

We emphasize that, although we focus on two specific examples below, the entire
procedure is rather general, and well understood for a wide class of elliptic
problems. The fact that the extended system has several components makes
everything more cumbersome to write down, but all the ideas are the same as the
ones used for a computer-assisted proof of a regular equilibrium. For this reason, we merely provide the central estimates in the main body of the text, and defer most of the technical details to the Appendix. We also refer readers who are unfamiliar with such computer-assisted proofs to~\cite[Section 1.3.1]{Bre25} for a basic example, where the same ideas are presented in a simpler context.

\subsection{Application to the Ohta--Kawasaki model}
\label{sec:proofOK}

In this section, we focus on the Ohta--Kawasaki model~\eqref{dbcp}, which is used to model the evolution of diblock copolymer microstructures~\cite{OhtKaw86}.  The equilibria of this equation are solutions of the elliptic system given by  
\begin{align}
\label{eq:OK}
 \begin{cases}
			F(\lambda,u)= 
			-\Delta \left( \Delta u + \lambda \left( u - u^3 \right) \right) -
  \lambda \sigma (u - \mu) = 0  & \text{ in }\Omega, \\
			\quad\;\;\partial_n u = \partial_n \Delta u = 0 & \text{ on }\partial\Omega,
	\end{cases}
\end{align}
where $\partial_n$ denotes the normal derivative, and we assume that~$\mu$
and~$\sigma$ are fixed constants. Equilibrium solutions and continuation of
solutions with respect to parameters have been studied for this and closely
related systems using computer-assisted proofs~\cite{cai:watanabe:19a, 
rizzi:etal:22a,sander:wanner:21a, vandenberg:williams:17a, vandenberg:williams:19a,
vandenberg:williams:21a, wanner:16a, wanner:17a}, and in
particular~\cite{lessard:sander:wanner:17a,rizzi:sander:wanner:24a} with
respect to pitchfork bifurcations. However, the methods of the last two papers
were restricted in the first case to $\Z_2$-symmetries, and in the second one
to $\Z_n$-symmetries. In each case, detailed knowledge of the symmetry group~$X_s$
was required. Our new theorem expands the cases that are possible to treat, and
the current paper is the first one to present transcritical bifurcation points,
to consider higher-dimensional domains, and to consider the case of non-zero mass
parameter $\mu$. 

\subsubsection{Sequence space preliminaries}
\label{sec:sequence_space}

In this work, we consider rectangular domains $\Omega \subset \R^d$, given by $\Omega = \prod_{i=1}^d (0,L_i)$, and study symmetry-breaking bifurcations of steady states of~\eqref{eq:OK}. While we only provide examples in dimensions one and two, the theory applies for general $d$.
The shape of the domain together with the Neumann boundary conditions make it natural to look for stationary solutions of~\eqref{eq:OK} as cosine series. 
For any $s\in\R$, we thus consider the following weighted $\ell^1$ spaces:
\begin{equation*}
\ell^1_{s} := \left\{ u=(u_n)_{n\in\N^d}, \ \left\Vert u\right\Vert_{\ell^1_s} := \sum_{n\in\N^d} \vert u_n\vert \xi^{(s)}_n <\infty \right\},
\end{equation*}
where
\begin{equation}
\xi^{(0)}_n := \sqrt{(2-\delta_{n_1,0}) \ldots (2-\delta_{n_d,0})}, \qquad n\in\N^d, 
\end{equation}
and
\begin{equation}
\xi^{(s)}_n := \xi^{(0)}_n (1+ \vert n\vert_2)^s, \qquad n\in\N^d,
\end{equation}
and with $\vert n \vert_2 = \sqrt{n_1^2 + \dots + n_d^2}$. 
To any sequence $u\in\ell^1_s$, $s\geq 0$, we associate the continuous function
\begin{align}
\label{eq:u_function}
u: x\mapsto \sum_{n\in\N^d} u_n \xi^{(0)}_n \prod_{i=1}^d \cos\left(\frac{n_i x_i \pi} {L_i}\right).
\end{align}
\begin{remark}
In order to not overburden the notation, we use the same symbol to denote a
sequence $u$ in $\ell^1_{s}$ and the function $u$ given
by~\eqref{eq:u_function}, but it should be clear from context which object we
are considering in any given instance.
\end{remark}
The weights $\xi^{(0)}_n$ are incorporated into~\eqref{eq:u_function} in order
to obtain a weight-free formula for the scalar product
\begin{align}
\label{eq:scalar_product}
\langle u,v \rangle_{L^2(\Omega)} := \frac{1}{\vert\Omega\vert} \int_\Omega u(x)v(x)\d x = \sum_{n\in\N^d} u_n v_n.
\end{align}
They are also included in the $\ell^1_s$ norms so that, for all $s\geq 0$, one
has the estimate
\begin{align*}
\sup_{x\in\Omega} \vert u(x)\vert \leq \Vert u\Vert_{\ell^1_s}.
\end{align*}
Another important property of the space $\ell^1_s$ is that it is a Banach algebra for all $s\geq 0$, for the discrete convolution product corresponding to pointwise products of functions. That is, for $s\geq 0$ and $u,v\in\ell^1_s$, we recall that the product $uv\in\ell^1_s$, where
\begin{align}
\label{eq:prod_on_ell1}
(uv)_n = \xi^{(0)}_n \sum_{k\in\Z^d} \dfrac{u_{\vert k\vert} v_{\vert n-k\vert} }{\xi^{(0)}_{\vert k\vert} \xi^{(0)}_{\vert n-k\vert}},\quad n\in\N^d,
\end{align}
with $\vert k\vert = (\vert k_1\vert, \ldots, \vert k_d\vert)$.
\begin{remark}
The usage of weighted~$\ell^1$ spaces is common in many other works, and in particular in computer-assisted proofs using cosine series. Within this literature, it is more usual to consider
\begin{align*}
u: x\mapsto \sum_{n\in\N^d} u_n \left(\xi^{(0)}_n\right)^2 \prod_{i=1}^d \cos\left(\frac{n_i x_i \pi}{L_i}\right)
\end{align*}
instead of~\eqref{eq:u_function}, which has the advantage of yielding a formula
without weights for the corresponding discrete convolution product. However, in
this work we also make use of the Euclidean structure of~$L^2$, and prioritized having a clean formula without weights for the scalar product~\eqref{eq:scalar_product} rather than for the discrete convolution, which is why we selected the convention~\eqref{eq:u_function}.
\end{remark}

As a final preliminary step, for any $N\in\N^d$ we introduce the projection operators $\Pi^{\leq N}$ and $\Pi^{>N}$ on $\ell^1_s$, given by
\begin{align*}
\left(\Pi^{\leq N} u \right)_n := 
\begin{cases}
u_n \quad &\text{if } n\leq N,\\
0 \quad &\text{otherwise,}
\end{cases}
\qquad \text{and} \qquad 
\left(\Pi^{> N} u \right)_n := 
\begin{cases}
0 \quad &\text{if } n\leq N,\\
u_n \quad &\text{otherwise.}
\end{cases}
\end{align*}
Here, $n\leq N$ means that $n_j\leq N_j$ for all $j\in\{1,\ldots,d\}$.

\subsubsection{Zero finding problem and approximate inverse}
\label{sec:F_A_OK}

In order to study steady states of~\eqref{eq:OK}, and symmetry-breaking bifurcations with respect to the parameter $\lambda$, we consider $X = \ell^1_{s}$ for some $s\geq 0$, $Y = \ell^1_{s-4}$, and $F:\R\times X\to Y$ given by~\eqref{eq:OK}
where the product on $\ell^1_s$ is given by~\eqref{eq:prod_on_ell1}, and for any $u$ in $\ell^1_s$, 
\begin{align*}
\left(\Delta u\right)_n = -\sum_{i=1}^d \left(\frac{n_i \pi}{L_i}\right)^2 u_n,\quad n\in\N^d.
\end{align*}
Following Section~\ref{subsec:genbif:ext}, we consider $\cX = \R\times X_{\sym}\times X$, $\cY = \R\times Y_{\sym}\times Y$, and $\cF:\cX\to\cY$ the extended system given by~\eqref{genbif:eqn:extsys1}. Regarding the linear form $l\in X^*$ involved in $\cF$, we pick
\begin{align*}
l: \phi \mapsto \sum_{n\in \N^d} \bphi_n \phi_n,
\end{align*}
where~$\bphi$ is some fixed element in~$X$, typically an approximation of the bifurcating eigenfunction. In particular, it will
be convenient to assume that we have $\bphi\in\Pi^{\leq N} X$ for some $N\in\N^d$. 
The precise definition of the symmetry subspaces~$X_{\sym}$ and~$Y_{\sym}$ in~$\cX$ 
and~$\cY$ is not important for the moment, but will be given in 
Section~\ref{sec:pitchfork}. At the moment we only assume that we indeed have 
the inclusion $F(\R,X_\sym)\subset Y_\sym$, and that the restriction
$D_u F(\lambda,u)|_{X_{\sym}} \in \cL(X_{\sym},Y_{\sym})$ is a Fredholm operator
of index zero. Notice that the full operator $D_u F(\lambda,u) \in \cL(X,Y)$ is a 
Fredholm operator of index zero, since it is a relatively compact perturbation 
of~$-\Delta^2$, and in each of our examples the same argument readily applies 
to the restriction~$D_u F(\lambda,u)|_{X_{\sym}}$.
 
With the notations of Theorem~\ref{thm:NewtonKantorovich}, the spaces are 
$\cX^{1} = \R$, $\cX^2 = X_{\sym}$ and $\cX^3 = X$. Moreover, denoting 
$x=(\lambda,u,\phi)$ a generic element of $\cX$, the partial derivatives 
correspond to $D_1 = D_\lambda$, $D_2= D_u$ and $D_3 = D_\phi$. 
Finally, the projections defined on the~$\ell^1_s$ spaces extend in a natural
way to the product spaces~$\cX$ or~$\cY$. For this, given a triple 
$(\lambda,u,\phi)\in\cX$, one defines both $\Pi^{\leq N} (\lambda,u,\phi) :=
(\lambda,\Pi^{\leq N}u,\Pi^{\leq N}\phi)$ and $\Pi^{> N} (\lambda,u,\phi) := 
(0,\Pi^{> N}u,\Pi^{> N}\phi)$. In particular, we take the convention that 
$\Pi^{\leq N}\R := \R$ and $\Pi^{> N}\R := \{0\}$.

Suppose now that we have an approximate zero $\bx = (\blambda,\bu,\bphi)\in \cX$
of~$\cF$. 
Our goal is to use Theorem~\ref{thm:NewtonKantorovich} with the quasi-Newton
operator~$T$ defined in~\eqref{eq:T}, and therefore we first need to construct a 
suitable approximate inverse~$A$ of~$D\cF(\bx)$. This construction is standard 
for such semilinear problems in the computer-assisted proof literature, and is 
a straightforward generalization of the one that could have been used to study 
a regular steady state of~\eqref{eq:OK}.
 
In order to understand the rationale behind the construction of~$A$, it can be 
helpful to look at a block-matrix representation of~$D\cF(\bx)$:

\begin{equation}
\label{eq:DFblocks}
D\cF(\bx) = 
\left(
\begin{array}{ccc}
0 & 0 &\ell\\
D_{\lambda}F(\blambda,\bu) &D_uF(\blambda,\bu) &0 \\
D_{\lambda u}F(\blambda,\bu) \bphi & D_{uu}F(\blambda,\bu) \bphi & D_u F(\blambda,\bu)
\end{array}
\right).
\end{equation}
Since~$D_uF(\lambda,u)$ is given as the sum of~$-\Delta^{2}$ and relatively 
compact terms, the leading order terms in~$D\cF(\bx)$ are of the form
\begin{equation*}
\left(
\begin{array}{ccc}
0 &
0 &
0 \\

0 &
-\Delta^2 &
0 \\ 

0  &
0 &
-\Delta^2
\end{array}
\right).
\end{equation*}
At least for the high frequencies, the inverse of the operator~$D\cF(\bx)$ 
should therefore be well approximated by
\begin{equation*}
\left(
\begin{array}{ccc}
0 &
0 &
0 \\

0 &
-\Delta^{-2} &
0 \\ 

0  &
0 &
-\Delta^{-2}
\end{array}
\right).
\end{equation*}
Regarding the low frequencies, we consider an operator~$A^{\leq N}$ between the
finite-dimensional spaces~$\Pi^{\leq N}\cY$ and~$\Pi^{\leq N}\cX$, obtained by 
numerically inverting $\Pi^{\leq N} D\cF(\blambda,\bu,\bphi)\Pi^{\leq N}$, and 
then define the full operator~$A$ as follows
\begin{align*}
A := A^{\leq N} - \Delta^{-2}\Pi^{>N},
\end{align*}
with some slight abuse of notation. More precisely, given $y=(\lambda,u,\phi)\in\cY$,
we define
\begin{align*}
\begin{cases}
A\Pi^{\leq N} y := A^{\leq N}\Pi^{\leq N}y \\
A\Pi^{> N} y := (0,-\Delta^{-2}\Pi^{>N} u, -\Delta^{-2}\Pi^{>N} \phi),
\end{cases}
\end{align*}
where, for any $u$ in $\ell^1_s$, 
\begin{align*}
\left(\Delta^{-1} u\right)_n := 
\begin{cases}
0 &\quad  n=0,\\
-\left(\displaystyle \sum_{i=1}^d \left(\frac{n_i \pi}{L_i}\right)^2\right)^{-1} u_n &\quad n\in\N^d\setminus\{0\}.
\end{cases}
\end{align*}
Using the expression~$\Delta^{-1}$ is a slight abuse of notation, justified by the fact
that~$\Delta^{-1}$ (or in this case its square~$\Delta^{-2} = (\Delta^{-1})^2$)
only appears composed with the projection~$\Pi^{>N}$ on high frequencies.

In the sequel, we assume that the approximate zero~$(\blambda,\bu,\bphi)$ of~$\cF$ 
belongs to~$\Pi^{\leq N}\cX$, for the same truncation parameter~$N$ as the one used 
for constructing~$A$. This choice is made with the purpose of simplifying the 
exposition of the estimates to come, but we emphasize that selecting different 
truncation parameters for the approximate solution and for the approximate inverse
would make the proof more efficient in practice.

\subsubsection{Bounds $Y$, $Z^1$ and $Z^2$ for Theorem~\ref{thm:NewtonKantorovich}}
\label{sec:boundsOK}

We now provide computable estimates~$Y_m$, $Z^1_{m,i}$ and~$Z^2_{m,i,j}$ satisfying 
assumptions~\eqref{eq:def_Y}-\eqref{eq:def_W} of Theorem~\ref{thm:NewtonKantorovich}, 
for the operator~$T$ given by~\eqref{eq:T}, with the spaces~$\cX$ and~$\cY$, and
both the map~$\cF$ and the linear operator~$A$ introduced in Section~\ref{sec:F_A_OK}. 
For notational convenience, we write $x = (\lambda,u,\phi)$ and $\bx = (\blambda,\bu,\bphi)$.  
We again refer readers unfamiliar with computer-assisted proofs to~\cite[Section~1.3.1]{Bre25}, 
where similar estimates are derived in a simpler context, and with more details.  

\paragraph{The $Y$ bounds.} Since the nonlinear operator~$\cF$ is polynomial, 
the evaluation~$\cF(\bx)$ is finite (note that in fact it belongs to~$\Pi^{\leq
3N}\cX$), and so is $T(\bx)-\bx = A\cF(\bx)$. 
Therefore, obtaining the bounds $Y=(Y_1,Y_2,Y_3)$ only amounts to finite computations,
and we can control rounding errors by using interval arithmetic~\cite{MooKeaClo09,Tuc11}.

\paragraph{The $Z^1$ bounds.}
In order to obtain the $Z^1$ estimates, one has to estimate the norm of the different 
components of $DT(\bx) = I-AD\cF(\bx)$. Such estimates are standard but somewhat tedious 
to write out fully because of the different components, and can be found in
Appendix~\ref{sec:app:Z12}. We only give a brief overview of their derivation here. 
To find~$Z^1$, we divide our estimates into an estimate on the finite-dimensional 
space, and an estimate on the tail such that, for $1\leq m,i\leq 3$,
\begin{equation*}
Z^1_{m,i} := \max\left(Z^{1,\finite}_{m,i},\, Z^{1,\tail}_{m,i}\right).
\end{equation*}
We can obtain all terms~$Z^{1,\finite}_{m,i}$ with finite numerical calculations using 
interval arithmetic, as more precisely described in the Appendix. 
The terms for the tail involve analytical calculations. Our derivations in the 
appendix show that one can take the following for the estimates of the tail terms: 
\begin{equation*}
\begin{array}{cc}
Z^{1,\tail}_{1,i}=0 & \mbox{ for }  i=1,2,3,\\ 
Z^{1,\tail}_{m,1}=0 & \mbox{ for } m=1,2,3, 
\end{array}
\end{equation*}
\begin{align*}
Z_{2,2}^{1,\tail} = Z_{3,3}^{1,\tail} = \vert\blambda\vert \left( \frac{\left\Vert 1-3\bu^2\right\Vert_X}{\rho_N} +\frac{\vert\sigma\vert}{\rho_{3N}^2} \right), \quad \mbox{and} \quad
Z_{3,2}^{1,\tail} = 6\vert\blambda\vert \frac{\left\Vert \bu\bphi\right\Vert_X}{\rho_N},\ Z_{2,3}^{1,\tail}=0,
\end{align*}
where $\rho_N$ is defined in~\eqref{eq:rhoN}.

\paragraph{The $Z^2$ bounds.} These estimates are obtained by computing the second 
derivative of~$\cF$ and bounding each of its components using the Banach algebra 
property of~$X$. The calculation involves many terms, but each one is elementary.
A concrete example is given in Appendix~\ref{sec:app:Z12}.

\subsection{Application to the SKT system}
\label{sec:proofSKT}

In this section we focus on the SKT system~\eqref{parabolicsys}, which is used 
to study the evolution of populations with two competing species~\cite{ShiKawTer79}. 
The equilibria are solutions of the elliptic system given by  
\begin{equation}
\label{eq:SKT}
\begin{array}{cc}
			F(\lambda,u) = \left(
			\begin{array}{l}
			\DS \Delta \left( \lambda u_1 + d_{11} u_1^2 + d_{12} u_1u_2 \right)
            + \left( r_1 u_1 - a_1 u_1^2 - b_1 u_1u_2 \right)  \\[1ex]
			\DS \Delta \left( \lambda u_2 + d_{21} u_1u_2 + d_{22} u_2^2 \right)
      + \left( r_2 u_2 - a_2 u_2^2 - b_2 u_1u_2 \right)
			\end{array}
			\right)
            =0	& \mbox{ in }\Omega, \\[3ex]
			\partial_n u =  0 & \mbox{ on }\partial\Omega 
\end{array}
\end{equation}
where $u = (u_1,u_2)$ denotes the densities of the two species. One key feature 
of this model is the presence of cross-diffusion terms, which can give rise to a 
repulsive effect leading to pattern formation, i.e., to non-homogeneous steady states. 
There is an extensive body of work dedicated to establishing the existence of such 
steady states, and we refer in particular to~\cite{LouNi96} for a thorough study of 
the competing effects of diffusion ($\lambda$), cross-diffusion ($d_{12}$ and~$d_{21}$) 
and self-diffusion ($d_{11}$ and~$d_{22}$), which leads to (non-)existence results of 
non-homogeneous steady states. A computer-assisted approach was also introduced
in~\cite{breden:22a}, enabling a more quantitative study of these steady states, 
especially in parameter regimes where many of them co-exist.

Here, we again consider rectangular domains $\Omega \subset \R^d$, given by 
$\Omega = \prod_{i=1}^d (0,L_i)$, and study symmetry-breaking bifurcations of 
steady states of~\eqref{eq:SKT}. To that end, we apply Theorem~\ref{genbif:thm:extsys}(b)
via Theorem~\ref{thm:NewtonKantorovich}, in a very similar fashion to what we did in 
Section~\ref{sec:proofOK} for the Ohta--Kawasaki equation. Therefore, we describe
the first steps in a more concise way. The only significant new challenge is the 
presence of nonlinear terms inside the Laplacian, which requires a more involved 
approximate inverse~$A$. However, we again emphasize that these difficulties are 
already present when trying to obtain computer-assisted proofs for regular steady
states of~\eqref{eq:SKT}, and the approach proposed in~\cite{breden:22a} to handle 
such problems generalizes in a straightforward way to the extended 
system~\eqref{genbif:eqn:extsys1}. 

\subsubsection{Zero finding problem and approximate inverse}

We use again the notation introduced in Section~\ref{sec:sequence_space}.
Since~\eqref{eq:SKT} is a two-component system, we consider
$X = \ell^1_s \times \ell^1_s$, also with $s\geq 0$, $Y = \ell^1_{s-2} \times 
\ell^1_{s-2}$, and $F:\R\times X\to Y$ from~\eqref{eq:SKT}, together
with the abbreviation~$F = (F_1,F_2)$.

The map~$\cF$ defining the extended system~\eqref{genbif:eqn:extsys1} is
considered between the spaces~$\cX$ and~$\cY$, where this time one has
$\cX = \R\times \ell^{1,\sym}_s\times \ell^{1,\sym}_s\times \ell^{1}_s 
\times \ell^{1}_s$ and $\cY = \R\times \ell^{1,\sym}_{s-2}\times 
\ell^{1,\sym}_{s-2}\times \ell^{1}_{s-2} \times \ell^{1}_{s-2}$. 
Once again, the precise definition of the symmetry subspaces~$\ell^{1,\sym}_{s}$ 
and~$\ell^{1,\sym}_{s-2}$ is not important at this stage, and we merely assume 
that the inclusion $F(\R,X_\sym)\subset Y_\sym$ is satisfied, and that the
restriction~$D_u F(\lambda,u)|_{X_{\sym}} \in \cL(X_{\sym},Y_{\sym})$ is a 
Fredholm operator of index zero. With the notation of 
Section~\ref{sec:NewtonKantorovich} one then has $\cX^{1} = \R$, 
$\cX^2 = \cX^3 = \ell^{1,\sym}_{s}$, $\cX^4 = \cX^5 = \ell^{1}_{s}$, and, 
denoting $x=(\lambda,u,\phi)=(\lambda,u_1,u_2,\phi_1,\phi_2)$ a generic 
element of~$\cX$, the partial derivatives correspond to $D_1 = D_\lambda$, 
$D_2= D_{u_1}$, $D_3= D_{u_2}$, $D_4 = D_{\phi_1}$, and $D_5 = D_{\phi_2}$. 

Because of the nonlinear terms inside the Laplacian operator, the construction 
of the approximate inverse~$A$ is slightly more complicated than in
Section~\ref{sec:F_A_OK}, and follows the strategy that was introduced 
in~\cite{breden:22a}. 
We still start by numerically constructing~$A^{\leq N}$, which is an approximation 
of $\Pi^{\leq N} D\cF(\bx)^{-1}\Pi^{\leq N}$. It will be convenient to split~$A^{\leq N}$ 
as a $5\times 5$-block matrix $\left(A^{\leq N}_{i,j}\right)_{1\leq i,j\leq 5}$, 
where, with the notation of Theorem~\ref{thm:NewtonKantorovich}, 
$A^{\leq N}_{i,j} = \pi_i A^{\leq N} \pi_j$. The main difference will be that, 
in the tail part of $A$, we can no longer simply use~$\Delta^{-1}$.
Let us now introduce
\begin{align*}
\alpha =
\begin{pmatrix}
\blambda + 2d_{11}\bu_1 + d_{12}\bu_2 & d_{12}\bu_1 \\
d_{21}\bu_2 & \blambda + d_{21}\bu_1 + 2d_{22}\bu_2
\end{pmatrix}, \;
\beta =
\begin{pmatrix}
2d_{11}\bphi_1 + d_{12}\bphi_2 & d_{12}\bphi_1 \\
d_{21}\bphi_2 & d_{21}\bphi_1 + 2d_{22}\bphi_2
\end{pmatrix},
\end{align*}
which are matrices with entries in~$\ell^1_s$ appearing naturally in~$D\cF(\bx)$. 
More precisely, 
\begin{align}
\label{eq:DuFSKT}
D_uF(\blambda,\bu) u = \Delta \left(\alpha u\right) + \text{lower order terms},
\end{align} 
and
\begin{align*}
D_{uu}F(\blambda,\bu)(u,\bphi) = \Delta \left(\beta u\right) + \text{lower order terms}.
\end{align*} 
Thus, looking back at~\eqref{eq:DFblocks}, the leading order terms in~$D\cF(\bx)$ 
for this example are
\begin{equation*}
\left(
\begin{array}{ccc}
0 & 0 & 0 \\ 
0 & \Delta(\alpha \cdot) & 0 \\ 
0 & \Delta(\beta \cdot) & \Delta(\alpha \cdot)
\end{array}
\right).
\end{equation*}
Next we numerically compute 
\begin{align*}
\bw =
\begin{pmatrix}
\bw_{11} & \bw_{12} \\
\bw_{21} & \bw_{22}
\end{pmatrix}
\qquad \text{and}\qquad 
\bg =
\begin{pmatrix}
\bg_{11} & \bg_{12} \\
\bg_{21} & \bg_{22}
\end{pmatrix},
\end{align*}
with entries~$\bw_{ij}$ and~$\bg_{ij}$ in~$\Pi^{\leq N}X$, such that 
$\bw \approx \alpha^{-1}$ and $\bg \approx - \bw \beta \bw$, which
then means that
\begin{equation*}
\left(
\begin{array}{cc}
\bw & 0\\
\bg & \bw
\end{array}
\right)
\approx 
\left(
\begin{array}{cc}
\alpha & 0 \\
\beta & \alpha
\end{array}
\right)^{-1}.
\end{equation*}
Finally one defines~$A$ by having it act like 
$\begin{pmatrix}
\bw & 0 \\
\bg & \bw
\end{pmatrix} \Delta^{-1}$, except on the finite-dimensional subspaces
defined by~$\Pi^{\leq N}$, on which we use~$A^{\leq N}$. More precisely, 
for all $y=(\lambda,u,\phi)\in\cY$,
\begin{align*}
\begin{cases}
A\Pi^{\leq N} y := A^{\leq N}\Pi^{\leq N}y + \left(0,\ \Pi^{> N}[\bw  \Delta^{-1}\Pi^{\leq
N} u],\ \Pi^{> N}[\bg  \Delta^{-1}\Pi^{\leq N}u + \bw  \Delta^{-1}\Pi^{\leq N} \phi]\right)\\[1ex]
A\Pi^{> N} y := \left(0,\ \bw \Delta^{-1} \Pi^{> N} u,\ \bg  \Delta^{-1}\Pi^{> N}u + \bw  \Delta^{-1}\Pi^{> N} \phi\right).
\end{cases}
\end{align*}
We refer to~\cite{breden:22a} for a more schematic description of~$A$, in a simpler 
context with fewer components in~$\cF$. We also note that, as explained 
in~\cite{breden:22a}, a successful application of Theorem~\ref{thm:NewtonKantorovich} 
to this problem implies that~$\alpha$ is invertible, and therefore we see 
from~\eqref{eq:DuFSKT} that~$D_uF(\blambda,\bu)$ is indeed a Fredholm operator 
of index~$0$.

\subsubsection{Bounds $Y$, $Z^1$ and $Z^2$ for Theorem~\ref{thm:NewtonKantorovich}}

The way the bounds are derived is very similar to what we did in
Section~\ref{sec:boundsOK}, with technical adaptations due to the more 
complicated structure of~$A$. These bounds are contained in 
Appendix~\ref{sec:boundsSKT}.

\section{Determining the bifurcation type}
\label{sec:bif}

So far, we have given a description of how to show that a bifurcation occurs.
In this section, we present details on how one can distinguish the bifurcation type. 

\subsection{Establishing bifurcation symmetries}
\label{sec:pitchfork}

We now describe more precisely some of the various symmetries that we encountered 
in our numerical explorations of the bifurcation diagrams for the Ohta--Kawasaki
equation~\eqref{eq:OK} and the SKT system~\eqref{eq:SKT}, for 
$\Omega = \prod_{i=1}^d (0,L_i)$, in dimensions $d=1$ and $d=2$. In particular, we 
identify some cases in which we can prove that the bifurcation must be of pitchfork 
type. Using the notations of Section~\ref{sec:genbif} and 
Theorem~\ref{genbif:thm:suffcond}(b), we therefore strive to prove that 
\begin{align}
\label{eq:condition_pitchfork}
\psibif^*\left( D_{uu}F(\lambdabif,\ubif)[\phibif,\phibif]\right) = 0.
\end{align}
To that end, the following observation will prove to be helpful. 
\begin{lemma}
\label{lem:psi_*}
Under Assumptions~\ref{genbif:assump:nec} and~\ref{genbif:assump:sym}, $\psibif^*|_{Y_{\sym}} = 0$.
\end{lemma}
\begin{proof}
As already noted in the proof of Theorem~\ref{genbif:thm:suffcond}, $L|_{X_{\sym}}:X_\sym\to Y_\sym$ is onto. Therefore, $Y_\sym \subset R(L) = N(\psibif^*)$.
\end{proof}

\subsubsection{The one-dimensional case}
\label{sec:sym_1D}

We first consider both the SKT system~\eqref{eq:SKT} and the Ohta--Kawasaki 
equation~\eqref{eq:OK} on the one-dimensional domain $\Omega=(0,L)$. The 
symmetries we need to consider can be encoded in symmetry spaces of the form
\begin{align*}
\ell^{1,\sym}_s = \{u \in \ell^1_s, \; u_n = 0 \text{ for all }n \notin \nsym\N \},
\end{align*}
for some $\nsym \in\N_{\geq 2}$, which is specified in Table~\ref{tab:OK} and Table~\ref{tab:SKT} for each of our examples. 

Our goal here is to prove that, for $\nsym=2$ and $\nsym=5$ (possibly under additional 
assumptions), the symmetry-breaking bifurcations we proved to exist in 
Section~\ref{sec:examples} must be of pitchfork type, by showing that
$D_{uu}F(\lambdabif,\ubif)[\phibif,\phibif]$ has to lie in a subspace which 
is contained in the kernel of $\psibif^*$.
For this, we first introduce
\begin{align*}
\ell^{1,\asym}_s = \{u \in \ell^1_s, \; u_n = 0 \text{ for all }n \in \nsym\N \},
\end{align*}
such that $\ell^{1}_s = \ell^{1,\sym}_s \oplus \ell^{1,\asym}_s$. We then 
take $X_\sym=\ell^{1,\sym}_s$ and $Y_\sym=\ell^{1,\sym}_{s-4}$ for the 
Ohta--Kawasaki equation, and $X_\sym=\ell^{1,\sym}_s\times\ell^{1,\sym}_s$, 
$Y_\sym=\ell^{1,\sym}_{s-2}\times\ell^{1,\sym}_{s-2}$ for the SKT system. 
Introducing $X_\asym=\ell^{1,\asym}_s$ and $Y_\asym=\ell^{1,\asym}_{s-4}$,
or $X_\asym=\ell^{1,\asym}_s\times\ell^{1,\asym}_s$ and 
$Y_\asym=\ell^{1,\asym}_{s-2}\times\ell^{1,\asym}_{s-2}$, the trigonometric
formula $\cos(p)\cos(q) = \frac{1}{2}\cos(p+q) + \frac{1}{2}\cos(p-q)$ yields 
that $L(X_\asym) \subset Y_\asym$. In particular, Assumption~\ref{genbif:assump:sym} 
implies $\phibif\in X_\asym$. For the remainder of the discussion, we have to
distinguish cases.

\paragraph*{Case $\nsym=2$.}  For this specific value of~$\nsym$, the product of two 
elements of~$\ell^{1,\asym}_s$ is in~$\ell^{1,\sym}_s$ (the product of two odd
functions is even), therefore $D_{uu}F(\lambdabif,\ubif)[\phibif,\phibif]\in Y_{\sym}$, 
and by Lemma~\ref{lem:psi_*} we must have~\eqref{eq:condition_pitchfork}. In summary,
when $\nsym = 2$ we can only have pitchfork bifurcations.

\paragraph*{Case $\nsym=5$.} Here the situation is more subtle, because the product 
of two elements of~$\ell^{1,\asym}_s$ no longer needs to be in~$\ell^{1,\sym}_s$. 
We therefore need to further decompose the space~$\ell^{1,\asym}_s$, and introduce
\begin{align*}
\ell^{1,\asym_1}_s = \left\{u \in \ell^{1,\asym}_s, \; u_n = 0 \text{ for all }n \in (5\N+2)\cup (5\N+3) \right\},
\end{align*}
\begin{align*}
\ell^{1,\asym_2}_s = \left\{u \in \ell^{1,\asym}_s, \; u_n = 0 \text{ for all }n \in (5\N+1)\cup (5\N+4) \right\},
\end{align*}
such that $\ell^{1,\asym}_s = \ell^{1,\asym_1}_s \oplus \ell^{1,\asym_2}_s$. The 
spaces $X_{\asym_i}$ and $Y_{\asym_i}$ for $i=1,2$ are defined accordingly.
Since~$\ubif\in X_{\sym}$, the above-mentioned trigonometric formula for 
$\cos (p) \cos (q)$ implies the inclusions $L(X_{\asym_1}) \subset Y_{\asym_1}$ 
and $L(X_{\asym_2}) \subset Y_{\asym_2}$. 

We now make two additional assumptions, namely that $\phibif\notin X_{\asym_2}$, and 
that the restriction $L|_{X_{\asym_2}}:X_{\asym_2}\to Y_{\asym_2}$ is a Fredholm
operator of index~$0$. One then readily obtains $\phibif\in X_{\asym_1}$, and the 
same arguments as in Lemma~\ref{lem:psi_*} yield that $\psibif^*|_{Y_{\sym}\oplus 
Y_{\asym_2}} = 0$. On the other hand, the product of two elements of $\ell^{1,\asym_1}_s$ 
lies in $\ell^{1,\sym}_s\oplus \ell^{1,\asym_2}_s$ (using again the trigonometric 
formula for $\cos (p) \cos (q)$), and therefore $D_{uu}F(\lambdabif,\ubif)[\phibif,\phibif]$ 
belongs to the space~$Y_{\sym}\oplus Y_{\asym_2}$, and one 
obtains~\eqref{eq:condition_pitchfork}. In summary, under the above additional
assumptions (which are easy to verify in practice) that~$\phibif\notin X_{\asym_2}$ 
and that the restriction~$L|_{X_{\asym_2}}:X_{\asym_2}\to Y_{\asym_2}$ is a Fredholm 
operator of index~$0$, we have a pitchfork bifurcation for $\nsym = 5$. We note that 
$\ell^{1,\asym_1}_s$ and $\ell^{1,\asym_2}_s$ play interchangeable roles here, so 
the same conclusion would hold if $\phibif\notin X_{\asym_1}$ and 
$L|_{X_{\asym_1}}:X_{\asym_1}\to Y_{\asym_1}$ is a Fredholm operator of index~$0$.

There are other values of~$\nsym$, such as for example~$\nsym=3$, for which no similar 
argument can be made to work. Unsurprisingly, in practice we observe transcritical
bifurcations for such values of~$\nsym$. A computer-assisted method for proving that 
a given bifurcation is indeed transcritical is given in Section~\ref{sec:transcritical}.

\subsubsection{The two-dimensional case}

As could be expected, there is a broader variety of symmetries in the two-dimensional 
case, i.e., when we consider the SKT system~\eqref{eq:SKT} or the Ohta--Kawasaki 
equation~\eqref{eq:OK} on the domain $\Omega=(0,L_1)\times(0,L_2)$. Below are some 
examples of subspaces of~$\ell^1_s$ encoding some of the symmetries that we found 
numerically (and that were broken through some bifurcation):
\begin{align*}
&\left\{u \in \ell^1_s, \; u_n = 0 \text{ for all }n \notin \left(\nsym_1 2\N\times \nsym_2 2\N\right) \cup \left(\nsym_1 (2\N+1) \times \nsym_2 (2\N+1)\right)\right\}, \\
&\left\{u \in \ell^1_s, \; u_n = 0 \text{ for all }n \notin \nsym_1 \N\times \nsym_2 \N\right\}, \\
&\left\{u \in \ell^1_s, \; u_n = 0 \text{ for all }n \notin \nsym_1 (2\N+1)\times \nsym_2 (2\N+1)\right\}, \\
&\left\{u \in \ell^1_s, \; u_n = 0 \text{ for all }n \notin \nsym_1 \N\times \nsym_2 (2\N+1)\right\},
\end{align*}
where each time $\nsym \in(\N_{\geq 1})^2$. Each of these was observed in the
two-dimensional Ohta--Kawasaki model. We now focus on the first case, namely
\begin{equation*}
\ell^{1,\sym}_s = \left\{u \in \ell^1_s, \; u_n = 0 \text{ for all }n \notin \left(\nsym_1 2\N\times \nsym_2 2\N\right) \cup \left(\nsym_1 (2\N+1) \times \nsym_2 (2\N+1)\right)\right\},
\end{equation*}
which is the one that also occurs for the two-dimensional SKT system, and show how the 
bifurcation can be proven to be of pitchfork type for some values of~$\nsym$ (sometimes 
under extra assumptions, that are easy to verify in practice). 
As in Section~\ref{sec:sym_1D}, one mainly has to introduce a suitable
complement~$\ell^{1,\asym}_s$ of~$\ell^{1,\sym}_s$, given by
\begin{equation*}
\ell^{1,\asym}_s = \left\{u \in \ell^1_s, \, u_n = 0 \text{ for all }n \in \left(\nsym_1 2\N\times \nsym_2 2\N\right) \cup \left(\nsym_1 (2\N+1) \times \nsym_2 (2\N+1)\right)\right\},
\end{equation*}
and possibly to split $\ell^{1,\asym}_s$ into further subspaces that are stable
under~$L$, and behave appropriately under products. The subspaces~$X_\sym$, $X_\asym$, 
and~$Y_\sym$, etc. are defined as in the one-dimensional case.

\paragraph*{Case $\nsym = (2,1)$ for the Ohta--Kawasaki equation.} We have
\begin{equation*}
\ell^{1,\sym}_s = \left\{u \in \ell^1_s, \; u_n = 0 \text{ for all }n \notin \left(4\N\times 2\N\right) \cup \left((4\N+2) \times (2\N+1)\right)\right\},
\end{equation*}
and we consider
\begin{equation*}
\ell^{1,\asym_1}_s = \left\{u \in \ell^1_s, \; u_n = 0 \text{ for all }n \notin \left(4\N\times (2\N+1)\right) \cup \left((4\N+2) \times 2\N\right)\right\},
\end{equation*}
\begin{equation*}
\ell^{1,\asym_2}_s = \left\{u \in \ell^1_s, \; u_n = 0 \text{ for all }n \notin \left((4\N+1)\times \N\right) \cup \left((4\N+3)\times \N\right)\right\},
\end{equation*}
so that
\begin{align*}
\ell^1_s = \ell^{1,\sym}_s \oplus \ell^{1,\asym_1}_s \oplus \ell^{1,\asym_2}_s.
\end{align*}
It is straightforward to check that both~$L(X_{\asym_1}) \subset Y_{\asym_1}$ and 
$L(X_{\asym_2}) \subset Y_{\asym_2}$ are satisfied. Assuming
that~$\phibif\notin \ell^{1,\asym_2}_s$ holds, one then obtains
$\phibif\in \ell^{1,\asym_1}_s$, and straightforward calculations show that $\phibif^2\in \ell^{1,\sym}_s$ 
and $\ubif\phibif^2\in\ell^{1,\sym}_s$. Thus,
\begin{equation*}
D_{uu}F(\lambdabif,\ubif)[\phibif,\phibif] = 6\lambdabif \Delta(\ubif\phibif^2) \in \ell^{1,\sym}_{s-2}.
\end{equation*} 
Finally, using Lemma~\ref{lem:psi_*}, one indeed obtains the equality 
in~\eqref{eq:condition_pitchfork}. In summary, under the additional assumption
$\phibif\notin \ell^{1,\asym_2}_s$, the bifurcation 
is of pitchfork type.

\paragraph*{Case $\nsym = (1,4)$ for the SKT system.} The procedure is exactly the same, with
\begin{equation*}
\ell^{1,\sym}_s = \left\{u \in \ell^1_s, \; u_n = 0 \text{ for all }n \notin \left(2\N\times 8\N\right) \cup \left((2\N+1) \times (8\N+4)\right)\right\},
\end{equation*}
\begin{equation*}
\ell^{1,\asym_1}_s = \left\{u \in \ell^1_s, \; u_n = 0 \text{ for all }n \notin \left((2\N+1)\times 8\N\right) \cup \left(2\N \times (8\N+4)\right)\right\},
\end{equation*}
\begin{equation*}
\ell^{1,\asym_2}_s = \left\{u \in \ell^1_s, \; u_n = 0 \text{ for all }n \notin \N \times \bigcup_{j\notin \{0,4\}} (8\N+j)\right\},
\end{equation*}
and taking into account the different formula for
$D_{uu}F(\lambdabif,\ubif)[\phibif,\phibif]$. Under the additional assumption 
$\phibif\notin X_{\asym_2}$, the bifurcation is a pitchfork bifurcation.

There are other values of~$\nsym$, such as for example~$\nsym=(1,3)$, for which no 
similar argument can be made to work. Unsurprisingly, in practice we observe 
transcritical bifurcations for such values of~$\nsym$. A computer-assisted method 
for proving that a given bifurcation is indeed transcritical is presented below.

\subsection{Proving transcritical bifurcations}
\label{sec:transcritical}

In this last section, we present a method which allows us to prove that a given 
bifurcation, obtained via Theorem~\ref{thm:NewtonKantorovich} and 
Theorem~\ref{genbif:thm:extsys}(b), is transcritical. Using the notations of 
Section~\ref{sec:genbif} and Theorem~\ref{genbif:thm:suffcond}(a), we therefore
strive to prove that 
\begin{align}
\label{eq:condition_transcritical}
\psibif^*\left( D_{uu}F(\lambdabif,\ubif)[\phibif,\phibif]\right) \neq 0.
\end{align}
In contrast to the pitchfork bifurcations discussed in Section~\ref{sec:pitchfork}, 
whose existence follows from specific structural assumptions on the solutions, 
transcritical bifurcations are generic, which makes it harder to prove by hand 
that a given bifurcation is actually transcritical. Informally speaking, 
if $\psibif^*\left( D_{uu}F(\lambdabif,\ubif)[\phibif,\phibif]\right)$ is 
in fact equal to~$0$, then there should be a good reason for why this is the
case, and we might be able to understand it. However, there may not be a clear 
reason why this quantity is different from~$0$.

Fortunately, this is one of the many situations where computer-assisted proofs 
complement more traditional techniques very well. Indeed, once we have rigorous 
enclosures of~$\lambdabif$, $\ubif$, $\phibif$, and~$\psibif$, we readily get a 
rigorous enclosure of $\psibif^*\left( D_{uu}F(\lambdabif,\ubif)[\phibif,\phibif]\right)$. 
From such an enclosure, one cannot hope to prove that~\eqref{eq:condition_pitchfork} 
holds, but as we saw in Section~\ref{sec:pitchfork} this can be done by hand. On the 
other hand, provided the enclosure is sharp enough, it can be used to 
prove~\eqref{eq:condition_transcritical}, which was the case that seemed harder 
to study by hand.

The computer-assisted proofs for the existence of a bifurcation described in 
Sections~\ref{sec:NewtonKantorovich}-\ref{sec:proofSKT} already gave us relatively 
tight enclosures for~$\lambdabif$, $\phibif$, and~$\ubif$. The missing piece is an
explicit enclosure of~$\psibif^*$, which we are going to obtain using a separate 
computer-assisted proof. Once this is done, we readily get a computable enclosure 
for $\psibif^*\left( D_{uu}F(\lambdabif,\ubif)[\phibif,\phibif]\right)$, and can 
check whether it implies~\eqref{eq:condition_transcritical}.

Recall that we want to find~$\psibif^*\in Y^*$ such that $R(L) = N(\psibif^*)$. 
By the Fredholm alternative, this is equivalent to finding a nontrivial 
$\psibif^*\in N(L^*)$. The equation $L^*\psi^* = 0$ is the one we are going 
to solve using a computer-assisted proof which is very similar to the ones described 
earlier in Section~\ref{sec:CAPs}.

We now focus on rigorously enclosing $\psi^*$ for the Ohta--Kawasaki equation 
of Section~\ref{sec:proofOK}, but the procedure we propose is easily adapted 
to the SKT system of Section~\ref{sec:proofSKT}.
First, we recall that we work here with spaces $X$ and $Y$ which are 
weighted $\ell^1$ spaces, and whose duals are thus
(isomorphic to) weighted $\ell^\infty$ spaces. We therefore introduce, 
for $s\in\R$, the spaces
\begin{equation*}
\ell^\infty_{-s} := \left\{ u=(u_n)_{n\in\N^d}, \ \left\Vert u\right\Vert_{\ell^\infty_{-s}} := \sup_{n\in\N^d} \frac{\vert u_n\vert}{\xi^{(s)}_n}  <\infty \right\}.
\end{equation*}
We then have that $\left(\ell^1_s\right)^* \cong \ell^\infty_{-s}$.
In particular, any $\psi^*\in \left(\ell^1_s\right)^*$ can be uniquely represented by the element $\psi\in \ell^\infty_{-s}$ such that, for all $u\in \ell^1_s$,
\begin{align*}
\psi^*(u) = \langle \psi^*,u\rangle_{(\ell^1_s)^*,\ell^1_s} =  \sum_{n\in\N^d} \psi_n u_n.
\end{align*}
With this identification, we have $\Vert \psi^*\Vert_{\left(\ell^1_s\right)^*} = \Vert \psi\Vert_{\ell^\infty_{-s}}$.


In Section~\ref{sec:proofOK}, for the purpose of enclosing~$\lambdabif$,
$\ubif$, and~$\phibif$, we considered the two sequence spaces $X=\ell^1_s$ and
$Y = \ell^1_{s-4}$, for some $s\geq 0$. In order to enclose $\psibif$, we thus 
consider the zero-finding problem $\cG^*:\R\times \ell^\infty_{-(s-4)} \to 
\R\times \ell^\infty_{-s}$ given by
\begin{align*}
\cG^*(\mu,\psi) = (\psi^*(\bphi)-1, D_uF(\lambdabif,\ubif)^*\psi - \mu\psi),
\end{align*}
where~$\bphi$ is the approximation of~$\phibif$ used to validate the existence of the 
bifurcation. At first glance, also solving for~$\mu$ and not only for~$\psi$ might look 
surprising, as we already know we want~$\mu$ to be equal to~$0$. However, if we want to 
successfully apply Theorem~\ref{thm:NewtonKantorovich}, the solution we look for must 
be locally unique, and adding~$\mu$ as an unknown is a natural way of regularizing the 
problem. Moreover, this extra unknown is no impediment, as there is a simple a posteriori 
check that can be used to prove that, if we have a zero of~$\cG^*$, then~$\mu$ was in 
fact equal to~$0$.

\begin{lemma}
Let $(\mu_0,\psibif)$ be a zero of the above-defined operator~$\cG^*$. In
addition, suppose that $\Vert \psibif^* - \bpsi^* \Vert_{Y^*} \leq r_\psi$ 
and that $\Vert \phibif - \bphi \Vert_X \leq r_\phi$. If  
\begin{equation}
\label{eq:cond_psiphi}
\left\Vert \bpsi^*\right\Vert_{Y^*} r_\phi + \left\Vert \bphi\right\Vert_{Y} r_\psi + r_\psi r_\phi < \left\vert \bpsi^*(\bphi) \right\vert,
\end{equation}
then the identity $D_uF(\lambdabif,\ubif)^*\psibif=0$ is satisfied. 
\end{lemma}
\begin{proof}
Since $(\mu_0,\psibif)$ is a zero of~$\cG^*$, $\psibif$ is an eigenfunction
of $D_uF(\lambdabif,\ubif)^*$, with eigenvalue~$\mu_0$, whereas~$\phibif$ is an 
eigenfunction of~$D_uF(\lambdabif,\ubif)$, with eigenvalue~$0$. By a classical 
argument, if~$\mu_0$ were to be nonzero, then we would have $\psibif^*(\phibif) = 0$. 
Indeed,
\begin{align*}
\mu_0 \psibif^*(\phibif) =  \psibif^*(D_uF(\lambdabif,\ubif)\phibif) = 0.
\end{align*}
Yet, our assumption~\eqref{eq:cond_psiphi} implies that $\psibif^*(\phibif) \neq 0$,
and therefore one must have $\mu_0 = 0$.
\end{proof}

In order to get a rigorous enclosure for a zero of~$\cG^*$, we proceed exactly as we 
did when studying the extended zero-finding problem~$\cF$. That is, we first obtain 
numerically an approximate zero~$(\bmu,\bpsi)$ of~$\cG^*$, then introduce a suitable 
approximate inverse~$A^*$ of~$D\cG^*(\bmu,\bpsi)$, and finally apply 
Theorem~\ref{thm:NewtonKantorovich} to the quasi-Newton map
\begin{align}
\label{eq:TforG}
T:(\mu,\psi) \mapsto (\mu,\psi) - A^*\cG^*(\mu,\psi).
\end{align}
\begin{remark}
\label{rem:opellinftyVSopell1}
This fixed point map is defined on $\R\times \ell^\infty_{-(s-4)}$. In order to apply Theorem~\ref{thm:NewtonKantorovich}, we need to control various operator norms, and in particular for operators on $\ell^\infty_{-(s-4)}$. This can easily be done, as the operator norm associated to a weighted $\ell^\infty$ norm is nothing but the weighted supremum of the weighted $\ell^1$-norm of the rows. That is, for any $B^*\in\cL(\ell^\infty_{-s})$,
\begin{align*}
\Vert B^* \Vert_{\ell^\infty_{-s}} = \sup_{i\in\N^d}\frac{1}{\xi^{(s)}_i} \sum_{j\in\N^d} \vert B^*_{i,j} \vert \xi^{(s)}_j.
\end{align*}
However, if there exists $B\in\cL(\ell^1_s)$ such that~$B^*$ is the adjoint of~$B$, 
then we also have the identity $\Vert B^* \Vert_{\ell^\infty_{-s}} = \Vert B \Vert_{\ell^1_s}$, 
which brings us exactly back to the setting of Section~\ref{sec:proofOK}, where we 
derived all the estimates using weighted~$\ell^1$ spaces.

We can take advantage of this by considering $B\in\cL(\R\times\ell^1_{s},\R\times\ell^1_{s-4})$ 
given by
\begin{equation*}
B = \left(
\begin{array}{cc}
0 & -\bpsi^*\\
\bphi & D_uF(\lambdabif,\ubif)-\bmu
\end{array}
\right),
\end{equation*}
and then construct an approximate inverse $A\in\cL(\R\times\ell^1_{s-4},\R\times\ell^1_{s})$ 
of this operator~$B$, again in the form $A = A^{\leq N} - \Delta^{-2}\Pi^{>N}$.
In view of the identity~$D\cG^*(\bmu,\bpsi) = B^*$, we use the adjoint~$A^*$ of~$A$ 
for defining~$T$ in~\eqref{eq:TforG}. Thus
\begin{align*}
DT(\bmu,\bpsi) = I - A^* D\cG^*(\bmu,\bpsi) = (I - BA)^*,
\end{align*}
and all the estimates for~$DT$ necessary for Theorem~\ref{thm:NewtonKantorovich} can 
be obtained for~$I - BA$, with weighted~$\ell^1$ spaces.
\end{remark}

The technical details necessary to apply Theorem~\ref{thm:NewtonKantorovich} 
to~\eqref{eq:TforG} are very similar to those of Section~\ref{sec:proofOK}, and 
are presented in Appendix~\ref{app:sec:detailbounds}.
\section*{Acknowledgments}
M.B. is supported by the ANR project CAPPS: ANR-23-CE40-0004-01.
The research of E.S.\ was partially supported by the Simons Foundation
under Award~636383, and T.W.\ was partially supported by the Simons
Foundation under Award~581334 and the Grant SFI-MPS-TSM-00013686.

\appendix
\section*{Appendix}

This appendix provides technical details regarding the derivation of the 
computable bounds required by our computer-assisted proofs. Appendix~\ref{sec:app:Z12} 
relates to Section~\ref{sec:proofOK}, Appendix~\ref{sec:boundsSKT} to 
Section~\ref{sec:proofSKT} and Appendix~\ref{app:sec:detailbounds} to 
Section~\ref{sec:transcritical}. We emphasize once more that the main 
ideas can already be understood when studying a regular steady state, 
and that the remaining challenges are only of technical nature.

\section{Computable bounds for the Ohta--Kawasaki equation}\label{sec:app:Z12}

\paragraph{The $Z^1$ bounds.}
In order to obtain the $Z^1$ estimates, we have to estimate the norm of the 
different components of $DT(\bx) = I-AD\cF(\bx)$. We start with the following
observation. For any $x\in\Pi^{>3N}\cX$, $D\cF(\bx)x \in \Pi^{>N}\cX$, i.e.,
$D\cF(\bx)\Pi^{>3N} = \Pi^{>N}D\cF(\bx)\Pi^{>3N}$, and hence 
\begin{align*}
AD\cF(\bx)\Pi^{>3N} = \Delta^{-2}\Pi^{>N}D\cF(\bx)\Pi^{>3N}.
\end{align*}
This quantity no longer depends on the finite part~$A^{\leq N}$ of~$A$, and the 
term~$\Delta^{-2}\Pi^{>N}$ will allow us to easily obtain explicit estimates
showing that $AD\cF(\bx)\Pi^{>3N}$ becomes small when~$N$ is large. On the 
other hand, for any $x\in\Pi^{\leq 3N}\cX$, $D\cF(\bx)x \in \Pi^{\leq 5N}\cX$, 
and $AD\cF(\bx)x \in \Pi^{\leq 5N}\cX$. In particular, 
\begin{equation*}
AD\cF(\bx)\Pi^{\leq 3N} = \Pi^{\leq 5N}AD\cF(\bx)\Pi^{\leq 3N},
\end{equation*}
which is a finite dimensional operator, hence we can obtain its norm (or 
more precisely, the norm of its components, as required by 
Theorem~\ref{thm:NewtonKantorovich}) with a finite computation.

This motivates the introduction of two separate sets of bounds~$Z^{1,\finite}_{m,i}$ 
and~$Z^{1,\tail}_{m,i}$, such that, for $1\leq m,i\leq 3$,
\begin{align*}
\left\Vert\pi_m D_i T(\bar{x})\Pi^{\leq 3N} \right\Vert_{B(\cX^i,\cX^m)} &\leq Z^{1,\finite}_{m,i} \\
\left\Vert\pi_m D_i T(\bar{x})\Pi^{> 3N} \right\Vert_{B(\cX^i,\cX^m)} &\leq Z^{1,\tail}_{m,i}.
\end{align*}
By standard properties of~$\ell^1$ norms, $Z^1_{m,i} := \max\left(Z^{1,\finite}_{m,i},\, Z^{1,\tail}_{m,i}\right)$ 
then satisfies~\eqref{eq:def_Z}.

As already observed above, for every $1\leq m,i\leq 3$,
\begin{equation*}
\left\Vert\pi_m D_i T(\bar{x})\Pi^{\leq 3N} \right\Vert_{B(\cX^i,\cX^m)} = \left\Vert\pi_m \Pi^{\leq 5N}D_i T(\bar{x}) \Pi^{\leq 3N}\right\Vert_{B(\cX^i,\cX^m)},
\end{equation*}
which amounts to a weighted~$\ell^1$ operator norm of a finite matrix, therefore
we can obtain all the~$Z^{1,\finite}_{m,i}$ terms with finite computations.

It remains to find computable tail estimates. First note that
one can take $Z^{1,\tail}_{1,i}=0$ for every $i=1,2,3$ (simply because, for all 
$x=(\lambda,u,\phi)\in\cX$, $\Pi^{>N} x = (0,\Pi^{>N}u,\Pi^{>N}\phi)$, hence 
one has $\pi_1\Pi^{>N} x=0$). Moreover, $Z^{1,\tail}_{m,1}=0$ for $m=1,2,3$, as
\begin{align*}
D_1 \cF(\bx) = D_\lambda \cF(\bx) = 
\begin{pmatrix}
0 \\
-\Delta(\bu-\bu^3) - \sigma(\bu-\mu) \\
-\Delta((1-3\bu^2)\bphi) - \sigma\bphi
\end{pmatrix} \in \Pi^{\leq 3N}\cX,
\end{align*} 
hence $D_1 \cF(\bx)\Pi^{>3N} = 0$.
We look at the remaining terms separately. We have
\begin{align*}
\pi_2 D_2 \cF(\bx) u = D_uF(\blambda,\bu)u = -\Delta^2 u -\blambda \Delta((1-3\bu^2)u) - \blambda\sigma u,
\end{align*}
therefore
\begin{align*}
\pi_2 D_2 T(\bx) \Pi^{>3N}u = -\blambda\Delta^{-1}\Pi^{>N}\left((1-3\bu^2)\Pi^{>3N}u\right) -\blambda\sigma\Delta^{-2}\Pi^{>3N}u,
\end{align*}
and 
\begin{align}
\label{eq:Z1tail_example}
\left\Vert \pi_2 D_2 T(\bx) \Pi^{>3N}u \right\Vert_X \leq \vert \blambda\vert \left( \frac{\left\Vert 1-3\bu^2\right\Vert_X}{\rho_N} +\frac{\vert\sigma\vert}{\rho_{3N}^2} \right) \left\Vert \Pi^{>3N}u \right\Vert_X,
\end{align}
where $\rho_N$ denotes the smallest eigenvalue modulus of $\Delta\Pi^{>N}$, namely 
\begin{align}
\label{eq:rhoN}
\rho_N = \min_{1\leq j\leq d}\left(\frac{\pi(N_j+1)}{L_j}\right)^2.
\end{align}
We emphasize that the estimate~\eqref{eq:Z1tail_example} relies on the fact that $X=\ell^1_s$ is a Banach algebra, which is why we assumed $s\geq 0$.
In conclusion, we take
\begin{align*}
Z_{2,2}^{1,\tail} = \vert\blambda\vert \left( \frac{\left\Vert 1-3\bu^2\right\Vert_X}{\rho_N} +\frac{\vert\sigma\vert}{\rho_{3N}^2} \right).
\end{align*}
Similarly, from
\begin{align*}
\pi_3 D_2 \cF(\bx) u = D_{uu}F(\blambda,\bu)(u,\bphi) = 6\blambda \Delta(\bu\bphi u),
\end{align*}
one obtains
\begin{align*}
Z_{3,2}^{1,\tail} = 6\vert\blambda\vert \frac{\left\Vert \bu\bphi\right\Vert_X}{\rho_N}.
\end{align*}
Next,
\begin{align*}
\pi_2 D_3 \cF(\bx) = D_\phi F(\blambda,\bu)= 0,
\end{align*}
hence we take $Z_{2,3}^{1,\tail} =0$. Finally 
\begin{align*}
\pi_3 D_3 \cF(\bx) \phi = D_uF(\blambda,\bu)\phi,
\end{align*}
and therefore we take $Z_{3,3}^{1,\tail} =Z_{2,2}^{1,\tail}$. This completes our 
derivation of the $Z^1$ bounds.

\paragraph{The $Z^2$ bounds.} The derivation of the~$Z^2$ estimates is more
elaborate, because it involves many terms. Apart from that, however, it is very 
straightforward. First, observe that one has the identity
$DT(x)-DT(\bx) = -A(D\cF(x)-D\cF(\bx))$, 
and therefore
\begin{align*}
\left\Vert\pi_m (D_i T(x)- D_i T(\bar{x})) \right\Vert_{B(\cX^i,\cX^m)} &\leq \sum_{l=1}^3 \left\Vert\pi_m A\pi_l \right\Vert_{B(\cY^l,\cX^m)} \left\Vert\pi_l (D_i \cF(x)- D_i \cF(\bar{x})) \right\Vert_{B(\cX^i,\cY^l)},
\end{align*}
where we go via the~$\cY$ spaces to account for the fact that $D \cF(x)- D\cF(\bar{x})$ 
is not bounded on~$\cX$. Owing to the structure of~$A$, each 
$\left\Vert\pi_m A\pi_l \right\Vert_{B(\cY^l,\cX^m)}$ amounts to a finite 
computation, and we are left with estimating
\begin{align*}
\left\Vert\pi_l (D_i \cF(x)- D_i \cF(\bar{x})) \right\Vert_{B(\cX^i,\cY^l)} &\leq \sum_{j=1}^3 \sup_{x\in\Box(\bx,\rstar)} \left\Vert\pi_l D_{ij} \cF(x) \right\Vert_{BL(\cX^i\times\cX^j,\cY^l)} \left\Vert\pi_j (x-\bx) \right\Vert_{\cX^j},
\end{align*}
where~$BL(\cX^i\times\cX^j,\cY^l)$ denotes the space of bounded bilinear operators
from~$\cX^i\times\cX^j$ into~$\cY^l$. That is, we are going to find computable 
constants~$\tilde{Z}^2_{l,i,j}$ such that, for all $1\leq i,j,l\leq 3$,
\begin{equation*}
\sup_{x\in\Box(\bx,\rstar)}\left\Vert\pi_l D_{ij} \cF(x) \right\Vert_{BL(\cX^i\times\cX^j,\cY^l)} \leq \tilde{Z}^2_{l,i,j},
\end{equation*}
and then take
\begin{equation*}
Z^2_{m,i,j} := \sum_{l=1}^3 \left\Vert\pi_m A\pi_l \right\Vert_{B(\cY^l,\cX^m)} \tilde{Z}^2_{l,i,j}.
\end{equation*}
There are quite a few~$\tilde{Z}^2_{l,i,j}$ terms to estimate, but each of them is rather 
trivial. For instance, recalling that $\pi_2\cF(x) = F(\lambda,u)$, we have
\begin{align*}
\pi_2 D_{12}\cF(x) = D_{\lambda u} F(\lambda,u) = -\Delta\left((1-3u^2)\cdot\right)\cdot,
\end{align*}
meaning that
\begin{align*}
D_{\lambda u} F(\lambda,u)(\tilde\lambda,\tilde u) = -\Delta\left((1-3u^2)\tilde{u}\right)\tilde{\lambda}.
\end{align*}
In particular, using $\left\Vert \Delta\right\Vert_{B(X,Y)} \leq
\max_i\left(\frac{\pi}{L_i}\right)^2$ and $x\in\Box(\bx,\rstar)$, one obtains
\begin{align*}
\left\Vert D_{\lambda u} F(\lambda,u) (\tilde\lambda,\tilde u) \right\Vert_{Y} &\leq \max_i\left(\frac{\pi}{L_i}\right)^2 \left\Vert 1-3u^2\right\Vert_X \left\Vert \tilde u\right\Vert_X  \vert \tilde\lambda\vert \\
&\leq \max_i\left(\frac{\pi}{L_i}\right)^2 \left(\left\Vert 1-3\bu^2\right\Vert_X + 6\left\Vert \bu\right\Vert_X \rstar_u + 3 (\rstar_u)^2 \right)\left\Vert \tilde u\right\Vert_X  \vert \tilde\lambda\vert,
\end{align*}
hence 
\begin{align*}
\sup_{x\in\Box(\bx,\rstar)} \left\Vert\pi_2 D_{12} \cF(x) \right\Vert_{BL(\cX^1\times\cX^2,\cY^2)} \leq \max_i\left(\frac{\pi}{L_i}\right)^2 \left(\left\Vert 1-3\bu^2\right\Vert_X + 6\left\Vert \bu\right\Vert_X \rstar_2 + 3 (\rstar_2)^2\right),
\end{align*}
and we can take
\begin{align*}
\tilde{Z}^2_{2,1,2} = \max_i\left(\frac{\pi}{L_i}\right)^2\left( \left\Vert 1-3\bu^2\right\Vert_X + 6\left\Vert \bu\right\Vert_X \rstar_2 + 3 (\rstar_2)^2\right).
\end{align*}
All of the other~$\tilde{Z}^2_{l,i,j}$ terms are obtained in a similar fashion.

\section{Computable bounds for the SKT equation}\label{sec:boundsSKT}

In this appendix, we mostly emphasize the differences with the Ohta--Kawasaki 
equation, which are mainly due to the nonlinear diffusion terms in the SKT 
system, and refer the reader again to~\cite{breden:22a} for details in a simpler case.

\paragraph*{The $Y$ bounds.} Everything is very similar to the Ohta--Kawasaki case. 
Because the system is now quadratic instead of cubic, $\cF(\bx)\in\Pi^{\leq 2N}$, but 
then $A\cF(\bx)\in\Pi^{\leq 3N}$ because of the multiplication by $\bw\in\Pi^{\leq N}$
and $\bg\in \Pi^{\leq N}$ contained in~$A$. But in the end $A\cF(\bx)\in\Pi^{\leq 3N}$ is 
again finite, therefore each $Y_i = \left\Vert \pi_i A\cF(\bx) \right\Vert_{\cX_i}$ can be 
computed.

\paragraph{The $Z^1$ bounds.} We again introduce separate~$Z^{1,\finite}$ and~$Z^{1,\tail}$ 
bounds such that, for all choices of $1\leq m,i\leq 5$,
\begin{align*}
\left\Vert\pi_m D_i T(\bar{x})\Pi^{\leq 2N} \right\Vert_{B(\cX^i,\cX^m)} &\leq Z^{1,\finite}_{m,i} \\
\left\Vert\pi_m D_i T(\bar{x})\Pi^{> 2N} \right\Vert_{B(\cX^i,\cX^m)} &\leq Z^{1,\tail}_{m,i},
\end{align*}
the slightly different projection sizes being due to $F$ only being quadratic.

Remembering that $A$ now incorporates multiplication operators by~$\omega$ and~$\gamma$, 
we get, for every $1\leq m,i\leq 5$,
\begin{equation*}
\left\Vert\pi_m D_i T(\bar{x})\Pi^{\leq 2N} \right\Vert_{B(\cX^i,\cX^m)} = \left\Vert\pi_m \Pi^{\leq 4N}D_i T(\bar{x})\Pi^{\leq 2N} \right\Vert_{B(\cX^i,\cX^m)}.
\end{equation*}
Therefore, we can again obtain all the $Z^{1,\finite}_{m,i}$ with finite computations.

Regarding the tail estimates, as was the case for the Ohta--Kawasaki example, many of 
the~$Z^{1,\tail}_{m,i}$ are equal to zero, and we only give the nontrivial ones 
below. By slightly abusing notation, we apply the $\ell^1_s$-norm entry-wise to 
matrices, e.g.,
\begin{align*}
\left\Vert \bw\right\Vert_{\ell^1_s} := 
\begin{pmatrix}
\left\Vert \bw_{11}\right\Vert_{\ell^1_s} & \left\Vert \bw_{12}\right\Vert_{\ell^1_s} \\
\left\Vert \bw_{21}\right\Vert_{\ell^1_s} & \left\Vert \bw_{22}\right\Vert_{\ell^1_s}
\end{pmatrix}.
\end{align*}
We also denote once more by~$\rho_N$ the smallest eigenvalue modulus of~$\Delta\Pi^{>N}$, 
given by~\eqref{eq:rhoN}.
Using that~$\ell^1_s$ is a Banach algebra (because we assumed $s\geq 0$),
one then obtains
\begin{align*}
\renewcommand{\arraystretch}{1.5}
\begin{pmatrix}
Z^{1,\tail}_{2,2} & Z^{1,\tail}_{2,3} \\
Z^{1,\tail}_{3,2} & Z^{1,\tail}_{3,3}
\end{pmatrix} 
&=
\left\Vert
\begin{pmatrix}
1 & 0 \\
0 & 1
\end{pmatrix}
-\omega \alpha
\right\Vert_{\ell^1_s} \\
&\quad 
+ \frac{1}{\rho_N} \left\Vert
\omega
\right\Vert_{\ell^1_s}\left\Vert
\begin{pmatrix}
r_1 - 2a_{1}\bu_1 - b_{1}\bu_2 & -b_{1}\bu_1 \\
-b_{2}\bu_2 & r_2 - b_2\bu_1 - 2a_{2}\bu_2
\end{pmatrix}
\right\Vert_{\ell^1_s},
\end{align*}
\begin{align*}
\renewcommand{\arraystretch}{1.5}
\begin{pmatrix}
Z^{1,\tail}_{4,4} & Z^{1,\tail}_{4,5} \\
Z^{1,\tail}_{5,4} & Z^{1,\tail}_{5,5}
\end{pmatrix} 
=
\begin{pmatrix}
Z^{1,\tail}_{2,2} & Z^{1,\tail}_{2,3} \\
Z^{1,\tail}_{3,2} & Z^{1,\tail}_{3,3}
\end{pmatrix},
\end{align*}
and
\begin{align*}
\renewcommand{\arraystretch}{1.5}
\begin{pmatrix}
Z^{1,\tail}_{4,2} & Z^{1,\tail}_{4,3} \\
Z^{1,\tail}_{5,2} & Z^{1,\tail}_{5,3}
\end{pmatrix} 
=
\left\Vert
\gamma \alpha + \omega\beta
\right\Vert_{\ell^1_s} 
&+ \frac{1}{\rho_N} \left\Vert
\gamma
\right\Vert_{\ell^1_s}  \left\Vert
\begin{pmatrix}
r_1 - 2a_{1}\bu_1 - b_{1}\bu_2 & -b_{1}\bu_1 \\
-b_{2}\bu_2 & r_2 - b_2\bu_1 - 2a_{2}\bu_2
\end{pmatrix}
\right\Vert_{\ell^1_s}\\
& 
+ \frac{1}{\rho_N} \left\Vert
\omega
\right\Vert_{\ell^1_s}  \left\Vert
\begin{pmatrix}
2a_{1}\bphi_1 + b_{1}\bphi_2 & b_{1}\bphi_1 \\
b_{2}\bphi_2 & b_2\bphi_1 + 2a_{2}\bphi_2
\end{pmatrix}
\right\Vert_{\ell^1_s}.
\end{align*}

\paragraph{The $Z^2$ bounds.} These bounds can be derived exactly as explained 
in Appendix~\ref{sec:app:Z12}, but things are even simpler here because the SKT 
system only has quadratic nonlinearities.

\section{Computable bounds for proving transcriticality}\label{app:sec:detailbounds}

In this final appendix we sketch how to obtain the bounds necessary to apply 
Theorem~\ref{thm:NewtonKantorovich} to~\eqref{eq:TforG}.
Compared to Section~\ref{sec:proofOK} and Appendix~\ref{sec:app:Z12}, one extra 
difficulty that one has to deal with is that the map~$\cG^*$ depends on~$(\lambdabif,\ubif)$, 
but in practice we only have access to~$(\blambda,\bu)$ and to the error bounds
\begin{align*}
\vert \lambdabif - \blambda\vert \leq r_\lambda,\qquad \Vert \ubif - \bu \Vert_{X} \leq r_u.
\end{align*}
Thus, we split $\cG^* = \bcG^* + \varepsilon_{\cG^*}$, where
\begin{align*}
\bcG^*(\mu,\psi) := (\psi^*(\bphi)-1, D_uF(\blambda,\bu)^*\psi - \mu\psi),
\end{align*}
and
\begin{align*}
\varepsilon_{\cG^*}(\mu,\psi) := (0, \left(D_uF(\lambdabif,\ubif)-D_uF(\blambda,\bu)\right)^*\psi).
\end{align*}
The term~$\bcG^*$ is the one for which we do most of the careful estimates, 
whereas~$\varepsilon_{\cG^*}$ is treated as an error term, and thus bounded more crudely. 

We first collect some straightforward-to-prove but useful a priori estimates.
\begin{lemma}
\label{lem:est_ells}
For all $s\in\R$, and $u\in\ell^1_s$,
\begin{equation}
\label{eq:estLap}
\left\Vert \Delta u\right\Vert_{\ell^1_{s-2}} \leq \max_i\left(\frac{\pi}{L_i}\right)^2 \left\Vert u\right\Vert_{\ell^1_{s}},
\end{equation}
\begin{equation}
\label{eq:estinvLapN}
\left\Vert \Delta^{-1}\Pi^{>N} u\right\Vert_{\ell^1_{s+2}} \leq \max_i\left(\frac{L_i}{\pi}\right)^2 \left(\frac{1+m_N}{m_N} \right)^2\left\Vert \Pi^{>N}u\right\Vert_{\ell^1_{s}},
\end{equation}
where $m_N = \min_i N_i+1$.
\end{lemma}
Before formulating the next auxiliary result, we need to introduce some further notation. 
For $A\in \cL(\R\times \ell^1_s, \R\times\ell^1_t)$, let us denote
\begin{align*}
\left\Vert A \right\Vert_{s\to t} := 
\begin{pmatrix}
\left\vert \pi_1 A \pi_1 \right\vert & \left\Vert \pi_1 A \pi_2 \right\Vert_{\left(\ell^1_{s}\right)^*} \\
\left\Vert \pi_2 A \pi_1 \right\Vert_{\ell^1_{t}} & \left\Vert \pi_2 A \pi_2 \right\Vert_{B(\ell^1_{s},\ell^1_{t})}
\end{pmatrix},
\end{align*}
and similarly for $A^*\in \cL(\R\times \ell^\infty_{-t}, \R\times\ell^\infty_{-s})$. We now give a computable bound for $A^* \epsilon_{\cG^*}$.
\begin{lemma}
Assuming $s\geq 2$, we have the estimate
\begin{multline}
\label{eq:estAepsG}
\left(\left\Vert A^* \epsilon_{\cG^*} \right\Vert_{s-4\to s-4}\right)^* \leq 
\begin{pmatrix}
0 & 0\\
0 & r_\lambda \vert \sigma\vert 
\end{pmatrix} \left\Vert A \right\Vert_{s-4\to s-4} 
\\
+ \begin{pmatrix}
0 & 0\\
0 & \max_i\left(\frac{\pi}{L_i}\right)^2\left( r_\lambda\Vert 1-3\bu^2\Vert_{\ell^1_{s-2}} + 3  (\vert\blambda\vert+r_\lambda) r_u (2\Vert\bu\Vert_{\ell^1_{s-2}} + r_u) \right)
\end{pmatrix} \left\Vert A \right\Vert_{s-4\to s-2} .
\end{multline}
\end{lemma}
\begin{proof}
We simply write
\begin{align*}
D_uF(\lambdabif,\ubif) - D_uF(\blambda,\bu) &= -\lambdabif \Delta\left((1-3\ubif^2)\cdot \right) - \lambdabif\sigma + \blambda \Delta\left((1-3\bu^2)\cdot \right) + \blambda\sigma \\
&= (\blambda-\lambdabif)\sigma + (\blambda-\lambdabif) \Delta\left((1-3\bu^2)\cdot \right) + 3\lambdabif \Delta\left((\ubif-\bu)(\ubif+\bu)\cdot \right),
\end{align*}
and use Lemma~\ref{lem:est_ells}.
\end{proof}

After these preparations we can finally derive computable estimates~$Y_m$,
$Z^1_{m,i}$, and~$Z^2_{m,i,j}$ which satisfy assumptions~\eqref{eq:def_Y}-\eqref{eq:def_W} 
of Theorem~\ref{thm:NewtonKantorovich}, for the operator~$T$ given by~\eqref{eq:TforG}, 
with the space $\cX=\R\times\ell^\infty_{-(s-4)}$, the map~$\cG$, and the linear operator~$A$
introduced in Section~\ref{sec:transcritical}, all under the critical assumption that $s\geq 2$.

\paragraph{The $Y$ bounds.} We simply split
\begin{align*}
T(\bmu,\bpsi) - (\bmu,\bpsi) &= -A^* \cG^*(\bmu,\bpsi) \\
&= -A^* \bcG^*(\bmu,\bpsi) -A^* \varepsilon_{\cG^*}(\bmu,\bpsi).
\end{align*}
The first term can be computed exactly, as was the case for previous~$Y$ bounds, 
and~\eqref{eq:estAepsG} provides us with computable upper bounds for the second term.

\paragraph{The $Z_1$ bounds.} We also start by splitting
\begin{align*}
DT(\bmu,\bpsi) = I - A^* D\bcG^*(\bmu,\bpsi) - A^* D\varepsilon_{\cG^*}(\bmu,\bpsi).
\end{align*}
Note that~$\varepsilon_{\cG^*}$ is linear, hence the last term is simply equal 
to~$- A^* \varepsilon_{\cG^*}$, and we can again use~\eqref{eq:estAepsG} to estimate 
it. Regarding the $I - A^* D\bcG^*(\bmu,\bpsi)$ term, we proceed as in 
Remark~\ref{rem:opellinftyVSopell1}, but replacing~$(\lambdabif,\ubif)$ 
by~$(\blambda,\bu)$. That is, we consider
$\bB\in\cL(\R\times\ell^1_{s},\R\times\ell^1_{s-4})$ given by
\begin{equation*}
\bB = \left(
\begin{array}{
  m{0.5cm} | m{3.5cm} 
}
\parbox[c][0.5cm][c]{0.5cm}{\centering $0$} &
\parbox[c][0.5cm][c]{3.5cm}{\centering $-\bpsi^*$} \\ \hline

\parbox[c][3.5cm][c]{0.5cm}{\centering $\bphi$} &
\parbox[c][3.5cm][c]{3.5cm}{\centering $D_uF(\blambda,\bu)-\bmu$} 
\end{array}
\right),
\end{equation*} 
for which we have
\begin{align*}
I - A^* D\bcG^*(\bmu,\bpsi) = \left(I - \bB A \right)^*.
\end{align*}
We can then estimate $I-\bB A:\R\times\ell^1_{s-4} \to \R\times\ell^1_{s-4}$ in
a very similar manner to what was done for previous~$Z_1$ estimates, the only 
difference being that~$A$ is now on the right (instead of being on the left
in~$I-AD\cF(\bx)$). Thus, it is sufficient to use the projections~$\Pi^{\leq N}$ 
and~$\Pi^{>N}$ to split between the finite part and the tail part of~$Z_1$. 
The finite part is obtained exactly as before, and there is only one non-zero 
tail term, namely 
\begin{align*}
\left(I + (D_uF(\blambda,\bu)-\bmu)\Delta^{-2}\right)\Pi^{>N}   &=  -\blambda\Delta\left((1-3\bu^2)\Delta^{-2}\Pi^{>N} \cdot \right) - (\blambda\sigma + \bmu)\Delta^{-2}\Pi^{>N}.
\end{align*}
Using Lemma~\ref{lem:est_ells}, and the fact that~$\ell^1_{s-2}$ is a Banach 
algebra (it is exactly for this step that we assumed $s\geq 2$), one obtains
\begin{align*}
\left\Vert \Delta\left((1-3\bu^2)\Delta^{-2}\Pi^{>N} \cdot \right) \right\Vert_{B(\ell^1_{s-4},\ell^1_{s-4})} &\leq \max_i\left(\frac{\pi}{L_i}\right)^2 \left\Vert 1-3\bu^2\right\Vert_{\ell^1_{s-2}} \left\Vert \Delta^{-2}\Pi^{>N} \right\Vert_{B(\ell^1_{s-4},\ell^1_{s-2})} \\
&\leq \left(\frac{\max_i L_i}{\min_i L_i}\right)^2 \left(\frac{1+m_N}{m_N} \right)^2 \left\Vert 1-3\bu^2\right\Vert_{\ell^1_{s-2}} \left\Vert \Delta^{-1}\Pi^{>N} \right\Vert_{B(\ell^1_{s-4},\ell^1_{s-4})} \\
&\leq  \left(\frac{\max_i L_i}{\min_i L_i}\right)^2 \left(\frac{1+m_N}{m_N} \right)^2 \frac{\left\Vert 1-3\bu^2\right\Vert_{\ell^1_{s-2}}}{\varrho_N},
\end{align*}
with~$\varrho_N$ introduced in Section~\ref{sec:boundsOK}. Thus, we get
\begin{align*}
\left\Vert \left(I + (D_uF(\blambda,\bu)-\bmu)\Delta^{-2}\right)\Pi^{>N}  \right\Vert_{B(\ell^1_{s-4},\ell^1_{s-4})} \leq Z^{1,\tail}_{2,2},
\end{align*}
with
\begin{align*}
Z^{1,\tail}_{2,2} = \vert\blambda\vert  \left(\frac{\max_i L_i}{\min_i L_i}\right)^2 \left(\frac{1+m_N}{m_N} \right)^2 \frac{\left\Vert 1-3\bu^2\right\Vert_{\ell^1_{s-2}}}{\varrho_N} + \frac{\vert\blambda\sigma + \bmu\vert}{\varrho_N^2}. 
\end{align*}

\paragraph{The $Z_2$ bounds.} This is trivial, as the only nonlinear term 
in~$\cG^*$ is~$-\mu \psi$, which leads to
\begin{align*}
\left\Vert\pi_2 D_{12} \cG^*(\mu,\psi) \right\Vert_{BL(\cX^1\times\cX^2,\cY^2)}  = \left\Vert\pi_2 D_{21} \cG^*(\mu,\psi) \right\Vert_{BL(\cX^2\times\cX^1,\cY^2)}  = 1,
\end{align*}
and to all the other terms being zero.

\addcontentsline{toc}{section}{References}
\footnotesize
%
%
\bibliography{wanner1a,wanner1b,wanner2a,wanner2b,wanner2c,bibfile_breden}
\bibliographystyle{abbrv}
\end{document}